\documentclass[a4paper]{amsart} 
\usepackage[T1]{fontenc}
\usepackage[utf8]{inputenc}
\usepackage{amsfonts,amssymb}
\usepackage{mathrsfs}
\usepackage{mathtools}			% for norms and absolute values
	\mathtoolsset{centercolon}	% better looking :=
\usepackage{xcolor}			% coloured text
\usepackage{todonotes}
\usepackage{comment}
\usepackage{graphicx}
\usepackage{import}			% for pictures
\usepackage{calc}				% for pictures
\usepackage[colorlinks]{hyperref}
\usepackage{bookmark}
\usepackage[capitalise,noabbrev]{cleveref}

\theoremstyle{plain}
\newtheorem{thm}{Theorem}[section]
\newtheorem{cor}[thm]{Corollary}

\newtheorem{lemma}[thm]{Lemma}
\newtheorem{prop}[thm]{Proposition}

\theoremstyle{definition}
\newtheorem{defi}[thm]{Definition}

\theoremstyle{remark}
\newtheorem{rem}[thm]{Remark}

\numberwithin{equation}{section}
\newcommand{\numberset}{\mathbb}
\newcommand{\BB}{\mathbb{B}}
\newcommand{\CC}{\numberset{C}}
\newcommand{\DD}{\numberset{D}}
\newcommand{\FF}{\mathbb{F}}
\newcommand{\HH}{\numberset{H}}
\newcommand{\KK}{\mathbb{K}}
\newcommand{\NN}{\numberset{N}}

\newcommand{\RR}{\numberset{R}}

\newcommand{\ZZ}{\numberset{Z}}
\newcommand{\RRbar}{\overline{\RR}}

\DeclareMathOperator{\SL}{SL}

\DeclareMathOperator{\PSL}{PSL}

\DeclareMathOperator{\height}{ht}
\DeclareMathOperator{\Id}{Id}
\DeclareMathOperator{\Realpart}{Re}
\renewcommand{\Re}{\Realpart}
\DeclareMathOperator{\Imaginarypart}{Im}
\renewcommand{\Im}{\Imaginarypart}
\DeclareMathOperator{\diameter}{diam}
\DeclareMathOperator{\lipschitz}{Lip}
\DeclareMathOperator{\Aut}{Aut}

\renewcommand{\epsilon}{\varepsilon}
\renewcommand{\phi}{\varphi}

\newcommand{\cA}{\mathscr{A}}

\newcommand{\cD}{\mathcal{D}}
\newcommand{\cF}{\mathcal{F}}
\newcommand{\cG}{\mathcal{G}}

\newcommand{\cK}{\mathcal{K}}
\newcommand{\cL}{\mathcal{L}}
\newcommand{\cM}{\mathcal{M}}
\newcommand{\cT}{\mathcal{T}}
\newcommand{\cU}{\mathcal{U}}
\newcommand{\cV}{\mathcal{V}}
\newcommand{\cW}{\mathcal{W}}

\newcommand{\Arc}{\mathcal{A}}

\renewcommand{\emptyset}{\varnothing}
\newcommand{\ab}[1]{{a_{#1}}} 			% boundary expansion entries
\newcommand{\bb}[1]{{b_{#1}}}
\newcommand{\cc}[1]{{w_{#1}}}
\newcommand{\arclength}[1]{|#1|_{\partial\mathbb{D}}}
\newcommand{\wordlength}[1]{\ell(#1)}

\DeclarePairedDelimiter{\abs}{\lvert}{\rvert}		% \norm*{} sets adjustable height
\begin{document}
%%%%%%%%%%%%%%%%%%%%%%%%%%%%%%%%%%%%%%%%%%
%%%										%%%
%%%		TITLE, AUTHORS and ADDRESSES	%%%
%%%										%%%
%%%%%%%%%%%%%%%%%%%%%%%%%%%%%%%%%%%%%%%%%%

\title[Hall rays for Lagrange spectra of Riemann surfaces]{Persistent Hall rays for Lagrange spectra at cusps of Riemann surfaces}

\author[M. Artigiani]{Mauro Artigiani}
\address{Centro De Giorgi \\ Collegio Puteano \\ Scuola Normale Superiore \\ Piazza dei Cavalieri, 3 \\ I-56100 Pisa \\ Italy }
\email{mauro.artigiani@sns.it}

\author[L. Marchese]{Luca Marchese}
\address{Universit\'e Paris 13, Sorbonne Paris Cit\'e,
LAGA, UMR 7539, 99 Avenue Jean-Baptiste Cl\'ement, 93430 Villetaneuse, France.}
\email{marchese@math.univ-paris13.fr}

\author[C. Ulcigrai]{Corinna Ulcigrai}
\address{School of Mathematics \\ University of Bristol \\ University Walk \\ Bristol \\ BS8~1TW \\ United Kingdom}
\email{corinna.ulcigrai@bristol.ac.uk}

	%11J06 Markov and Lagrange spectra and generalizations
	%37D40 Dynamical systems of geometric origin and hyperbolicity (geodesic and horocycle flows, etc.)
	%30F35 Fuchsian groups and automorphic functions


\date{\today}

%%%%%%%%%%%%%%%%%%%%%%%%%%%%%%%%%%%%%%
%%%									%%%
%%%				DOCUMENT			%%%
%%%									%%%
%%%%%%%%%%%%%%%%%%%%%%%%%%%%%%%%%%%%%%

\begin{abstract}
We study Lagrange spectra at cusps  of finite area Riemann surfaces. % Fuchsian groups. 
These spectra %, together with the closely related Markov spectra, 
are penetration spectra that describe the asymptotic depths of penetration of geodesics in the cusps.
Their study is in particular motivated by Diophantine approximation on Fuchsian groups.
In the classical case of the modular surface and classical Diophantine approximation, Hall proved in 1947 that the classical Lagrange spectrum contains a half-line, known as a Hall ray.
We generalize this result to the context of Riemann surfaces with cusps and Diophantine approximation on Fuchsian groups. %, answering a question which was left open despite many results in the literature for various geometric generalizations of these spectra. %, the existence of Hall rays  for Riemann surface with cusps was only solved in the case of Markov spectra. 
One can measure excursion into a cusp both with respect to a natural height function or, more generally, with respect to any proper function. 
We prove the existence of a Hall ray for the Lagrange spectrum of any non co-compact, finite covolume Fuchsian group with respect to any given cusp, both when the penetration is measured by a height function induced by the imaginary part as well as by any proper function close to it with respect to the Lipschitz norm.
This shows that Hall rays are stable under (Lipschitz) perturbations. 
%Our main result is the   We show that in both cases the corresponding Lagrange spectrum contains a Hall ray. 
As a main tool, we use the boundary expansion developed by Bowen and Series to code geodesics and produce a geometric continued fraction-like expansion and mimic the key ideas in Hall's original argument. A key element in the proof of the results for proper functions is a generalization of Hall's theorem on the sum of Cantor sets, where we consider functions which are small  perturbations in the Lipschitz norm of the sum.  
\end{abstract}

\maketitle

%\tableofcontents
\section{Introduction}
The classical Lagrange spectrum is a well studied subset of the real line, which can be described either in terms of Diophantine approximation or dynamics, as penetration spectrum of geodesics on the modular surface (see \cref{sec:classical}).  Hall proved in 1947 that the classical Lagrange spectrum contains a semi-infinite interval, known as a Hall ray. We generalize this result to the context of Riemann surfaces with cusps and Diophantine approximation on Fuchsian groups, answering a question which was left open despite many results in the literature for various geometric generalizations of these spectra (see \cref{sec:History}). Definitions of the Lagrange spectra we study and main results are presented in \cref{sec:Lagrange_defs} and \cref{sec:mainresult_Fuchsian}.
%are given in \cref{sec:Lagrangefuchsian}, while our main results are stated in  and \cref{sec:mainresult_stable}. 

%The main  results in this paper concern the presence  of a \emph{Hall ray}, namely a semi-infinite interval of the form $[r,+\infty)$, in Lagrange spectra of Riemann surfaces with cusps and Diophantine approximation in Fuchsian groups.

\subsection{The classical Lagrange spectrum}\label{sec:classical}
Classical Diophantine approximation is the study of how well one can approximate a real number $\alpha$ by rational ones.
The well-known results of Dirichlet and Hurwitz imply that for all irrational real numbers $\alpha$ there are infinitely many $p/q$, $p\in \ZZ, q\in \NN$, such that
\[
	\abs*{\alpha - \frac{p}{q}}< \frac{1}{\sqrt{5}q^2}
\]
and that this is the best possible result for \emph{every} real number $\alpha$, as one can see by considering the golden mean $\alpha=\frac{1+\sqrt{5}}{2}$.
A natural question is hence if \emph{fixing} $\alpha$ one can improve the constant appearing in the denominator.
This leads to the introduction of the (classical) \emph{Lagrange spectrum} $\cL \subset \RRbar:=\RR \cup \{+\infty\}$ as follows.
For a given $\alpha\in\RR$, let  $L(\alpha)\in \RRbar $ be such that 
\[
L(\alpha):= \sup \{ k : |\alpha - p/q| < 1/k q^2 \ \text{for infinitely many } q\in \mathbb{N},\  p \in \mathbb{Z} \} .
\]
% for any $c>L$, we have $|\alpha - p/q| > 1/cq^2$ for all $p$ and $q$ large enough and $L(\alpha)$ is minimal with respect to this property (for example, $L(\frac{1+\sqrt{5}}{2}) = \sqrt{5}$).  
Then  $\cL$  is the collection of values $\{ L(\alpha), \alpha \in \RR\}$. %if and only if there exists $\alpha\in\RR$ such that $L(\alpha)=L$.  
Equivalently, one can also write (see for example~\cite{MatheusCIRM})
\begin{equation}\label{eq:classicdefLagrange}
	\cL = \left\{ L(\alpha)= \limsup_{q,p\to +\infty} \frac{1}{q|q\alpha -p|},\quad \alpha \in\RR \right\} \subset \RRbar=\RR \cup \{+\infty\}.
\end{equation}
One can see that for almost every $\alpha$ one has $L(\alpha)=\infty$, but $L(\alpha)<\infty$ for a set of full Hausdorff dimension, which consists exactly of so called \emph{badly approximable} (or bounded type) numbers.  

A close relative of the Lagrange spectrum is the \emph{Markoff spectrum} $\cM$ obtained by replacing the $\limsup$ in~\eqref{eq:classicdefLagrange}  with a $\sup$.
Both Markoff and Lagrange spectra have been intensively studied by many authors, and much of their beautiful and rich structures is known.
Both $\cL$ and $\cM$ are closed subsets of $\RRbar$, with the strict inclusion $\cL\subset\cM$.
It is for instance known that:
\begin{itemize}
\item The minimum of $\cL$ and $\cM$ is $\sqrt{5}$, which is known as \emph{Hurwitz constant} \cite{Hurwitz,Markoff};
\item $\cL\cap (0,3)=\cM\cap(0,3)$ is an explicit discrete set that accumulates to $3$ (see for example \cite{CusickFlahive});
\item $\cL$ contains a semi-infinite interval $[R,\infty)$~\cite{Hall}. This part of the Lagrange spectrum is called the \emph{Hall ray}. The exact value of $R$ where the Hall ray begins (i.e.\ the smallest $r$ such that $[r,\infty)\subset \cL$) is known after the work of Freiman~\cite{Freiman} and hence known as \emph{Freiman constant}. 
\end{itemize}

The (classical) Lagrange spectrum admits also a geometro-dynamical interpretation as the spectrum of asymptotic depths of penetration for the geodesic flow into the cusp of the modular surface $X=\PSL(2,\ZZ)\backslash\HH$.
More precisely, the Lagrange spectrum is the set of values $L \in \RRbar$ which can be realized as 
\begin{equation}\label{classical_limsup}
L=	\limsup_{t\to +\infty}{\operatorname{height}(\gamma(t))},
\end{equation}
for a parametrized geodesic $\gamma(t)$ on $X=\PSL(2,\ZZ)\backslash\HH$, where $\operatorname{height}(x)$ denotes the hyperbolic \emph{height} in the cusp, given by the imaginary part of the unique lift of $x$ from the modular surface to its classical fundamental domain $\cF=\{z\in\HH : |z|>1, |\Re (z)|<\frac{1}{2}\}$.

\begin{rem}\label{how_to_go_to_infinity}
If one replaces the $\limsup$ for $t \to +\infty$ in~\eqref{classical_limsup}  with a $\limsup$ for $|t| \to +\infty$, the value on a given $\gamma$ might change (since it then also depends on the backward endpoint of $\gamma$), but one can show, choosing symmetric geodesics that the set of such values taken on all geodesics still produces exactly $\cL$ (see the Appendix of \cite{CusickFlahive} for a proof). 
 \end{rem}
%\todo[inline]{Mauro, puoi aggiungere la ref a Cusick e Flahivre? possibile che non la avessimo gia'? da qualche parte si potrebbe dire che e' una buona ref classica...}

%This shows that the Lagrange spectrum is a special case of \emph{penetration spectrum}: more in general, given a (non-compact) Riemannian manifold $M$ with negative sectional curvature (or a more general space where geodesics are defined), if $f: M \to \mathbb{R}^+$ is a proper function, the penetration spectrum with respect to $f$ is given by L=	\limsup_{t\to \infty}{\operatorname{height}(\gamma(t))},

Over the course of time, both the Markoff and the Lagrange spectrum have been generalized to many different contexts, either from the Diophantine approximation or from the geometric point of view, exploiting their dynamical definition as penetration spectra.  

\subsection{A brief history of  generalizations of Markoff and Lagrange spectra.}\label{sec:History}
We do not attempt here to summarize all the developments in this area, which started more than a century ago and has seen a surge of recent developments, but we will only briefly survey some of the results,  in particular those which are closer to the main topic of this paper, namely the presence of Hall rays. 
The interested reader can find further information in the monograph~\cite{CusickFlahive} by Cusick and Flahive, the introduction of~\cite{HMU} and the recent survey by Matheus~\cite{MatheusCIRM}, and refer to the references therein. 

The first  natural generalization  of the classical Lagrange and Markoff spectra is obtained by replacing the modular group $\PSL(2,\ZZ)$ with a more general Fuchsian group. Both the dynamical and Diophantine approximation definition extend naturally to this context (see \cref{sec:Lagrange_defs}).  
In particular, important classes of examples are given by the cases of Hecke groups and more generally triangle groups.
The minimum value in these spectra, which is also called, as in the classical case, the \emph{Hurwitz constant}, is computed for Hecke and triangle groups respectively by Haas and Series in~\cite{HaasSeries} and Vulakh in~\cite{Vulakh97a}.  
 Markoff spectra of Fuchsian groups were studied in detail by Vulakh in~\cite{Vulakh97b, Vulakh2000}; 
in particular, in~\cite{Vulakh97b} the author gives the complete description of the discrete part of the Markoff spectrum (and hence of the Lagrange spectrum) of any Hecke group.

Another natural generalization leads to study Markoff and Lagrange spectra for quotient of higher dimensional hyperbolic spaces by discrete subgroups~\cite{Vulakh95a}.
In particular, the case of Bianchi groups has connections with the approximation of a complex number with numbers from a given imaginary quadratic number field, see~\cite{Vulakh95b, Maucourant}.

Penetration spectra and more general objects, such as spiraling spectra, can be studied more generally in the context of (variable) negative curvature, see in particular the works by Paulin in collaboration with Hersonsky~\cite{HersonskyPaulin} and Parkkonen~\cite{ParkkonenPaulin}.  
In~\cite{ParkkonenPaulin} it is shown that both the Lagrange and Markoff spectrum of a finite volume Riemannian manifold with sectional curvature less than $-1$ and \emph{dimension at least} $3$ contain a Hall ray.
Remarkably, Parkkonen and Pauline also managed to obtain a \emph{universal} estimate on the beginning of both spectra.

In the case of surfaces, Schmidt and Sheingorn proved in~\cite{SchmidtSheingorn} that the Markoff spectrum of a hyperbolic surface of constant negative curvature $-1$ contains a Hall ray.  %We stress that such result has no implication on the existence of a Hall ray for the corresponding Lagrange spectrum (see the comments after Theorem \cref{thm:Hall}). 
%We stress that such result does not imply that there is a Hall ray for the corresponding Lagrange spectrum.
%In general, indeed, it is much easier to construct values in the Markoff spectrum than in the Lagrange spectrum, essentially because of the presence of a supremum instead than a limsup in \cref{eq:classicdefLagrange}: in order to show that a certain value is in the Markov spectrum it is enough to construct a geodesic along which the supremum is achieved, while for the Lagrange spectrum one has to construct a sequence of times which tends to the given value.
Recently Moreira and Roma\~na proved that, for generic small perturbations of dynamically defined analogues of the Lagrange and Markoff spectra on negatively curved surfaces, these spectra contain intervals arbitrarily close to infinity. We stress that  neither of these two results  does however imply  the existence of Hall rays for Lagrange spectra. 
In a similar spirit, in~\cite{MoreiraHausdorff} continuity of the Hausdorff dimension of the Spectra, when intersected with the open interval $(-\infty,t)$ for $t\in\RR$, is proved.
Very recently, in~\cite{CMMHausdorff}, the same result is proved for generic perturbations of dynamically defined spectra on negatively curved surfaces.
An introduction of Moreira's work as well as the classical theory of these spectra can be found in~\cite{MatheusCIRM}.  

Another generalization is introduced by Hubert and two of the authors in~\cite{HMU}, where  Lagrange spectra are defined in the context of translation surfaces and interval exchange transformations.
Also in this case one has a geometric definition as penetration spectra for the Teichm\"uller geodesic flow, as well as an interpretation motivated by Diophantine approximation for interval exchange maps, see \cite{HMU}.
A version of the latter already appears in the work by Boshernitzan, see~\cite{Boshernitzan}; different types of Lagrange spectra for interval exchange transformations, in particular in the case of 3 interval exchanges, are also studied by Ferenczi in~\cite{Ferenczi}.
For Lagrange spectra of \emph{strata} of translation surfaces, the existence of Hall rays was in~\cite{HMU} and the Hurwitz constant was recently found by Boshernitzan and Delecroix~\cite{BD}.
The authors proved in~\cite{AMU} that also for the particularly symmetric class of translation surfaces made of Veech surfaces, the Lagrange spectrum contained a Hall ray and the first values of the Lagrange spectrum a particular example of a square-tiled Veech surface are studied in detail in~\cite{HLMU}.

%The main  results in this paper concern the presence  of a \emph{Hall ray}, namely a semi-infinite interval of the form $[r,+\infty)$, in Lagrange spectra of Riemann surfaces with cusps and Diophantine approximation in Fuchsian groups. 

 The generalizations of Lagrange spectra we study in this paper are in the context of Diophantine approximation on Fuchsian groups (see \cref{sec:Lagrange_defs}) and penetration spectra for Riemann surfaces with cusps with respect to proper functions (see \cref{sec:mainresult_Fuchsian}). 
%(due to Lehner~\cite{Lehner} and studied among others by Haas, Series, Vulak) and penetration spectra for Riemann surfaces with cusps (following the definition by Paulin-Parkonnen \cite{}, see the next section). %In both set ups, we prove the existence of Hall rays. 
%In analogy with the classical Lagrange spectrum, we now define these Lagrange spectra first from the Diophantine approximation point of view, then interpret them  geometrically in terms of essential heights of geodesics and finally dynamically as limsup of a proper function along the geodesic flow. 
	
\subsection{Lagrange spectra and Diophantine approximation in Fuchsian groups}\label{sec:Lagrange_defs}
The definition of  Lagrange spectrum for a Fuchsian group $G$ in terms of Diophantine approximation on $G$ is due to  to Lehner~\cite{Lehner}, inspired by Ford's geometric proof of Hurwitz theorem~\cite{Ford} and was studied among others by Haas, Series, Vulakh~\cite{Haas,HaasSeries,Vulakh97a,Vulakh97b,Vulakh2000}.
%we study in this paper is the one to the context of Diophantine approximation on Fuchsian groups (due and penetration spectra for Riemann surfaces with cusps (following the definition by Paulin-Parkonnen \cite{}, see the next section). 
In analogy with the classical Lagrange spectrum, we now define these Lagrange spectra first from the Diophantine approximation point of view, then interpret them  geometrically in terms of essential heights of geodesics and, finally, in a more dynamical way. 
%After having stated our first result in \cref{subsec:thmHall} we give a more general dynamical definition of the Spectrum as limsup of a proper function along the geodesic flow (\cref{sec:mainresult_Fuchsian}).

We denote with $\HH=\{z=x+iy\in\CC : y>0\}$ the upper half-plane with the hyperbolic metric.
The group of isometries of $\HH$ can be identified with $\PSL(2,\RR)$ (see the beginning of \cref{sec:background}).
Discrete subgroups of $\PSL(2,\RR)$ are called \emph{Fuchsian groups}.
Since Fuchsian groups act by isometries on $\HH$ the quotient of the hyperbolic plane by any such group inherits a natural metric from the hyperbolic metric on $\HH$. Thus, $X:=G\backslash\HH$  is a hyperbolic surface (possibly with orbifold singularities coming from fixed points of elliptic elements in $G$). 
A Fuchsian group $G$ is a \emph{lattice} if the quotient $X=G\backslash\HH$ has finite volume, with respect to the natural volume form induced by the metric.
We consider only Fuchsian groups that are so-called \emph{non uniform} lattices, meaning that the quotient has finite volume but is not compact. %In this case the surface $G\backslash\HH$ has in this case finitely many punctures, which are either cusps or possibly orbifold points (coming from fixed points of elliptic elements in $G$). 

The action of $\PSL(2,\RR)$ %$\begin{pmatrix} a& b\\ c& d \end{pmatrix} \in \PSL(2,\RR)$ on $\HH$ by M\"obius transformations $z \mapsto \frac{az+b}{cz+d}$ 
extends by continuity to an action on the boundary $\RRbar=\RR \cup\{\infty\}$ of $\HH$.
We recall that an element of $\PSL(2,\RR)$ is parabolic if it has trace equal to $2$.
The set of \emph{cusps} of $G$ is the set of points of $\RRbar$ fixed by a non trivial parabolic element of $G$. If $X=G\backslash\HH$ has finite volume, then the set of cusps of $X$ coincides with the set of \emph{ends} of the surface itself. 
We are going to assume for now that $\infty$ is a parabolic fixed point for $G$ (we later remove this assumption, see \cref{rem:conjugacy_rays} and \cref{othercusps}). 
We call an element of $\RRbar$ \emph{$G$-rational} if it is the fixed point under some non trivial parabolic transformation in $G$.
The complement of $G$-rational numbers in $\RRbar$ is the set of \emph{$G$-irrational numbers}. 

\smallskip
%\subsubsection*{Diophantine approximation in Fuchsian groups}{subsec:Diophantine}
Diophantine approximation on a Fuchsian group $G$ consists in approximating $G$-irrational numbers by $G$-rational ones, or $G$-rational ones in the $G$-orbit of a fixed cusp. The definition of  Lagrange  spectrum $\cL(G,\infty)$  in terms of Diophantine approximation on $G$ introduced by Lehner~\cite{Lehner} is the following.   Given $g\in G$, we denote by $a(g)$ and by $c(g)$ the first entry on the first and second row of $g$ respectively.
For $\alpha\in\RR$ define $L_G (\alpha)$ to be:% the supremum of the $k>0$ such that
\begin{multline*}
L_G(\alpha):=  \sup \{ k : \left|\alpha-g\cdot\infty\right|=  \left|\alpha-\frac{a(g)}{c(g)} \right| <\frac{1}{kc(g)^2}  \text{ for infinitely many } g\in G \\ \text{ s.\ t.\ } g \cdot \infty \text{ are all distinct} \}.
\end{multline*}
Then, we define $\cL(G,\infty):=\{ L_G(\alpha), \alpha \in \RR\}$. 
We remark that if we take $G=\PSL(2,\ZZ)$ in the previous definitions, $L_G(\alpha)$ coincides with the one given in~\eqref{eq:classicdefLagrange} for $L(\alpha)$, so  that $\cL(G,\infty)$ is indeed a generalization of the classical Lagrange spectrum $\cL$. 
%The condition that $g \cdot \infty$ are all distinct.  

\smallskip
%\subsubsection*{Lagrange spectra via essential heights of geodesics}{subsec:essentialheight}
We can interpret this definition geometrically as follows. We recall that a geodesic in the hyperbolic plane is uniquely determined by its two extremal points in $\RRbar$.
Given two points $x$ and $y$ in $\RRbar$, %denoting by $\gamma(x,y)$ the hyperbolic geodesic connecting $x$ and $y$, we 
throughout the paper  we  denote by $\gamma(x,y)$ the hyperbolic geodesic connecting $x$ to $y$ which has $x$ (resp.\ $y$) as a backward (resp.\ forward) end point, i.e.\ if $\gamma(t)$ is the geodesic parametrization of $\gamma$, 
\[ 
\gamma(-\infty):=\lim_{t \to -\infty} \gamma(t)=x, \qquad  \gamma(+\infty):=\lim_{t \to +\infty} \gamma(t)=y.
\]
We define the \emph{naive height} of the geodesic $\gamma=\gamma(x,y)$ by
\begin{equation}\label{eq:height}
	\height(\gamma) = \begin{cases}
							\frac{1}{2} \abs{x-y}, & \text{if $x,y\in\RR$,}\\
							\infty,	& \text{otherwise.}
				\end{cases}
\end{equation}
Thus, the naive height $\height(\gamma)$ is the Euclidean radius of the semi-circle which represents the geodesic $\gamma$ in the upper half plane $\mathbb{H}$. 
%Given any parametrization $t\mapsto \gamma(t)$ of $\gamma$ one immediately sees that $\height(\gamma)=\sup_{t\in\RR} \Im{\gamma(t)}$.

We say that two elements $g$ and $h$ in $G$ are \emph{equivalent modulo infinity} if there is an element $k \in G$ that fixes infinity and such that $g=k h$. We remark that if $g$ and $h$ are equivalent modulo infinity they differ by a horizontal translation and hence, for every geodesic $\gamma$ in $\HH$, $\height(g(\gamma))=\height(h(\gamma))$. 
Choose a set $G_\infty$ of representatives of the equivalence classes of $G$ modulo infinity. 
The \emph{essential height} of a geodesic $\gamma$ on $X$ is defined by
\begin{equation}\label{eq:essentialheight}
	\height_{G}(\gamma)=\sup{\{k : \height(g(\tilde\gamma)) \geq k \quad \text{for\ infinitely\  many $g\in G_\infty$}\}},
\end{equation}
where $\tilde{\gamma}$ is any lift of the geodesic $\gamma$ from $X$ to the universal cover $\HH$. 

The following Lemma provides a geometric interpretation of the constant $L_G(\alpha)$ in terms of essential height.
We include below also its short proof, which can be found e.g.\ in~\cite{HaasSeries}, since it provides an educational example, for the non familiar reader, of the interplay between Diophantine approximation and penetration in the cusps. 
\begin{lemma}[\cite{HaasSeries}]\label{thm:Lagrangeasheight}
Let $G$ be a non uniform lattice in $\PSL(2,\RR)$.
For every real number $\alpha$ we have
\[
	L_G(\alpha)=2\height_G (\gamma(\infty,\alpha)),
\]
where $\gamma(\infty,\alpha)$ is the vertical geodesic from $\infty$ to $\alpha$.
\end{lemma}
\begin{proof} Assume that $k>0$ is such that there exists a sequence $g_i$ of infinitely many elements of $G$ such that
\[
	|\alpha-g_i\cdot\infty|<\frac{1}{kc(g_i)^2},
\]
and the points $g_i\cdot\infty$ are all distinct. 
The vertical hyperbolic geodesic $\gamma(\infty,\alpha)$ intersects each of the Euclidean disks $D_i$ of radius $1/kc(g_i)^2$ tangent to $\RR$ at the points $g_i\cdot\infty$.
Equivalently $g_i^{-1}(\gamma(\infty,\alpha))\cap g_i^{-1}(D_i)\neq\emptyset$.
Since the $g_i\cdot\infty$ are all distinct, we have that the elements $g_i^{-1}$ are not equivalent modulo infinity.
A simple computation shows that $g_i^{-1}(D_i)=\{z\in\HH : \Im z\geq k/2\}$.
Thus, the assumption that $\gamma(\infty,\alpha)$ crosses $D_i$ implies that $k$ is such that $\height(\gamma)\geq k/2$ for infinitely many elements of $G_\infty$. Thus, $k$ belongs to the set of which $L_G(\alpha)$ is supremum if and only if $k/2$ belongs to the set of which $\height_{G_\infty}(\gamma)$ is the supremum. This gives the desired equality. 
%the respective sets for which $L_G(x)$ and  $\height_{G_\infty}(\gamma)$ suprema coincide and hence. 
% $1/k$ belongs to the set of which  is the supremum.  This shows that  is what we wanted to prove.
\end{proof}

One  can define the Lagrange spectrum 	$\cL(X,\infty)$ of the hyperbolic  surface $X= G\backslash\HH$ with respect to the cusp at $\infty$ to be
\[
	\cL(X,\infty) :=\{2\height_G (\gamma), \gamma \text{ geodesic on $X=G\backslash\HH$}\}.
\]
%Then \cref{thm:Lagrangeasheight} shows that $\cL(X,\infty)$ to be equal to $\cL(G,\infty)$, where $G$ is the uniformizing Fuchsian group. 
The reason for the constant $2$ appearing in the definition is 
 apparent from \cref{thm:Lagrangeasheight},  since  
%from \cref{thm:Lagrangeasheight},  one can now
one can use it to show that these definitions coincide if  $G$ is the uniformizing Fuchsian group of $X$.

\begin{cor}
If $X=G\backslash\HH$ then we have that $\cL(X,\infty)=\cL(G,\infty)$.
\end{cor}

\begin{proof}
\cref{thm:Lagrangeasheight} directly gives the inclusion $\cL(G,\infty)\subset \cL(X,\infty)$. 
Conversely, if $\gamma=\gamma(\alpha^-, \alpha^+)$, consider the two vertical geodesics $\gamma^-:= \gamma(\infty,\alpha^-)$ and $\gamma^+:= \gamma(\infty,\alpha^+)$.
Suppose, with loss of generality, that $L_G(\alpha^+)>L_G(\alpha^-)$.
Let us show that this implies that $\height_G(\gamma)=\height_G(\gamma^+)$. 
Let $\height_G(\gamma)=h$.
This implies, by~\eqref{eq:essentialheight}, that, for every $\epsilon>0$, there exists a sequence $g_i$ of infinitely many elements of $G$ such that $g_i(\gamma)\cap\cU_{h-\epsilon}\neq\emptyset$, where $\cU_l=\{z\in\HH : z=x+iy, y>r\}$.
Equivalently, $\gamma\cap g_i^{-1}(\cU_{h-\epsilon})\neq\emptyset$.
It is enough now to observe that this can only happen to a portion of the geodesic $\gamma$ bounded away from the past endpoint $\alpha^-$ for, otherwise, we would have that
\[
	|\alpha^- - g_i^{-1}\cdot\infty|<\frac{1}{2(h-\epsilon)c(g_i^{-1})^2},
\]
that is $L_G(\alpha^-)\geq 2h = L_G(\alpha^+)$, which is a contradiction.
\end{proof}

%\todo[inline]{Quote Ford... sentence from Series: The geometrical ideas we use are based on the observation that the Diophantine
%problems (1) and (2) are deeply related to the behaviour of geodesies on the hyperbolic
%Riemann surface SG = H/G. Geometrical methods were used to calculate h(G) in the
%classical case of SL(2, T)' by Ford [4].}
%\mdoubts{Apparentemente Haas e Series si sono inventati la referenza. L'articolo che citano non è presente sul numero che dicono loro, ma soprattutto Ford \emph{non} ha mai scritto un articolo con quel titolo, o almeno così dice Mathscinet. Comunque, ho scritto che Lehner si è ispirato a Ford (citando l'articolo giusto, spero) all'inizio di questa sezione.}

\smallskip
%\subsubsection*{Dynamical interpretation}{subsec:dynamicalheight}
Finally, one can also interpret $\cL(G,\infty)$ in a more \emph{dynamical} way, analogously to what happens in the classical case (see~\eqref{classical_limsup} and \cref{how_to_go_to_infinity}). Given any parametrization $t\mapsto \gamma(t)$ of $\gamma$ (in particular the one given by the geodesic flow) one immediately sees that $\height(\gamma)=\sup_{t\in\RR} \Im{\gamma(t)}$. If $\height_G (\gamma)$ is sufficiently large (greater than the starting point of the maximal Margulis neighborhood, see~\eqref{eq:fundamentalhorodisk}), one can see that  one equivalently has 
\begin{equation}\label{heightlimsup}
	\height_G (\gamma)=\limsup_{|t|\to \infty}{\operatorname{height}(\gamma(t))},
\end{equation}
where, as before,  $\operatorname{height}(x)$ is the imaginary part of the unique lift of a point from $x$ to a chosen fundamental domain of $X$ which has two vertical lines.
This equivalence can be seen as a byproduct of the proof of Perron's formula for the essential height, see \cref{lemma:Perron} for details.

Thus, the Lagrange spectrum $\cL(X, \infty)$ describes  asymptotic depths of penetration of the geodesics of $X$ into the cusp $e=\infty$. This is the point of view that  we will  generalize in \cref{sec:mainresult_Fuchsian}, where we consider more general ways of measuring the penetration into a cusp.

\subsection{Hall rays for Diophantine approximation in Fuchsian groups}\label{subsec:thmHall}
The first  result we prove in this paper is the  following generalization to Fuchsian groups of Hall's theorem on the existence of a Hall ray for the classical Lagrange spectrum $\cL$, proved 1947 for the classical spectrum.

\begin{thm}[Hall ray for Fuchsian groups]\label{thm:Hall}
Let $G\subset \PSL(2,\mathbb{R})$ be a non uniform lattice. Assume that $\infty$ is a cusp of $G$. The Lagrange spectrum $\cL(G,\infty)$ of $G$ with respect to $\infty$ contains a \emph{Hall ray}, i.e.\  there exists an $L_0=L_0(G,\infty)\in\RR$ such that
\[
	[L_0,+\infty]\subset\cL(G,\infty).
\]
\end{thm} 

The result extends also to other cusps of $G$ as follows. Let us first remark that the presence of Hall rays does not depend on the choice of normalizations for the width of the cusp at $\infty$.

%We can remove the assumption that $\infty$ is a cusp of $G$ as follows. 
%The following remark  shows that the is no loss of generality. 
\begin{rem}\label{rem:conjugacy_rays} If $G' = \bar{g}\, G \, \bar{g}^{-1}$ is obtained by conjugating $G$ by an element of $\bar{g}\in \PSL(2,\RR)$  which fixes infinity,  $\cL(G,\infty)$ and $\cL(G',\infty)$ are obtained by each other by a smooth change of coordinates. 
In particular, $\cL(G,\infty)$ contains a Hall ray if and only if $\cL(G',\infty)$ does.
More precisely, if $\bar{g}=\left(\begin{smallmatrix} \lambda & \nu \\ 0 & 1/\lambda \end{smallmatrix} \right)$, then one has that the entry $c(\bar{g}g \bar{g}^{-1})=c(g)/\lambda^2$ for every $g\in G$.
Thus, using the explicit form of $\bar{g}$, we have
\[
\begin{split}
L_{G'}(\bar{g}\cdot\alpha)
	&=\sup \left\{ k : \left|\bar{g}\cdot\alpha-\bar{g}g \bar{g}^{-1} \cdot\infty\right|<\frac{1}{kc(\bar{g}g \bar{g}^{-1})^2}  \text{ for infinitely many } g\in G \right\}\\
	&=\sup \left\{ k : \lambda^2\left|\alpha-g\cdot\infty\right|<\frac{\lambda^4}{kc(g)^2}  \text{ for infinitely many } g\in G \right\}\\
	&= \frac{1}{\lambda^2} L_G(\alpha).
\end{split}
\]
\end{rem}

The Lagrange spectrum $\cL(G,e)$ with respect to a different cusp $e$ can be obtained by conjugating by an appropriate element of $\PSL(2,\RR)$ sending $e$ to $\infty$, once a normalization has been chosen (see for example \cref{sec:mainresult_Fuchsian} or \cref{sec:setup} for a natural one).
By \cref{rem:conjugacy_rays} the presence of a Hall ray does not depend on the actual choice of the normalization. We  have the following immediate corollary of \cref{thm:Hall}.

%Let us hence define the Lagrange spectrum with respect to a different cusp $e$ can be obtained by conjugating by an appropriate element  %We remark that if $X$ is a (smooth) Riemann surface, the group $G$ has no elliptic elements (see \cite{})> Thus, the definition 

\begin{cor}\label{othercusps}
Let $G\subset \PSL(2,\mathbb{R})$ be a non uniform lattice. For any  $e$ be a cusp of $G$, the Lagrange spectrum $\cL(G,e)$ of $G$ with respect to $e$ contains a \emph{Hall ray}.
\end{cor}

This results should be compared with the existence of Hall rays proved by Schmidt and Sheingorn in~\cite{SchmidtSheingorn} for the Markoff spectrum in an analogous setup. % We stress that such result does not imply that there is a Hall ray for the corresponding Lagrange spectrum.
In general, it is easier to construct values in the Markoff spectrum than in the Lagrange spectrum, essentially because of the presence of a supremum instead than a $\limsup$ in \cref{eq:classicdefLagrange}. In order to show that a certain value is in the Markoff spectrum is achieved, Schmidt and Sheingorn construct a geodesic which starts achieving a (sufficiently high) desired value of the height function. Then, to guarantee that this value is indeed the supremum, they  use a  symbolic coding (which essentially  counts winding numbers in the cusp at $\infty$)  and \emph{slide} the endpoints of the geodesic to guarantee that further excursions in the cusp are of lower height. On the other hand, for the Lagrange spectrum,  one needs to construct a \emph{sequence} of increasing times for which the height tends to the desired  value. To achieve this much more delicate form of control of the geodesic behavior, we also use symbolic coding (in the form of the boundary expansions first described by Bowen and Series) but then need to adapt to the Fuchsian setting Hall's original ideas in particular by reducing the result to the study of a sums of Cantor sets on the boundary. See \cref{sec:outline} for more details on the strategy of proof. 

%The key idea \todo[inline]{Move this Sheingorn Schmidt comment back to history? Add some comments on the proof? namely }

%Thus, these Lagrange spectra also generalizes the interpretation of the Lagrange spectrum as penetration spectrum for geodesics on the modular surface to other hyperbolic surfaces with cusps.  

\subsection{Hall rays for dynamical Lagrange spectra of Riemann surfaces} \label{sec:mainresult_Fuchsian}
%\todo[inline]{Add the word \emph{dynamical}? un bel modo di chiamare questi spettri con lisup e' dynamical Lagrange spectra, following Matheus... forse vale la pena usarlo...}
The Lagrange spectra $\cL(X, \infty)$,  defined in terms of Diophantine approximation in Fuchsian groups, can be interpreted, as  we saw in \cref{sec:Lagrange_defs}, as  penetration spectra for geodesics at the cusp at $\infty$ with respect to the height function.  From this point of view, it is natural to consider simultaneous penetration in other cusps and, more  generally, different notions of \emph{penetration}. Simultaneous penetration in the cusps can be defined with respect to any \emph{proper} function from the surface to $\mathbb{R}^+$ (see below). The main results stated in this section (\cref{thm:simultaneous} and \cref{thm:Hallperturbations}) concern these more general Lagrange penetration spectra and shows that Hall rays defined with respect to height functions are \emph{stable}, i.e.\ persistent  under (Lipschitz) perturbation, in a sense which is made precise in \cref{sec:mainresult_stable}.

\smallskip
%\subsubsection*{Penetration spectra with respect to proper functions}
We consider in this section any Riemann surface $X$ with genus $g$ and $n$ punctures, such that $\chi(X):= 2-2g -n<0$.  We adopt in this paper the convention (used for example by Beardon \cite{Beardon}) to call Riemann surfaces also two dimensional hyperbolic orbifolds (the modular \emph{surface} is such  an example since it has two orbifold singularities). These, also called \emph{marked} or singular Riemann surfaces, are all finite quotients of smooth Riemann surfaces. The Uniformization theorem gives that  $X = G\backslash\HH $ is the quotient of $\HH$ under some Fuchsian group $G$ acting by the left action given by M\"obius transformations. Orbifold singularities of $X$ correspond to elliptic elements if $G$, so $X$ is a smooth Riemann surface iff $G$ contains no elliptic elements.
%A \emph{cusp} $e$ in $X$ is a point in $\RRbar$ that is the only fixed point of a non trivial parabolic element in $G$. We will assume as before that $G$ is a non uniform lattice, i.e. that $X$ has at least one cusp, but finite volume. In this case, the set of cusps of $X$ coincides with the set of \emph{ends} of the surface itself.

%   If $X$ has finite volume, then the set of cusps of $X$ coincides with the set of \emph{ends} of the surface itself. Let $e$ be a cusp of $X=G\backslash\HH$ and suppose, without loss of generality, that $\infty$ is mapped to $e$ under the projection map $\pi\colon\HH\to X$. 
%We define the Lagrange spectrum 	$\cL(X,\infty)$ of the Riemann surface with respect to the cusp $e=\infty$ to be
%\[
%	\cL(X,\infty) =\{2\height_G (\gamma), \gamma \text{ geodesic on $X=G\backslash\HH$}\}.
%\]
Let $h\colon\HH\to\RR_+$ be a $G$-invariant continuous function. Equivalently, $h$ induces a function on the quotient $X=G\backslash\HH$, which we will still denote by $h$. We assume that $h$ is \emph{proper}, meaning a function such that the preimage of a compact set is a compact sets. In particular, $h$ diverges in the cusps of $X$. 

One can define a generalization of the Lagrange spectrum by measuring the asymptotic excursion into the cusps with respect to the function $h\colon X=G\backslash\HH\to\RR_+$ as follows.  
Let $t\mapsto\gamma(t)$ be a geodesic on $X$ and let %the \emph{essential height with respect to the proper function $h$}
\begin{equation}\label{def:limsupdef}
	L_G(h,\gamma) = L(h,\gamma) :=\limsup_{t\to + \infty}{h(\gamma(t))}. 
\end{equation}
%\todo[inline]{Mauro, come avrai notato, ho cambiato notazione. Va controllato quante altre occorrenze ci sono. Vorrei proprio evitare di chiamarla essential height, filosoficamente mi sembra diversa e come abbiamo visto le definizioni non sono nemmeno equivalenti per valori bassi... si potrebbe anche fare $L_h(\gamma)$ o $L_G(h,\gamma)$ o simili. Pensavo di lasciare fuori $G$ perche' $h$ contiene info su $G$ tramite la sua invarianza... magari creiamo una macro e poi ci giochiamo.  }
%\mdoubts{Ok su tutto. Controllo il file, mettendo $L(h,\gamma)$.}
We will often drop the explicit dependence on $G$ since it is implit in the symmetries of the function $h$. Then we call $\cL(X,h)$ the corresponding Lagrange spectrum, given by the set of values $L(h,\gamma)$ for $\gamma$ geodesics on $X$. These type of Lagrange spectra are also called \emph{dynamical Lagrange spectra} in the literature. Dynamical spectra were in particular studied in the seminal works by  \cite{Maucourant,ParkkonenPaulin,HersonskyPaulinDiophantine} and have seen a recent surge of interest, see for example \cite{AMU,BD,CMMHausdorff,Ferenczi,HLMU,HMU,MatheusCIRM}.

If the surface $X$ has only one cusp at infinity, $\operatorname{height}(\cdot)$ is an example of a \emph{proper} function on $X$. Proper functions when there are more cusps can be build for example by measuring penetration in each cusp with respect to a height function in that cusp and either adding them up or taking the maximum of these functions.    

A natural example of proper function to consider is given  by Paulin and Parkkonen in \cite{PP}. For each cusp $e$, let $\beta_{X,e}$ be the Busemann function on $X$ for the end $e$, normalized to converge to $+\infty$ towards $e$ and to vanish on the boundary of the maximal open Margulis neighborhood (definitions can be found in \cite{PP}). We remark that $\beta_{X,\infty}= 2 \log \height_G$ (see~\cite{HersonskyPaulinDiophantine}), so in particular these two penetrations have the same Hall rays.  
The \emph{Busemann height function} $\beta_X$ is defined to be the maximum of the functions $\beta_{X,e}$ over the cusps $e$ of $X$. 

The Lagrange spectrum $\cL(X)$ of  the Riemann surface $X$ is then defined by Paulin and Parkkonen to be the dynamical spectrum $\cL(X,\beta_X)$ with respect to the Busemann height function $\beta_X$. Let us first highlight, for its intrinsic interest, a result which will follow as a special case of the more general \cref{thm:Hallperturbations} that we will state in \cref{sec:mainresult_stable}.% since it is of independent interest. 
%Sull'enunciato intermedio si può enunciare il risultato come teorema e dire che è un corollario dell'altro subito sotto ma che lo esplicitiamo perché di interesse a sé stante.

\begin{thm}[Hall ray for Riemann surfaces]\label{thm:simultaneous}
For any $X$ non compact, finite volume Riemann surface with $\chi(X)<0$, the Lagrange spectrum $\cL(X):=\cL(X,\beta_X)$  contains a \emph{Hall ray}.
%\[
%	[r,+\infty]\subset\cL(X,e).
%\]
\end{thm}
As evidenced by the brief history in the previous section, the existence of a Hall ray %, proved by  Hall in 1947 for the classical spectrum, 
 was known for dynamical Markoff and Lagrange spectra dimension greater than 3 \cite{ParkkonenPaulin} and  for Markoff spectra of Riemann surfaces~\cite{SchmidtSheingorn}, as well as for Lagrange spectra in other dynamical contexts, such as \cite{HMU,AMU}. Thus, our work deals with the only case that was surprisingly still open in the constant curvature case, namely Lagrange spectra in dimension $2$.  
It is worth to remark that the existence of Hall rays is actually not expected to hold in general in dimension $2$ with variable negative curvature, see~\cite{ParkkonenPaulin}. 

\subsection{Persistence of Hall rays for Lagrange spectra of Riemann surfaces}
%\subsection{Stable Hall ray for penetration into cusps of Riemann surfaces} 
\label{sec:mainresult_stable}
We state now the most general result we prove in this paper (of which \cref{thm:simultaneous} is a corollary), which shows in particular that the Hall ray for $\cL(X,\beta_X)$ is \emph{stable} under Lipschitz perturbations.  We consider proper  functions $h$ which behave in at least one  cusp as a Lipschitz perturbation of the essential height function in that cusp, in the following sense.  

%Recall that a horodisk at infinity is a set of the form
%\[ 
%	\cU_m = \{ z\in\HH : z=x+iy, y>m \},
%\]
%for some $m>0$, such that the image of $\cU_m$ on the surface is a Margulis neighborhood, i.e. is homeomorphic to a punctured disk.

%Let $\cU_m$  be the \emph{fundamental horodisk} at infinity,  given by
%\begin{equation}\label{eq:fundamentalhorodisk}
%	\cU_m = \{ z\in\HH : z=x+iy, y>m \},
%\end{equation}
%where $m$ is the \emph{minimal} $s>0$ with the property that $\cU_s$ is mapped to a punctured disk on the surface $X$.
%The projection of $\cU_m$ on $X$ is called the \emph{maximal Margulis neighborhood} of the cusp at $\infty$.

Recall that a horodisk at infinity is a set of the form $\cU_l = \{ z\in\HH : z=x+iy, y>l \}$ for some $l>0$, such that its image on the surface $X=G\backslash\HH$ is a \emph{Margulis neighborhood}, i.e.\ is homeomorphic to a punctured disk. 
The \emph{fundamental horodisk} at infinity is
\begin{equation}\label{eq:fundamentalhorodisk}
	\cU_m,  \quad  \text{where $m$ is the \emph{minimal} $l>0$ s.t.\ $\cU_l$ is a Margulis neighbourhood}.
\end{equation}
The projection of $\cU_m$ on $X$ is called the \emph{maximal Margulis neighborhood} of the cusp at $\infty$. We will call $m$ the \emph{height} of the maximal Margulis neighborhood.

Let $U$ be an open subset  in $\HH$ %or in the plane $\RR^2$
 and let $g\colon U\to\RR$ be a continuous function bounded on $U$.
Recall that the uniform norm of $g$ is 
\[
	\|g\|_\infty:=\sup_{z\in U}g(z).
\]
Recall also that when $g$ is a \emph{Lipschitz function}, the \emph{Lipschitz constant} of $g$ is
\begin{equation} \label{def:Lip_const}
	\lipschitz(g):=\sup_{z,w\in U}\frac{|g(z)-g(w)|}{|z-w|}.
\end{equation}
Finally, if $g\colon U\to\RR$ is a bounded Lipschitz function we define its \emph{Lipschitz norm} as 
\[
%\begin{equation} %\label{def:Lip_const}
	\|g\|_{\text{Lip}}:=\|g\|_\infty+\lipschitz(g).
\]

The following result shows that the presence of a Hall ray is open  under perturbations in the Lipschitz norm.  
\begin{thm}[Hall ray for perturbations]\label{thm:Hallperturbations}
Let $G\subset \PSL(2,\mathbb{R})$ be a non uniform lattice. Assume that  $\infty$ is a cusp of $G$. % and that $m=1$, i.e. the maximal Margulis neighbourhood of $\infinty$ is $\cU_1$. 
Let  $h\colon\HH\to\RR_+$ be a $G$-invariant continuous function such that the induced function on $X=G\backslash\HH$ is proper.
	
There exists a constant $\delta_G>0$ such that, if there exists an $l_0>0$ so that 
\begin{equation}\label{Lip_assumptionG}
	\left\|\bigl( h-\Im(\cdot)\bigr)|_{\cU_{l_0}}\right\|_{\text{Lip}}<\delta_G, 
\end{equation}
then the Lagrange spectrum $\cL(X,h)$ contains a \emph{Hall ray}.  
\end{thm}

\begin{rem}\label{m1}
More precisely, we show that if $G$ is normalized so that the height of the maximal Margulis neighborhood is equal to $1$, then $\cL(X,h)$ contains a \emph{Hall ray} as long as, for some $l_0>0$, we have
\begin{equation}\label{Lip_assumption}
	\left\|\bigl( h-\Im(\cdot)\bigr)|_{\cU_{l_0}}\right\|_{\text{Lip}}<\frac{1}{4\sqrt{2}}.
\end{equation}
%More generally, from our proof one can see that the constant $\delta_G$ depends on $G$ only through the width $\mu$ of the cusp at $\infty$.  
\end{rem}

\begin{rem} As in the case of \cref{thm:Hall} (see \cref{othercusps})), the assumption that $\infty$ is a cusp is not a real  restriction. Given a Riemann surface $X$ and a proper function $\overline{h}\colon X \to \mathbb{R}_+$, the spectrum $\cL(X,\overline{h})$ contains a Hall ray as as long as there exists a cusp  $e$ of $X$ and a uniformization $X=G\backslash\HH$ such that $e$ lifts to $\infty$ and $m=1$, and a lift $h\colon \HH \to \mathbb{R}^+$ of $\overline{h}$ for which~\eqref{Lip_assumption} is satisfied. 
\end{rem}

\subsection{Some ideas in the proofs}\label{sec:outline}

Let us now give some details on the way in which we prove the main results that we stated in the two previous sections.
Our approach follows, and adapts to our context, the classical approach of Hall in~\cite{Hall}, that we now summarize. The starting point of Hall's approach is 
a classical formula, due to Perron~\cite{Perron}, that allows to compute the Lagrange value of a real number $\alpha$ given its continued fraction expansion.
If $\alpha=[\ab{0};\ab{1},a_2\dots]$, then Perron's formula for its Lagrange value is the following expression:
\begin{equation}\label{Perron_formula_CF}
	L(\alpha) = \limsup_{n\to \infty}{ [0;a_{n-1},a_{n-2},\dots,\ab{0}] + \ab{n} + [0;a_{n+1},a_{n+2},\dots]}.
\end{equation}
The expression inside the $\limsup$ consist of  \emph{central digit}, $\ab{n}$, and two \emph{tails} given by continued fraction expansions. 
It is well known that the set $\KK_N$ of numbers in $[0,1]$ whose continued fractions expansion  digits are all bounded by an integer $N$ form a Cantor set. At the heart of Hall's work, there  is a statement about these Cantor sets: he proves that if $N\geq 4$ the Cantor set $\KK_N$ is  sufficiently \emph{thick} so that the sum set $\KK_N+ \KK_N$ contains an interval. Hall's idea is then to construct  real numbers $\alpha$ that realize any sufficiently large value $L$ in the Lagrange spectrum by their continued fraction expansion, using larger and larger blocks formed by a large central digit $\ab{n}$ set to be the integer part of $L$, and two carefully selected finite tails (with bounded digits), which converge to two elements in the Cantor sets which add up to the fractional part of $L$. 
%``tails'' contained in Cantor sets thick enough to be sure that their sum is an interval of length at least $1$.
% The tails are chosen so that they converge to numbers in a Cantor e
%``tails'' contained in Cantor sets thick enough to be sure that their sum is an interval of length at least $1$.
Evaluating Perron's formula on this special sequence then yields the desired Lagrange value.

\smallskip
The starting point for our work is that, since the naive height of a geodesic is (half) the difference of its end points on the real line, one can carry a strategy similar to Hall's one, with the endpoints playing the role of the \emph{tails}. The first key tool used to carry our this strategy is a nice symbolic coding: we use Bowen-Series boundary expansions (with respect to a finite index subgroup $\Gamma< G$ without elliptic points) to code geodesics and to provide a geometric substitute for classical continued fractions expansions. 

To study Lagrange spectra with respect to a proper function in presence of several cusps, it is key that through this coding one can \emph{see} (large) excursions into \emph{each} cusp. To control these excursions, we use  \emph{decomposition into cuspidal words}, a notion which was introduced in our previous work \cite{AMU}. A delicate point we show is that by (locally) bounding the lengths of cuspidal words one is able to estimate the penetration into \emph{all} cusps (see \cref{LemmaEstimateExcursionCuspidalWords}), in order to then be able to achieve Lagrange values by prescribing larger excursions  in a given chosen cusp. 
%one has to carefully understand parabolic elements through the coding and consider a proper \emph{cuspidal acceleration} of the Bowen-Series expansion.

%\todo[inline]{Add here about controlling penetration in the other cusps and also about coding}

For the general setting considered in \cref{thm:Hallperturbations}, %this simple observation has to be carefully adapted to the perturbed setting.
the final key tool %in the proof of \cref{thm:Hallperturbations} 
is a generalization of Hall classical result on sums of Cantor sets, which gives a sufficient condition for such a sum to contain an interval.
We prove what we call a \emph{stable} version of Hall's result, where stability is meant here under (bounded size) perturbations, with respect to the Lipschitz norm, of the sum
function.  %that that result can be generalized to the case of a \emph{stable sum}, where we consider a sufficiently small perturbation of the sum of two Cantor sets and show that also this perturbed sum contains an interval.
We give  more details and %\cref{thm:stableHall} for 
%details and 
a precise formulation of this result in the following \cref{SectionStableHallTheorem}. 

Finally, let us point out that \cref{thm:Hall} is \emph{morally} a special case of \cref{thm:Hallperturbations}. We say morally since  one cannot formally  deduce \cref{thm:Hall} from \cref{thm:Hallperturbations}. In fact the function $\operatorname{height} (\cdot) $ is not proper if $X$ has more than one cusp and, moreover, $L_G(\cdot)$ can be expressed in terms of the essential height of a geodesic as a $\limsup$ as $|t| \to \infty$ (see~\eqref{heightlimsup}) and not as $t \to \infty$. 
However both these points can be though easily taken into account via simple technical tricks (i.e.\ artificially creating a proper function with the same spectrum and using symmetric geodesics, as in \cref{{heightlimsup}} for the modular surface). 
We chose to present in \cref{sec:Hallpenetration} an independent proof of \cref{thm:Hall} for two reasons: first of all, the arguments required to prove \cref{thm:Hall} are much more direct and essentially exploit coding and geometric arguments, combined with a generalization to the Fuchsian context of Hall's original strategy.
Secondly, we believe that the proof of  \cref{thm:Hall} might serve as a gentle guide for the reader to the ideas exploited in the rather more technical \cref{thm:Hallperturbations}.
Indeed, while the main strategy is the same, additional layers of technical difficulty in the proof of \cref{thm:Hallperturbations} come from the need to simultaneously control excursions in all cusps and to generalize Hall's result on the sum of Cantor sets in order to be able to deal with Lipschitz perturbations of the height function.
We also remark that the full strength of the symbolic coding we use, in particular of the decomposition into cuspidal words (see \cref{sec:cuspidalwords}), is only used for   \cref{thm:Hallperturbations}.
For \cref{thm:Hall} it would  in principle be enough to control, through the coding, only the excursions into the cusp at $\infty$.
We instead use the same Cantor set ($\BB_N \subset \partial \DD$ and its image $\KK_N \subset \mathbb{R}=\partial \HH$, see \cref{sec:Cantor}) for both the proof of \cref{thm:Hall} and the one of \cref{thm:Hallperturbations}, in order to prove only once the distortion and gaps estimates needed to apply results on the sum or the perturbation of the sum of Cantor sets.   

%In the proofs of our major results we exploited the strategy of Hall to prove the existence of an interval in the classical Lagrange spectrum. This strategy involves a condition for the sum of two Cantor sets to contain an interval.
% We have already stated in Proposition 3.1 what we call a stable version of this %result, where stability is meant under small perturbations, with respect to the Lipschitz norm, of the sum
%function. In this section we are first going to introduce the terminology to prove an abstract version of
%Proposition 3.1. We will then show that the Cantor sets we are interested in satisfy the assumption of
%this general result and finally prove the abstract version of the stable sums of Cantor sets.

\subsection{A stable version of Hall's theorem on the sum of Cantor sets}\label{SectionStableHallTheorem}
We conclude this section by formulating the generalization of Hall's theorem on the sum of Cantor sets on which our result on perturbed Lagrange spectra, namely \cref{thm:Hallperturbations}, is based. This statement is an abstract result on Cantor sets, which might be of independent interest, given the large literature on these type of questions (for example~\cite{Newhouse,Astels,MoreiraYoccoz}). In order to formulate the result, We need first to give a series of definitions on the way a Cantor set is constructed and the properties of its \emph{holes}. 

 \smallskip

Let $\KK$ be any Cantor set in $\RR$. One  can present $\KK$ as intersection $\bigcap_{n\in\NN}\KK(n)$ of unions $\KK(n)$ of closed disjoint intervals. We now define the notion of \emph{slow subdivision}: intuitively, the reader should keep in mind that this definition convey that the Cantor set is build step by step by removing exactly one \emph{hole} at each stage.  We adopt the following notation. For any  $K$ be compact interval  and any  $B$ open interval with $B\subset K$ (where the inclusion is obviously strict), we denote by $K^L$ and $K^R$ the two closed subintervals  of $K$ such that 
\[
	K=K^L\sqcup B\sqcup K^R.
\]  
As suggested by the notation, we assume that $K^L$ is on the left side of $B$ and $K^R$ is on the right side of the \emph{hole} $B$.

A \emph{slow subdivision} of $\KK$ is a family of closed sets $\bigl(\KK(n)\bigr)_{n\in\NN}$ with $\KK(n+1)\subset\KK(n)$ for any $n\in\NN$ which satisfies the following properties.
\begin{enumerate}
\item	Any set $\KK(n)$ is the union of $n+1$ disjoint closed intervals, where in particular we have
			\[
				\KK(0)=[\min\KK,\max\KK].
			\]
\item	For any $n$ there is exactly one compact interval $K$ in $\KK(n)$ and a non-empty open subinterval $B_K$ of $K$ such that
			\[	
				K\cap\KK(n+1)=K\setminus B_K=K^L \sqcup K^R,
			\]
where $K^L$ and $K^R$ are two disjoint, non-empty, closed subintervals in $\KK(n+1)$.
\item 	We have 
			\[
				\bigcap_{n\in\NN}\KK(n)=\KK.
			\]
\end{enumerate}

The \emph{holes} of a Cantor set $\KK$ are the connected components of its complement which are contained in the interval $[\min \KK, \max \KK]$. We remark that holes are maximal open intervals in the complement.

\begin{rem}\label{rem:orderedholes}
In a slow subdivision the holes are naturally ordered: for any $n\in\NN$, the  \emph{$n^{\text{th}}$ hole} is the unique $B_K\subset K$, given by condition (2), which is removed from the connected component $K$ of $\KK(n)$ at stage $n+1$, i.e.\  such that  $K \cap \KK(n+1) = K \setminus B_K$. 
%the $n^{\text{th}}$ hole is the unique 
%by definition there exists exactly connected component $K\subset\KK(n)$ and hole $B_K\subset K$ such that
% are respectively the compact connected component of $\KK(n)$ and the hole in it, i.e. are such that .
%we say that the  $n^{t}$ \emph{hole} of the subd $B_n:=B_K$, where  %We will call $K_n:=K$ the corresponding interval of level $n$ and $B_K$
\end{rem}

%\begin{definition}
 We say that a slow subdivision $\big(\KK(n)\big)_{n\in\NN}$ of the Cantor set $\KK$ satisfies the $\epsilon$-\emph{stable gap condition} for some $\epsilon>0$  if for any $n$, the $n^{\text{th}}$ hole $B_K$ (according to the terminology introduced in \cref{rem:orderedholes}) and its right and left closed intervals $K^L$ and $K^R$ satisfy
\begin{equation}\label{EqStableGapCondition}
	\frac{|B_K|}{|K^L| }<1-\epsilon
	\quad
	\text{ and }
	\quad 
	\frac{|B_K|}{|K^R| }<1-\epsilon.
\end{equation}
We say that the Cantor set $\KK$ satisfies the $\epsilon$-\emph{stable gap condition} if it admits a slow subdivision which satisfies the $\epsilon$-stable gap condition. 

Given two Cantor sets $\KK$ and $\FF$, we say that the pair of Cantor sets $(\KK,\FF)$ satisfies the $\epsilon$-\emph{size condition} if we have
\begin{equation}\label{EqStableSizeCondition}
	|B|\leq (1-\epsilon)|\KK|
	\quad
	\text{ and }
	\quad
	|C|\leq (1-\epsilon)|\FF|,
\end{equation}
where $|\KK|:=\max\KK-\min\KK$ is the length of $\KK$,  $|\FF|:=\max\FF-\min\FF$ is the length of $\FF$, and $B$ and $C$ are any pair of holes in $\KK$ and in $\FF$ respectively.

\smallskip
We can now state our stable version of Hall's Theorem. Consider the \emph{sum} function $S_0\colon\RR^2\to\RR$ defined by $S_0(x_1,x_2):=x_1+x_2$. Fix an open subset $U\subset\RR^2$; abusing the notation we still denote $S_0$ the restriction of $S_0$ to $U$. 
\begin{thm}[Stable Hall Theorem]\label{thm:stableHall}
Fix $\epsilon>0$ and let $\KK$ and $\FF$ be two Cantor sets in $\RR$, each one satisfying the $\epsilon$-stable gap condition and such that $\KK\times\FF\subset U$.
Assume that the pair $(\KK,\FF)$ satisfies the $\epsilon$-size condition. Then for any function $S\colon U\to\RR$ such that 
\begin{equation}\label{EqConditionStableHallTheorem}
\frac{1-\lipschitz(S-S_0)}{1+\lipschitz(S-S_0)}
>
1-\epsilon
\end{equation}
we have
\[
S(\KK \times \FF) = S\left( \left[\min \KK, \max \KK \right]\times \left[\min \FF, \max \FF \right]\right).
\]
\end{thm} 
\begin{rem}\label{rk:stable_gives_Cantorsum}
Let us show that this is indeed a generalization of the classical Hall's theorem. Observe that $S$ as in the statement is automatically continuous, indeed it is the sum of $S_0$ and a Lipschitz function. Thus,  if $K$ and $F$ are closed intervals, then $S(K\times F)$ is a closed interval too, by continuity of $S$. It hence follows that, in the special case when $S=S_0$ (and $U=\mathbb{R}^2$) is the sum function, the theorem shows that $\KK+\FF$ contains an interval, which is the conclusion of Hall's theorem. The assumptions of Hall, on the other hand, correspond to the $\epsilon$-stable and $\epsilon$-size condition for the limit case $\epsilon=0$. 
\end{rem}

\subsection*{Structure of the paper}
The rest of the paper is organized as follows.  In \cref{sec:background} we describe the symbolic coding of geodesics that we use. We first recall the simplest case of Bowen-Series coding and the notion of boundary expansions (see \cref{sec:background}). 
We also introduce the notions of \emph{cuspidal words} and decomposition into cuspidal words. 
In \cref{sec:Hallpenetration} we present the proof of  \cref{thm:Hall}, in particular introducing Hall's arguments. 
The starting point is a generalization of Perron's formula for the classical spectrum through symbolic coding, see \cref{lemma:Perron} in \cref{sec:Perron}. 
The only result needed whose proof is given later is \cref{thm:sumHallFuchsian} on the difference of Cantor sets. 

In \cref{sec:Cantor} we describe the Cantor sets $\BB_N\subset \DD $ consisting of endpoints of geodesics such that the lengths of cuspidal words is bounded by $N$. 
We then prove some distortions estimates on M\"obius transformations (see \cref{sec:distortionestimates}) which are then applied to show that the image $\KK_N \subset \RR$ of the Cantor sets $\BB_N$ satisfy (and their rigid images) satisfy the assumptions of the Stable Hall theorem (see \cref{conditionsStableHall}). 
In \cref{sec:sum_in_R} we can then prove \cref{thm:sumHallFuchsian} thus completing the proof of \cref{thm:Hall}.  

In the following \cref{sec:penetration} we show that by locally bounding the lengths of cuspidal words one can control the distance from a compact core of $X$ (see in particular \cref{LemmaEstimateExcursionCuspidalWords}).  \cref{thm:Hallperturbations} is then proved using these preliminary results  and the Stable Hall \cref{thm:stableHall} in \cref{sec:Hallperturbations}. In particular, the key step to implement Hall's strategy is this context is provided by \cref{PropositionFormulaContinuedFraction}, proved in \cref{sec:Perronperturbations}.  

In \cref{sec:proof_Stable_Hall} we then give the proof of the Stable Hall \cref{thm:stableHall}. 
Finally, two Appendices contain respectively the proofs of some Lemmas on parabolic words (\cref{app1}) and some estimates which relate the Lipschitz norm of $h$ to the Lipschitz norm of the auxiliary function $H$ introduced in \cref{sec:heightendpoints}, proved in \cref{app2}.

\section{Symbolic coding}\label{sec:background}
Let $\NN=\{0,1, \dots \}$ denote the natural numbers.
We will use $\HH$ and the unit disk $\DD = \{ z \in \CC , |z|<1 \}$ interchangeably, by using the identification $\mathscr{C} \colon \HH \to \DD$ (here $\mathscr{C}$ stays for \emph{Cayley} map) given by 
\begin{equation}\label{eq:Cayley}
	\mathscr{C}(z)= \frac{z-i}{z+i}, \qquad z \in \HH.
\end{equation}

Let $\SL(2, \RR)$ be the set of $2\times 2$ matrices with real entries and determinant one, and similarly for $\SL(2,\CC)$.
The group $\SL(2, \RR)$ acts on $\HH$ by \emph{M\"obius transformations} (or \emph{homographies}).
Given $g \in \SL(2, \RR)$ we will denote by $g\cdot z$ the action of $g$ on $z \in \HH$ given by
\[
z \mapsto g\cdot z = \frac{a z+b}{c z+d}, \qquad \text{if} \quad g= \begin{pmatrix}a&b\\ c&d \end{pmatrix}.
\] 
This action of $\SL(2, \RR)$ on $\HH$ induces an action on the unit tangent bundle $T^1 \HH$, by mapping a unit tangent vector at $z$ to its image under the derivative of $g$ in $z$, which is a unit tangent vector at $g\cdot z$.
This action is transitive but not faithful and its kernel is exactly $\{ \pm \Id \}$, where $\Id$ is the identity matrix.
Thus, it induces an isomorphism between  $T^1 \HH$ and $\PSL(2, \RR) =\SL(2, \RR) / \{ \pm \Id \}$. Throughout the paper, we will often write $A \in \PSL(2, \RR)$ and denote by $A \in \SL(2,\RR)$ the equivalence class of the matrix $A$ in $\PSL(2, \RR)$. Equality between matrices in $\PSL(2,\RR)$ must be intended as equality as equivalence classes.
The group $\SL(2,\CC)$ also acts by M\"obius transformations on the Riemann sphere $\overline{\CC}:=\CC\cup\{\infty\}$ and we will denote this action with $g\cdot z$ too. 

\subsection{Cutting sequences}\label{subsection:cuttsequences}%Was joint with boundary, check refs
For a special class of Fuchsian groups, Bowen and Series developed in~\cite{BowenSeries} a geometric method of symbolic coding of points on $\partial \DD$, known as \emph{boundary expansion}, that allows to represent the action of a set of suitably chosen generators of the group as a subshift of finite type. 
Boundary expansions  can be thought of as a geometric generalization of the continued fraction expansion, which is related to the boundary expansion of the geodesic flow on the modular surface (see~\cite{Series:Modular} for this connection).  
We will now recall two  equivalent definitions of the simplest case of boundary expansions, either as cutting sequences of geodesics on $X=\Gamma\backslash\HH$ or as itineraries of expanding maps on $\partial \DD$.
For more details and a more general treatment we refer to the expository introduction to boundary expansions given by Series in~\cite{Series:TS}.
  
Let  $\Gamma \subset \PSL(2,\RR)$ be a Fuchsian group and assume in this section that $\Gamma$ be a co-finite, non cocompact and does not contain elliptic elements\footnote{The choice of a different name, $\Gamma$, for the Fuchsian group in this section is deliberate. To study the Lagrange spectrum of of general co-finite, non cocompact Fuchsian group $G$, that can a priori contain elliptic elements, we will exploit a subgroup $\Gamma< G$ without elliptic elements, as in this section.}.  
One can see that $\Gamma$  admits a fundamental domain which is an \emph{ideal polygon} $\cF$ in $\DD$, that is a hyperbolic polygon having finitely many vertices $\xi$ all lying on $\partial\DD$ (see for example Tukia~\cite{Tukia}).
We will denote by $s$ the sides of $\cF$, which are geodesic arcs with endpoints in $\partial\DD$.
Geodesic sides appear in pairs, i.e.\ for each $s$ there exists a side $\overline{s}$ and an element $g$ of $\Gamma$ such that the image $g(s)$ of $s$ by $g$ is $\overline{s}$.
Let  $2d$ ($d\geq 2$) be the number of sides of of $\cF$. 
Let $\cA_0$  be a finite alphabet of cardinality $d$ and label the $2d$-sides ($d\geq 2$) of $\cF$ by letters in 
\[
	\cA = \cA_0 \cup \overline{\cA_0} = \{ \alpha \in \cA_0\} \cup \{ \overline{\alpha}, \alpha \in \cA_0\} 
\]
in the following way.
Assign to a side $s$ an internal label $\alpha$ and an external one $\overline{\alpha}$.
The side $\overline{s}$ paired with $s$ has $\overline{\alpha}$ as internal label and $\alpha$ as the external one.
We then  see that the pairing given by $g(s) =\overline{s}$ transports coherently the couple of labels of the side $s$ onto the couple of labels of the side $\overline{s}$.
Let us denote by $s_\alpha$  the side of $\cF$ whose \emph{external} label is $\alpha$. 
A convenient set of generators for $\Gamma$ is given by the family of isometries $g_\alpha \in \PSL(2,\RR)$ for $\alpha\in\cA_0$, where $g_\alpha$ is the isometry which sends the side $s_{\overline{\alpha}}$ onto the side $s_{\alpha}$, and their inverses $g_{\overline{\alpha}}:= g_\alpha^{-1}$ for $\alpha\in\cA_0$, such that $g_\alpha^{-1}(s_{\alpha})=s_{\overline{\alpha}}$, see Theorem 3.5.4 in~\cite{KatokFuchsian}.  
Thus, $\cA$ can be thought as the set of labels of generators, see \cref{fig:fundomain}.
It is convenient to define an involution on $\cA$ which maps $\alpha \mapsto \overline{\alpha} $ and $\overline{\alpha}  \mapsto  \overline{\overline{\alpha}} = \alpha$. 

\begin{figure}
\centering
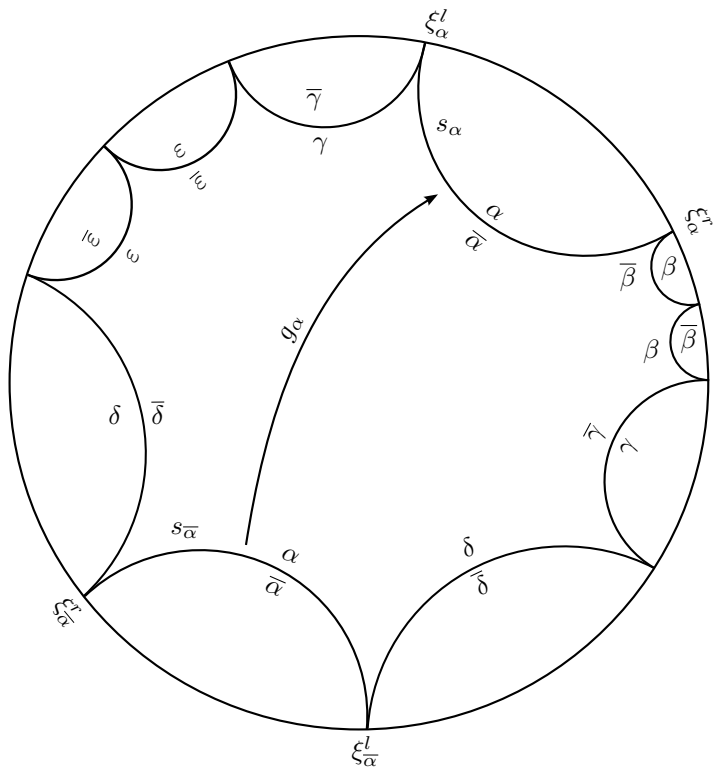
\caption{A hyperbolic fundamental domain, with sides labeling and the action of the generator $g_\alpha$.}
\label{fig:fundomain}
\end{figure}

Since $\cF$ is an ideal polygon, $\Gamma$ is a free group.
Hence every element of $\Gamma$ as a unique representation as a \emph{reduced word} in the generators, i.e.\ a word in which an element is never followed by its inverse.
We transport the internal and external labeling of the sides of $\cF$ to all its copies in the tessellation by ideal polygons given by all the images $g(\cF)$ of $\cF$ under  $g\in \Gamma$.
We label a side s of a copy $g(\cF)$ of $\cF$ with the labels of the side $g^{-1}(s) \in \partial \cF$.
We remark that this is well defined since we have assigned an internal and an external label to each side of $\cF$, and this takes into account the fact that every side of a copy $g(\cF)$ belongs also to another adjacent copy $g'(\cF)$.

Let  $\gamma$  be a hyperbolic geodesic ray, starting from the center $0$ of the disk and ending at a point $\xi \in \partial \DD$.
The \emph{cutting sequence} of $\gamma$ is the infinite reduced word obtained by concatenating the \emph{exterior} labels of the sides of the tessellation crossed by $\gamma$, in the order in which they are crossed.
In particular, if the cutting sequence of $\gamma$ is $\ab{0}, \ab{1}, \dots$, the $i^{\text{th}}$ crossing along $\gamma$ is from the region $g_{\ab{0}} \dots g_{\ab{i-1}} (\cF)$ to $g_{\ab{0}} \dots g_{\ab{i}}(\cF)$ and the sequence of sides crossed is
 \begin{equation}\label{crossing}
 	g_{\ab{0}} g_{\ab{1}} \cdots g_{\ab{n-1}} (s_{\ab{n}}), \qquad n \in \NN.
 \end{equation}
We remark that, since two distinct hyperbolic geodesics meet at most in one point, a word arising from a cutting sequence is reduced.
In other words, hyperbolic geodesics \emph{do not backtrack}. 
%More in general, one can associate cutting sequences to any oriented piece of a hyperbolic geodesics and describe a cross-section of the geodesic flow in terms of the shift map on boundary expansions (see for example Series~\cite{Series:Modular}), but we will not need it in this paper.

\smallskip

We complete this section explaining how to code \emph{complete} geodesics passing through $\cF$ at time zero. 
If $\gamma$ is a complete geodesic, parametrized in such a way that $\gamma(0)\in\cF$, let $\gamma_\pm (t)\colon\RR_+\to\DD$ defined by $\gamma_\pm\colon t\mapsto\gamma(\pm t)$.
In other words, $\gamma_+$ is the ray obtained moving along $\gamma$ forward in time and $\gamma_-$ is the one obtained moving along $\gamma$ backwards in time.
Code the first one by $(\bb{n})_{n\in\NN}$ and the second one by $(c_n)_{n\in \NN}$.
Then the cutting sequence of $\gamma$ is the infinite word $(\ab{n})_{n\in\ZZ}$ defined by
\[
	a_n =\begin{cases}
			\bb{n}, & \text{if $n\geq 0$},\\
			\overline{c_{n-1}}, & \text{if $n<0$}.
			\end{cases}.
\]
The bar for negative $n$'s is due to the fact that we were moving along $\gamma$ in the reverse orientation when defining the sequence $(c_n)_{n\in \NN}$.
 
\subsection{Boundary expansions}\label{subsection:boundary}
Let us now explain how to recover cutting sequences of geodesic rays by itineraries of an expanding map on $\partial \DD$.
The action of each $g \in \Gamma$ extends by continuity to an action on $\partial \DD$ which will be denoted by $\xi \mapsto g(\xi)$.   
Let $\Arc [\alpha]$ be the closed arc on $\partial\DD$ such that $\Arc[\alpha]\cup s_{\alpha}$ is the boundary of the connected component which is disjoint from the interior of $\cF$. 
Then it is easy to see from the geometry that the action $g_\alpha \colon \partial \DD \to \partial \DD$ associated to the generator $g_\alpha$ of $\Gamma$ sends the complement of $\Arc [\overline{\alpha}]$ to $\Arc [\alpha]$.
Moreover, if for each $\alpha \in \cA$ we denote by $\xi_\alpha^l$ and $\xi_\alpha^r$ the endpoints of the side $s_\alpha$, with the  convention that the right follows the left moving in clockwise sense on $\partial \DD$, we have 
\begin{equation}\label{action_g_on_labels}
	g_\alpha(\xi_{\overline{\alpha}}^r)=\xi_\alpha^l
	\qquad \text{and} \qquad
	g_\alpha(\xi_{\overline{\alpha}}^l)=\xi_\alpha^r.
\end{equation}
Some times it will be useful to write $\xi_\alpha^l=\inf\Arc[\alpha]$ and $\xi_\alpha^r=\sup\Arc[\alpha]$.
Let $\Arc = \bigcup_{\alpha} \stackrel{\circ}{ \Arc[\alpha]}\subseteq \partial \DD$, where $\stackrel{\circ}{ \Arc[\alpha] }$ denotes the arc $\Arc[\alpha]$ without endpoints. 
Define $F \colon \Arc \to \partial\DD$ by 
\[
	F(\xi) = g_{\alpha}^{-1}(\xi), \qquad \text{if } \xi \in \stackrel{\circ}{ \Arc[\alpha]}. 
\]
Let us call a point $\xi \in \partial \DD$ \emph{cuspidal} if it is a vertex of the ideal tessellation with fundamental domain $\cF$ and \emph{non-cuspidal} otherwise.
One can see that $\xi$ is non-cuspidal point if and only if $F^n(\xi)$ is defined for any $n \in \NN$.  
One can \emph{code} a trajectory $\{ F^n(\xi), n \in \NN\}$ of a non-cuspidal point $\xi \in \partial \DD$ with its \emph{itinerary} with respect to the partition into arcs $\{ \Arc [\alpha], \alpha \in \cA \}$, that is by the sequence $(\ab{n})_{n \in \NN}$, where $\ab{n} \in \cA$ are such that  $F^n (\xi) \in \Arc [\ab{n}]$ for any $n \in \NN$.
We will call such sequence the \emph{boundary expansion} of $\xi$.

Moreover, in analogy with the continued fraction notation, we will write 
\[
	\xi =[\ab{0}, \ab{1}, \dots]_{\partial\DD}. 
\]
When we write the above equality or say that $\xi$ has boundary expansion $(\ab{n})_{n \in \NN}$ we implicitly assume that $\xi$ is non-cuspidal.
 
One can show that the only restrictions on letters which can appear in a boundary expansion $(\ab{n})_{n \in \NN}$  is that $\alpha$ cannot be followed by $\overline{\alpha}$, that is
\begin{equation}\label{eqnobacktrack}
\ab{n+1}\not=\overline{\ab{n}}	 \qquad \text{ for any  $n\in\NN$}.
\end{equation}
We will call this property the \emph{no-backtracking condition}.
Boundary expansions can be defined also for cuspidal points (see Remark 4.3 in \cite{AMU}) but are unique exactly for non-cuspidal points.
Every  sequence in $\cA^{\NN}$ which satisfies the no-backtracking condition can be realized as a boundary expansion (of a cuspidal or non-cuspidal point).  

We will adopt the following notation. 
Given a sequence of letters $\ab{0}, \ab{1},  \dots, \ab{n}$, let us denote by
\[
	\Arc [\ab{0}, \ab{1},  \dots, \ab{n}] =\overline{ \Arc [\ab{0}] \cap F^{-1}( \Arc [ \ab{1}] )\cap  \dots \cap F^{-n}(\Arc [\ab{n}])}
\]
the closure of  set of points on $\partial\DD$ whose boundary expansion starts with  $\ab{0}, \ab{1},  \dots, \ab{n}$.
One can see that $\Arc[\ab{0},   \dots, \ab{n}]$ is a connected arc on $\partial \DD$ which is non-empty exactly when the sequence satisfies the no-backtracking condition~\eqref{eqnobacktrack}.
From the definition of $F$, one can work out that
\begin{equation}\label{arcexpression}
	\Arc [\ab{0}, \ab{1},  \dots, \ab{n}] = g_{\ab{0}} \dots g_{\ab{n-1}} \Arc [\ab{n}].
\end{equation}
Thus  two such arcs are \emph{nested} if one word contains the other as a beginning.
For any fixed $n\in \NN$, the arcs of the form $\Arc[\ab{0}, \ab{1},  \dots, \ab{n}]$, where $\ab{0}, \ab{1},  \dots, \ab{n}$ vary over all possible sequences of $n$ letters in $\cA$ which satisfy the no-backtracking condition, will be called an \emph{arc of level} $n$.  
To produce the arcs of level $n+1$, each arc of level $n$ of the form $\Arc[\ab{0}, \ab{1},  \dots, \ab{n}]$ is partitioned into $2d-1$ arcs, each of which has the form $\Arc[\ab{0}, \ab{1},  \dots, \ab{n+1}]$  for $\ab{n+1} \in \cA\setminus \{\overline{\ab{n}}\}$.
Each one of these arcs corresponds to one of the arcs cut out by the sides of the ideal polygon $\ab{0}\ab{1}  \dots \ab{n} \cF$ and contained in the previous arc $\Arc[\ab{0}, \ab{1},  \dots, \ab{n}]$.  

We summarize the previous discussion in the next result.

\begin{prop}[Bowen-Series]\label{thm:BowenSeries}
If $(\ab{n})_{n\in\NN}$ is the boundary expansion of $\xi\in\partial\DD$ we have
\[
	\xi=\bigcap_{n\in\NN} g_{\ab{0}}\dots g_{\ab{n}} \Arc[\ab{n+1}].
\]
Moreover, if $t\mapsto\gamma(t)$ is a hyperbolic geodesic ray with $\gamma(0)\in \cF$ and ending at $\xi=\gamma(+\infty)\in\partial\DD$, then the cutting sequence $(\ab{n})_{n\in\NN}$ of $\gamma$ coincides with the boundary expansion of $\xi$. 

The Bowen-Series map $F\colon\partial\DD\to\partial\DD$ acts as the right shift on the space $\Sigma\subset\cA^\NN$ 
of those infinite words $(\ab{n})_{n\in\NN}$ satisfying the no-backtracking condition~\eqref{eqnobacktrack}.
In other words we have
\[
	F\bigl([\ab{0},\ab{1},\ab{2},\dots]_{\partial\DD}\bigr)=
[\ab{1},\ab{2},\dots]_{\partial\DD}.
\] 
\end{prop}

Notice that the combinatorial no-backtracking condition~\eqref{eqnobacktrack} corresponds to the no-backtracking geometric phenomenon between hyperbolic geodesics we mentioned earlier.

\subsection{Cuspidal words and cuspidal sequences}\label{sec:cuspidalwords}
We now define an acceleration of the boundary expansion.
The acceleration is obtained by grouping together all steps which correspond to excursions in the same cusp,  in a similar way to how the Gauss map is obtained from the Farey map in the theory of classical continued fractions expansions.

%We first describe sequences that can taken to be  boundary expansions of the ideal vertices of $\cF$. 

\begin{defi}\label{defnestedarcs}
A \emph{left cuspidal word} (respectively a \emph{right cuspidal word}) is a word $\ab{0} \dots \ab{k}$ in the alphabet $\cA$ which satisfies the no-backtracking condition~\eqref{eqnobacktrack} and such that 
the $k+1$ arcs
\[
	\Arc[\ab{0}], \quad \Arc[\ab{0},\ab{1}], \quad \dots \quad  \Arc[\ab{0}, \dots, \ab{k-1}], \quad \Arc[\ab{0}, \dots, \ab{k}]
\]
all share as a  common left endpoint the left  endpoint $\xi_{\ab{0}}^l$ of $\Arc [\ab{0}]$ (respectively as right endpoint the  right endpoint $\xi_{\ab{0}}^r$ of $\Arc [\ab{0}]$), see \cref{fig:cuspidalword}. 
We simply write that $\ab{0} \dots \ab{k}$ is a \emph{cuspidal word} when left or right is not specified. 
We say that a sequence $(\ab{n})_{n \in \NN}$ is a \emph{cuspidal sequence} if any word of the form $\ab{0}\dots \ab{n}$ for $n \in \NN$ is a cuspidal word and that it is \emph{eventually cuspidal} if there exists $k \in \NN$ such that $(\ab{n+k})_{n \in \NN}$ is a {cuspidal sequence}. 
\end{defi}

\begin{figure}
\centering
\def\svgscale{0.8}
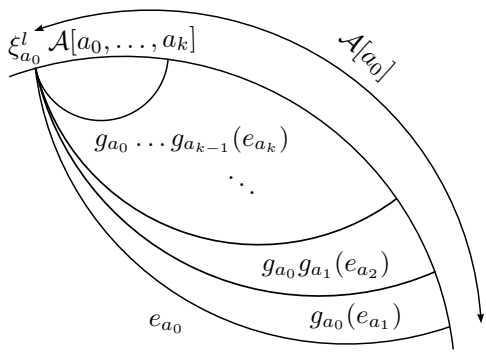
\caption{A left cuspidal word $\ab{0}\dots\ab{k}$.}
\label{fig:cuspidalword}
\end{figure}

Equivalently, $\ab{0} \dots \ab{k}$ is a left (respectively right) cuspidal word exactly when the arc $\Arc [\ab{0}, \dots, \ab{k}] \subset \partial \DD$ has a vertex of $\cF$ as its left (respectively right) endpoint.
We remark that given an ideal vertex $\xi$, there is a \emph{unique} left (right) cuspidal word of length $k+1$ such that the arc $\Arc[\ab{0}, \dots, \ab{k}]$ has $\xi$ as left (right) endpoint.
Indeed, such word can be obtained as follows.
Let $\ab{0}$ be such that $\Arc[\ab{0}]$ has $\xi$ as its left (right) endpoint. 
For any $0\leq i < k$,  the arc  $\Arc [\ab{0},\dots, \ab{i}]$ of level $i$ is subdivided at level $i+1$ into $2d-1$ arcs of level $i+1$  and $\Arc[\ab{0},\dots, \ab{i+1}]$ is  the unique one which contains the left (respectively right) endpoint of $\Arc[\ab{0},\dots, \alpha_{i}]$.

%\begin{rem}\label{boundary_cuspidal}
%Boundary expansions can be defined for all points in the boundary $\partial \DD$ (and not only non-cuspidal points)  as follows.
%\emph{Cuspidal sequences} can be taken by definition to be boundary expansions of vertices of the ideal polygon $\cF$.
%More precisely, if $\xi$ is the left endpoint of $e_\alpha$ and the right  endpoint of $e_\beta$, that is $\xi=\xi^l_\alpha=\xi^r_\beta$, then $\xi$ has exactly two boundary expansions which are respectively given by the unique left boundary expansion starting with $\alpha$ and the unique right boundary expansion starting with $\beta$.
%Similarly, \emph{eventually cuspidal} sequences can be taken to be boundary expansions of cuspidal points.
%\end{rem}

%\mdoubts{Forse possiamo spostare la parte che segue dopo i due lemmi?}
\smallskip
We will use cuspidal words to decompose an infinite word into blocks.
Let us begin with a geometric description of this process.
Let $\gamma\colon[0,+\infty)\to\DD$ be a geodesic such that $\gamma(0)\in\cF$ and that does not converge to a cuspidal point.
Call $(\ab{n})_{n\in\NN}$ its cutting sequence.
Let $(t_n)_{n\in\NN}$ be the sequence of times $t_n$ when $\gamma$ crosses a side of the tessellation of $\DD$ given by $\cF$.
More precisely let $t_0$ such that $\gamma(t_0)\in s_{\ab{0}}$ and
\begin{equation}\label{eq:timescuspidalacceleration}
	\gamma(t_n)\in g_\ab{0}\circ\dots\circ g_\ab{n-1} (s_\ab{n}),
\end{equation}
for any $n\geq 1$.
Define $n(0)=0$ and, inductively for $r\in\NN$, define $n(r+1)$ such that the segment $\gamma[t_{n(r)},t_{n(r+1)})$ only intersects copies $g(s)$ of sides $s$ of $\cF$ all sharing one common endpoint.

Having this picture in mind, we define the cuspidal decomposition of an infinite word as follows.
Consider an infinite word $(\ab{n})_{n\in\NN}$ satisfying the no-backtracking condition~\eqref{eqnobacktrack} and that is not eventually cuspidal.
Let $n(0)=0$ and define $n(1)\geq 1$ to be the minimum time such that the arcs
\[
		\Arc[\ab{n(0)}]
			\qquad
			\text{and}
			\qquad
		\Arc[\ab{n(0)},\dots,\ab{n(1)}]
\]
do not have a common endpoint.
In other words, $n(1)$ is the minimum time such that the word $\ab{n(0)}\dots\ab{n(1)}$ is not cuspidal.
Then set $C_0=\ab{0}\dots\ab{n(1)-1}$ to be the first maximal cuspidal word in the infinite word $(\ab{n})_{n\in\NN}$.
Similarly, for $r\geq 0$ define 
\[
	n(r+1) = \min\{n>n(r): \ab{n(r)}\dots\ab{n(r+1)} \text{ is not a cuspidal word}\},
\]
and $C_r=\ab{n(r)}\dots\ab{n(r+1)-1}$ to be the $r$-th cuspidal maximal word in $(\ab{n})_{n\in\NN}$.
It is clear then that concatenating the cuspidal words $(C_r)_{r\in\NN}$ we get the same infinite word as $(\ab{n})_{n\in\NN}$ that is $\ab{0}\ab{1}\dots\ab{n}\dots =C_0C_1\dots C_r\dots$.

In the sequel, we will also decompose bi-infinite words into cuspidal subwords.
This is done as before, the only difference is that $C_0$ is the maximal cuspidal word containing $\ab{0}$, and hence $n(0)\leq 0$.

\subsection{Parabolic words}\label{subsec:parabolicwords}
We end this section with two Lemmas that give a combinatorial description of cuspidal words, by showing that cuspidal words are obtained by repeating \emph{parabolic words} (defined below), which are in one to one correspondence with cusps (see \cref{repeatedparabolic}). The Lemmas  were essentially proved in \cite{AMU} (see Lemmas 4.8 and 4.9 in \cite{AMU}).
For completeness, we include their easy proofs in \cref{app1}. 

\begin{lemma}\label{LemmaCombinatorialPropertiesCuspidal}
Consider a word $\ab{0}\dots\ab{n}$ in the alphabet $\cA$ which satisfies the no-backtracking condition~\eqref{eqnobacktrack}.
Then
\begin{enumerate}
	\item The word $\ab{0}\dots\ab{n}$ is left cuspidal if and only if 
		\[
			g_{\ab{k}}(\xi^l_{\ab{k+1}})=\xi^l_{\ab{k}}
			\qquad
			\text{for any}
			\qquad
			k=0,\dots,n-1.
		\]
\item The word $\ab{0}\dots\ab{n}$ is right cuspidal if and only if 
		\[
			g_{\ab{k}}(\xi^r_{\ab{k+1}})=\xi^r_{\ab{k}}
			\qquad
			\text{for any}
			\qquad
			k=0,\dots,n-1.
		\]
\item The word $\ab{0}\dots\ab{n}$ is left (resp.\ right) cuspidal if and only if $\overline{\ab{n}}\dots\overline{\ab{0}}$ is right (resp.\ left) cuspidal.
\end{enumerate}
\end{lemma}
%The first two points of the Lemma were proved in \cite{AMU} (see Lemma 4.8). The last is left as an exercise for the reader.  For the proof of the next Lemma, see the proof of Lemma 4.9 in \cite{AMU}.
%The Lemma was essentially proved in \cite{AMU} (the first two points as Lemma ADD, the last in ADD).  See Appendix \ref{app1}.

The next Lemma connects cuspidal words and parabolic elements in $\Gamma$.
\begin{lemma}\label{Parabolicwordslemma}
Let $\ab{0}\dots\ab{n}$ be a left cuspidal word such that $\ab{0}\dots\ab{n}\ab{0}$ is a left cuspidal word too.
Then $g=g_{\ab{0}}\circ\dots\circ g_{\ab{n}}$ is a parabolic element of $\Gamma$ whose unique fixed point is 
\[
\xi^l_{\ab{0}}=\xi^r_{\overline{\ab{n}}}\in\partial\DD.
\]
\end{lemma}

%\todo[inline]{To add refs to AMU}
% This Lemma was proved in \cite{AMU} as Lemma ADD.  See Appendix \ref{app1}. 
\smallskip

A word $\ab{0}\dots\ab{n}$ as in the Lemma before is called a \emph{left parabolic word} if it has minimal length.
In the same way one defines a \emph{right parabolic word}.
Let us remark that $\ab{0}\dots\ab{n}$ is left parabolic if and only if its \emph{inverse word} $\overline{\ab{n}}\dots\overline{\ab{0}}$ is right parabolic, and the corresponding fixed point is $\xi^l_{\ab{0}}=\xi^r_{\overline{\ab{n}}}$. 
We write simply parabolic word when left or right is not specified.

\smallskip
From the Lemma,  the following combinatorial description  of cuspidal sequences follows (see point (2) of Lemma 4.9 in \cite{AMU}). 
\begin{cor}\label{repeatedparabolic}
For any right (left) cuspidal sequence $(\ab{n})_{n\in\NN}$ there exists an integer $k\geq 1$ and a right (left) parabolic word $a_0 a_1 \cdots a_{k-1}$ such that $(\ab{n})_{n\in\NN}$ is obtained repeating the word periodically, i.e.\ $a_n=a_{n \mod k}$ for every $n \in \NN$.
\end{cor}

\begin{rem}
If $g_\alpha=g_{\overline{\alpha}}^{-1}$ is parabolic generator of $\Gamma$, the two sides $s_\alpha$, $s_{\overline{\alpha}}$ share a common vertex $\xi$ and  $\xi = g_\alpha (\xi) = g_{\overline{\alpha}}^{-1}(\xi)$ is a cusp described by the (length one) parabolic words $\alpha$ and $\overline{\alpha}$.
More in general, one can see that cusps of $\Gamma\backslash\DD$ are in bijection with parabolic words, modulo inversion operation and cyclical permutation of the entries. %In particular, if $g_\alpha=g_\overline{\alpha}^{-1}$ is a parabolic generator of $\Gamma$, the two sides $s_\alpha, s_{\overline{\alpha}}$ share a common vertex $\xi$ and $g_\alpha(\xi)=\xi$
\end{rem}

\section{Hall ray for the hyperbolic height}\label{sec:Hallpenetration}
In this section we prove \cref{thm:Hall}, namely the existence of Hall rays for Diophantine approximation on Fuchsian groups.
This section is also meant as a guiding line for the following ones (and in particular \cref{sec:Hallperturbations}), where the more difficult and technical proof of \cref{thm:Hallperturbations} will be presented. 
In particular, \cref{subsec:Hall_argument} introduces (the adaptation of) Hall's original argument using boundary expansions as a replacement of continued fraction which is used in both proofs (and referred to in \cref{sec:outline}). 
In \cref{sec:setup} we choose a convenient fundamental domain for $G$ and define the tessellation with respect to which to code geodesics. 
We then describe the Cantor set on the boundary $\partial \DD$ of the disk corresponding to endpoints of geodesics which have bounded penetration in the cusps (see \cref{sec:cantorpreliminaries}). 
This Cantor set will be proven to satisfy the assumptions of the Stable Hall Theorem later on, in \cref{sec:sum_in_R}.
We will prove a Perron-like formula for sufficiently high geodesics in \cref{sec:Perron} and then prove \cref{thm:Hall}.

\subsection{Preliminaries to the proofs of the main results}\label{sec:setup}
In this section we are going to prepare the ground for the proofs of our two main results, \cref{thm:Hall} and \cref{thm:Hallperturbations}. 
Let  $G$  be a fixed Fuchsian group. We assume that $G$ is a non-uniform lattice and denote by $X=G\backslash\HH$ the corresponding finite volume, not compact (orbifold) surface.  We also assume that it is zonal, namely that $\infty$ is fixed by a parabolic element of $G$ and hence projects to a cusp of $X$. 

Conjugating $G$ with an appropriate element of $\PSL(2,\RR)$ which fixes $\infty$, we normalize $G$ so that $m=1$, where $m$ is the height of the fundamental horodisk at infinity (see~\eqref{eq:fundamentalhorodisk}). We remark that, for  \cref{thm:Hall}, we are not losing any generality,  since, by  \cref{rem:conjugacy_rays}, the presence of Hall rays in $\cL(G, \infty)$ is preserved by this conjugation. For \cref{thm:Hallperturbations}, we will first treat the case when $G$ has $m=1$, then   we will show how to deduce a result for $m\neq 1$ from the result for $m=1$ (see the proof of \cref{thm:Hallperturbations}). 

Let $\mu>0$ the \emph{width} of this cusp after this normalization, meaning that the matrix $p=\left( \begin{smallmatrix} 1 & \mu \\ 0 & 1 \end{smallmatrix}\right)$ is in $G$, and that $p$ is not the power of another element in $G$.

Since $G$ has finite covolume, it is finitely generated.
Any finitely generated Fuchsian group contains a finite index normal subgroup $\Gamma$ not containing any elliptic elements (see~\cite{Fox,EEK}).
It is well known that such $\Gamma$ admits a fundamental domain for the action on $\DD$ which is an \emph{ideal} polygon, that is a hyperbolic polygon having finitely many vertices all lying on $\partial\DD$ (see~\cite{Tukia}).
For technical reasons, we require some additional properties (in particular Condition (5) in the Lemma below) for the fundamental domain.
We will hence construct a suitable fundamental domain $\cF$ for the action of $\Gamma$ on $\DD$. The following Lemma summarizes the choice of $\cF$. 
The reader can refer to \cref{fig:dandh} for an example of a fundamental domain satisfying the requirements of the Lemma.

Here, and in the rest of the paper, we denote by $\arclength{\Arc[\alpha]}$ the length of the arc $\Arc[\alpha]$. Recall that $\varphi\colon \DD \to \HH$ is the inverse of the  Cayley map defined in \cref{eq:Cayley}. %$\mathscr{C}$ is the   and $\varphi$ its inverse. 

\begin{figure}
\centering
\def\svgscale{0.7}
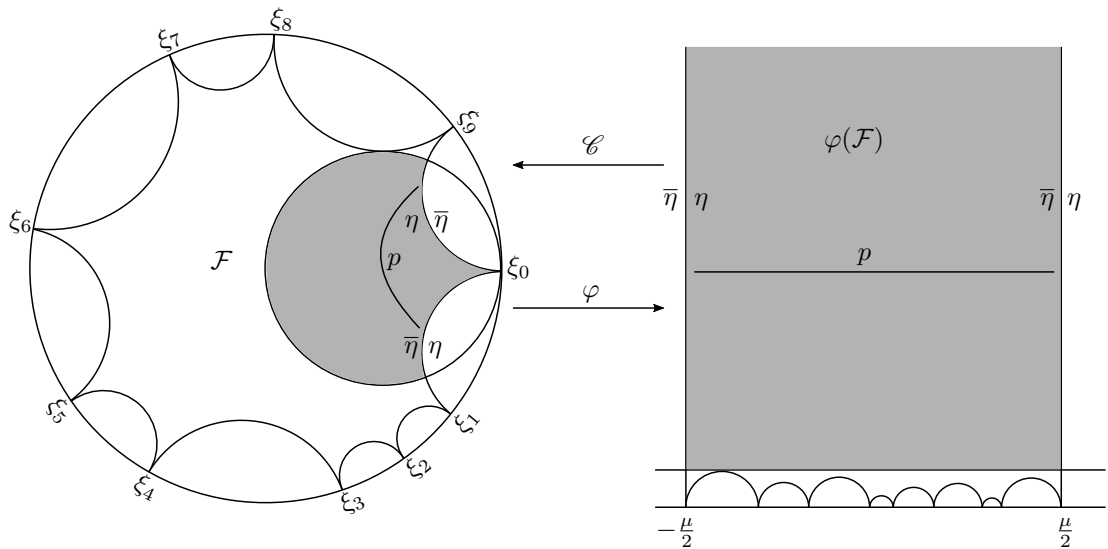
\caption{The fundamental domain $\cF$ described in \cref{costruzioneF}, with the horodisk $\cU_1$ in grey in both figures.}
\label{fig:dandh}
\end{figure}

\begin{lemma}\label{costruzioneF}
There exists a fundamental domain $\cF\subset \DD$ for $\Gamma$ with the following properties:
\begin{enumerate}
\item $\cF $ is an ideal polygon with vertices $\xi_0, \dots, \xi_{2d-1} \in \partial \DD$, where  $\xi_0:=1$ and hence $\varphi(\xi_0)=\infty$; 
\item the parabolic element $p$ is a generator and identifies the side of $\cF$ which share $\xi_0$ as endpoint; more precisely $p=g_\eta$, where $\eta \in \cA$ is  such that $\xi_\eta^l=\xi_0 =\xi_{\overline{\eta}}^r$; 
\item The endpoints of $s_\eta$ and $s_{\overline{\eta}}$ different than $\xi_0$, that correspond to $\xi_1$ and $\xi_{2d-1}$, are such that 
\[
	\varphi (\xi_\eta^r) = \varphi (\xi_1) = \frac{\mu}{2}, \qquad 
	\varphi (\xi_{\overline{\eta}}^l)= \varphi(\xi_{2d-1}) =-\frac{\mu}{2};
\]
\item the origin of the disk $\DD$ belongs to $\cF$;
\item for every arc $\Arc[\alpha] $ underlying a side $s_\alpha$ of $\cF$, we have 
\begin{equation}\label{eq:UpperSetHoleOrderZero}
	\arclength{\Arc[\alpha]}<\pi, \qquad \forall \alpha \in \cA.
\end{equation}
\end{enumerate}
\end{lemma}

As we said above, Condition (5) is needed  for technical reasons (more precisely for the distortion estimates in \cref{sec:distortionestimates}).

\begin{proof}
We will first construct a fundamental domain in $\HH$ so that it verifies Condition (1), (2) and (3), then lift it to $\DD$ and modify the choice so that also the other Conditions are verified. 

Let $H$ be the subgroup of $\Gamma$ generated by $p$. Since $p$ acts on the hyperbolic plane $\HH$ by $z\mapsto z+\mu$,  a fundamental domain for the action of $H$ on the hyperbolic plane $\HH$ is given by any vertical strip of width $\mu$ with the two vertical geodesics are identified by $p$.  We choose the one centered on the vertical axis $\{ z : -\frac{\mu}{2}<\Re(z)<\frac{\mu}{2}\}$.
Recall that, given a matrix $g\in\PSL(2,\RR)$, not fixing $\infty$, its \emph{isometric circle $I_g$} is the Euclidean semi-circle centered at $g^{-1}\cdot\infty=-d/c$ with radius $r_g=1/|c|$.
Then a fundamental domain  for $\Gamma$ is given by the intersection of a fundamental domain for $H$ with the points that lie outside \emph{every} isometric circle $I_g$ given by the elements $g\in \Gamma\setminus H$.
For more details we refer the interested reader to page 57 of~\cite{Lehner:automorphic}.
The transformations that identify a pair of boundary sides of this fundamental domain are a set of generators for $\Gamma$.

Let us remark that the fundamental domain such constructed cannot have vertices inside $\HH$, since each such point is necessarily fixed by some elliptic transformation.
We notice also that the construction implies that $p$ is one of the side pairings, and hence a generator for the group.

We call $\cF$ the fundamental domain in $\DD$ obtained 
transporting the fundamental domain we just built from $\HH$ to $\DD$ via $\varphi$ (the inverse of the Caley map).
By construction,  $\cF$ is an ideal polygons and has $1=\varphi(\infty)$ as a vertex, which we will denote $\xi_0$. The images of the vertical lines with real part $\pm \mu/2$ are the to  sides which share $\xi_0$ as a vertex. As in \cref{fig:dandh}, we will label by $\eta$ (resp.\ $\overline{\eta}$) the side such that  $\xi_\eta^l=\xi_0$ (resp.\ $\xi_{\overline{\eta}}^r=\xi_0$), so that $\varphi(\mu/2)=\xi_\eta^l$ (resp.\ $\varphi(-\mu/2)=\xi_{\overline{\eta}}^r$). This shows that Conditions (1)-(3) hold. 

\begin{figure}
\centering
\def\svgscale{0.6}
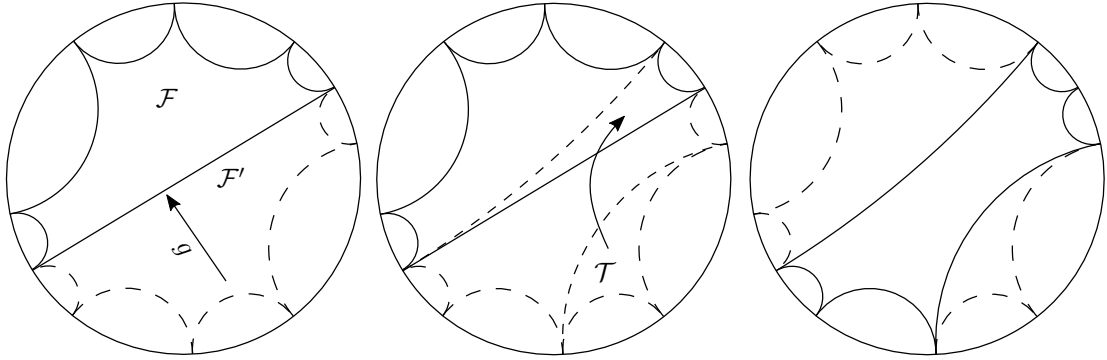
\caption{The surgery of the domain described in \cref{sec:setup}.}
\label{fig:domainsurgery}
\end{figure}

Moreover, since we are assuming that $m=1$, we have that $i$ belongs to the closure of the maximal Margulis neighborhood, which belongs by construction to the fundamental domain. Hence, the origin $0$ of the disk $\DD$ belongs to the closure of $\cF$. In particular, this means that $\arclength{\Arc[\alpha]}\leq \pi$ for every $\alpha \in \mathcal{A}$. Thus, to ensure simultaneously Conditions (4) and (5),   
we just need to ensure that $0$ does not lie on the boundary of the fundamental domain.  This means that $s_\alpha$ is not a straight line in $\DD$ for all $\alpha$ or equivalently that all the inequalities  $\arclength{\Arc[\alpha]}\leq \pi$ inequalities are all strict.

We can always assume that this is the case up to performing the following surgery of the fundamental domain. 
If $0$ lies on the boundary of the fundamental domain, it belongs to a side formed by a diameter, shared by two copies of the fundamental domain that we will call $\cF$ and $\cF'$ (see the left part of \cref{fig:domainsurgery}). 
Consider an ideal triangle $\cT$, contained in $\cF'$, bounded by the side $s$ of $\cF'$ that is paired with the diameter and an adjacent side of $\cF'$, as shown in the middle picture in \cref{fig:domainsurgery}. 
If $g$ is the element of $G$ that pairs $s$ and the diameter, we choose as new fundamental domain $\bigl(\cF'\setminus\cT\bigr)\cup g(\cT)$, as in the right of \cref{fig:domainsurgery}.
By construction, the origin is an internal point of the new fundamental domain, which implies that~\eqref{eq:UpperSetHoleOrderZero} is satisfied.

Let us remark that, after the surgery we just explained, $g$ still identifies two sides of the new fundamental domain, namely the two sides coming from the internal side of $\cT$ and its image, that are dashed in the middle of \cref{fig:domainsurgery} and become solid in the right of the same picture.
However, the generator that was matching the third side of $\cT$ to some other side of $\cF$ is changed.
We need to take care that this side is \emph{not} $s_\eta$ or $s_{\overline{\eta}}$.
This is always possible unless $\cF$ has only $4$ sides, necessarily identified in pairs by parabolic transformations.
In this case, $X$ must be the thrice punctured sphere (see p.~275 of~\cite{Beardon}), which is unique in its isometry class (see, e.g.\ Theorem~9.8.8 of~\cite{Ratcliffe}).
A fundamental domain in $\DD$ for the thrice punctured sphere satisfying all the assumptions of the Lemma is given by the ideal quadrilateral with vertices $\{\pm 1, \pm i\}$, and the two sides that share the point $1$ (resp.\ $-1$) identified.
This completes the proof.
\end{proof}

From now on, $\cF$ will be a fundamental domain for the subgroup $\Gamma < G$ given by the Lemma.  
Let us remark that the fundamental domain $\cF$ is a finite cover of a fundamental domain for $G$ (obtained by \emph{unfolding} the elliptic points) and hence the induced tessellation of the hyperbolic disk  by $\cF$  has tiles which are finite union of copies of a fundamental domain for $G$. 

We will use the tessellation  on $\DD$ induced by the ideal polygon $\cF$ to code geodesics using Bowen-Series coding explained in the previous Section.
Let us stress that we do \emph{not} pass to a \emph{finite cover} of the surface $X$, which is fixed, but only code geodesics in $\DD$ according to a super-tessellation, which is better suited to our purposes than the one corresponding to $G$.
This is similar to what happens in the continued fractions case, where instead of coding the geodesics with respect to the tessellation given by the classical fundamental domain for $\PSL(2,\ZZ)$ one uses the Farey tessellation, made by ideal triangles.

\subsection{Cantor sets and their sums}\label{sec:cantorpreliminaries} 
Since the vertex $\xi_0$ of the fundamental domain $\cF$ chosen in the previous section is such that $\varphi(\xi_0)=\infty$
the partition of $\partial\DD$ given by the arcs $\Arc[\alpha]$ for $\alpha\in\cA$ induces a partition of $\RR$, and not only one of $\RRbar=\partial\HH$. 
Given an infinite word $(\ab{n})_{n\in\NN}$ that satisfies the no-backtracking condition~\eqref{eqnobacktrack}, it will be useful to write
\[
	[\ab{0},\dots,\ab{n},\dots]_{\partial\HH}:=\phi\bigl([\ab{0},\dots,\ab{n},\dots]_{\partial\DD}\bigr).
\]

Now, fix a positive integer $N\geq 2$ and let $\BB_N=\BB_N^\eta \subset\partial\DD$ be the set of points $\xi$ whose boundary expansion $(\ab{k})_{k\in\NN}$ does not contain any cuspidal word of length $N+1$ and whose first letter is different from $\eta$ and $\overline{\eta}$.
One can show that the set $\BB_N$ is a Cantor set: we are going to briefly describe its structure and its gaps in \cref{sec:distortionestimates}. 
Denote with $\KK_N=\phi(\BB_N)$ its image in $\partial\HH$.
We remark that this is a compact set as $\BB_N$ does not contain $\xi_\eta^r$ nor a neighborhood around it and $\phi(\xi_\eta^r)=\infty$.

Let $m_N:=\min\KK_N$, $M_N:=\max\KK_N$ and for  $s\in\NN$ let $\KK_N^s=\KK_N+s\mu$ denote the translates by $z \mapsto z+ \mu$ of the Cantor set $\KK_N$, so
\[
\KK_N^s : = \left[ m_N+s\mu , M_N+s \mu \right].
%	m_N=\min\KK_N \qquad \text{and} \qquad M_N=\max\KK_N.
\]
The next Proposition is the analogue in our set up of Hall's theorem on the sum (difference) of Cantor sets given in terms of continued fractions.

\begin{prop}\label{thm:sumHallFuchsian}
There exists a natural number $N_0$ such that if $N\geq N_0$, for every  integer $s\geq 0$, both
  Cantor sets $\KK^s_N \pm \KK_N$  contain an interval of size at least $\mu$. More precisely,  we have
\begin{align*}
	 \KK_N^s+\KK_N  &= [ 2 m_N+ s\mu , 2M_N+s\mu ], \qquad & \left| \KK_N^s+	\KK_N\right|  &= 2(M_N-m_N) > \mu,   \\
\KK_N^s- \KK_N  &= [-(M_N-m_N)+s\mu , M_N-m_N +s\mu ], \qquad &   \left| \KK_N^s-	\KK_N\right| &= 2(M_N-m_N)> \mu.
\end{align*}
\end{prop}
%\mdoubts{Ho ricontrollato la parte di Luca, non capisco dove usi $s\geq s_0$\dots Ho cambiato il nome dell'etichetta perché era la stessa del Corollario successivo}
We will prove the Proposition in \cref{sec:sum_in_R}. Let us remark that it can be proved as an application of the classical result by Hall on the sum of Cantor sets  (the proof is very similar to the one given in \cite{AMU} for similar Cantor sets).
Since we need in any case to verify that the Cantor sets $\KK_N$ and $\KK_N^s$  satisfy the assumptions of the Stable Hall \cref{thm:stableHall} for the proof of \cref{thm:Hallperturbations}, we prove \cref{thm:sumHallFuchsian} in \cref{sec:sum_in_R} as a special case of the Stable Hall Theorem.

As a Corollary, we have the following result, which is the starting point to build values in the Hall ray.

\begin{cor}\label{cor:decomposition}
For any 
$L\geq \mu/2$, there exist two real numbers $x_1, x_2 $ and an integer $s\geq 1$ such that $x_1,x_2\in\KK_N$ and
\[
	L=s\mu + x_2-x_1 .
\]
\end{cor}
\begin{proof}
Remark that $ \KK_N^{s+1}-	\KK_N =\left( \KK_N^s-	\KK_N \right) +\mu $. Thus, since by  \cref{thm:sumHallFuchsian}, the length of each $\KK_N^s-	\KK_N$ is greater than $\mu$, the intervals $\KK_N^s-	\KK_N$, $s \in \NN$, overlap and hence
\[  
	\bigcup_{s\geq 1} (\KK_N^s-	\KK_N ) = [ -(M_N-m_N)+\mu, + \infty)  \supset \Bigl[\frac{\mu}{2} , +\infty\Bigr), 
	\]
where the last inclusion follows since  $M_N -m_N \geq \mu/2$ (also by  \cref{thm:sumHallFuchsian}). In particular
 for any $L\geq \mu/2$ there exists an integer $s \geq 1$ such that $L \in  \KK_N^s-	\KK_N$. Since $\KK_N^s = \KK_N+s \mu$, this means that there exist $x_1,x_2 \in \KK_N$ such that $L= (x_2+s\mu)-x_1 $ as desired. 
\end{proof}

\subsection{A generalized Perron formula via boundary expansions}\label{sec:Perron}
The starting point for the proof of existence of a Hall ray is a generalization of Perron's formula~\eqref{Perron_formula_CF} for values of the Lagrange spectrum, in which classical continued fractions are replaced by the Bowen-Series boundary expansions with respect to the finite index subgroup $\Gamma < G $ defined in \cref{sec:setup}. 

Let us remark that given a geodesic $\gamma=\gamma(x,y)$ whose cutting sequence with respect to the tessellation defined in \cref{sec:setup} is $(\ab{n})_{n\in\ZZ}$, the two endpoints $x$ and $y$ of $\gamma$ are given by
\[
	y=[\ab{0},\dots,\ab{n},\dots]_{\partial\HH}
		\qquad
		\text{and}
		\qquad
	x=[\overline{\ab{-1}},\dots,\overline{\ab{-n}},\dots]_{\partial\HH}.
\]
The bars in the expression for $x$ are due to the fact that moving from $\cF$ to $\RR$ towards $x$ we are traveling backwards along $\gamma$, as explained at the very end of \cref{subsection:cuttsequences}.

Hence, introduce the following notation:
\begin{equation}\label{def:backwardCF}
	[\ab{0},\dots,\ab{n},\dots]_{\partial\HH}^- = [\overline{\ab{0}},\dots,\overline{\ab{n}},\dots]_{\partial\HH}.
\end{equation}
Let $\height_G(\gamma)$ denote the essential height of the geodesic $\gamma$, see \cref{sec:Lagrange_defs}.

\begin{lemma}[Perron's formula for the essential height]\label{lemma:Perron}
Let $\gamma$ be a complete geodesic with $\gamma(0)\in\cF$ and cutting sequence $(\ab{n})_{n\in\ZZ}$.
Suppose that $\height_G(\gamma)> 1$.
Then 
\begin{equation}\label{eq:heightlimsup}
	\height_G(\gamma) = \frac{1}{2}\limsup_{n\in\ZZ}{ |[\ab{n},\ab{n+1},\dots]_{\partial\HH} - [\ab{n-1},\ab{n-2},\dots]_{\partial\HH}^-|}.
\end{equation}
\end{lemma}

We recall that are here assuming that $m=1$, where $m$ is the height of the maximal Margulis neighborhood. More in general, the formula \cref{eq:heightlimsup} holds for $\height_G(\gamma)> m $. 
In the proof of \cref{lemma:Perron}, given below (and in the rest of the paper) we will use the following observation, which follows from the definition of the Bowen-Series coding and boundary expansions.
\begin{lemma}\label{gammaj}
 For any non-zero integer $j$, let $\gamma_j$ be the geodesic defined by
\[
\gamma_j (t):= \begin{cases}  
					g_{\ab{j-1}}^{-1} g_{\ab{j-2}}^{-1} 
					\cdots   g_{\ab{0}}^{-1} \cdot \gamma (t),
						 & \text{ if } j\geq 1; \\
					g_{\ab{j}}^{-1} g_{\ab{j-1}}^{-1} \cdots 
					g_{\ab{-1}}^{-1} \cdot \gamma (t), 
						& \text{ if } j< 0. 
				\end{cases}
\]
The cutting sequence of $\gamma_j$ is $(a_{n+j})_{n \in \ZZ}$ and the endpoints of $\gamma_j$ are
\[
y_j:=[\ab{j},\ab{j+1},\dots]_{\partial\HH}, \qquad x_j:=[\ab{j-1},\ab{j-2}, \dots ]^-_{\partial\HH}.
\]
\end{lemma}

The geodesic $\gamma_j$ has the property that $\gamma_j(t) \in \cF$ for $t_j< t <t_{j+1}$ where $t_j$ is the $j^{\text{th}}$-crossing with the coding tessellation (compare with \cref{crossing}). We will call it the $j^{\text{th}}$ normalized geodesic.

\begin{proof}
For $j\geq 0$, the statements follow from the definitions of coding and boundary expansion, see \cref{thm:BowenSeries} and also \cref{crossing}.
 When $j<0$, consider the geodesic $\gamma'(t):= \gamma(-t)$, whose cutting sequence $(a'_n)_{n\in\NN}$ is hence given by $a_n'=\overline{a_{n-1}}$, and apply the previous case.
\end{proof}

\begin{proof}[Proof of \cref{lemma:Perron}] 
We first claim  that, if the essential height~\eqref{eq:essentialheight} of a geodesic $\gamma\colon\RR\to X$ is larger than $1$, it is sufficient to consider the elements $g\in G$ such that $g\cdot\gamma\cap\phi(\cF)\neq\emptyset$.
In fact, consider an arbitrary element $g\in G$ and let $\cU_1$ be the fundamental horodisk defined in~\eqref{eq:fundamentalhorodisk}.
If $g\cdot\gamma\cap\cU_1 = \emptyset$, then we have that $\height(g\cdot\gamma)\leq 1$. 
Otherwise, since the fundamental domain $\phi(\cF) \subset \HH$, by construction (see \cref{costruzioneF} and recall that $m=1$), contains a Euclidean rectangle delimited by $\Im z = 1$ and two vertical lines at $-\frac{\mu}{2}$ and $\frac{\mu}{2}$,  there exists an integer $k$ such that $(p^k \cdot g \cdot\gamma)\cap\phi(\cF)\neq\emptyset$.
Clearly $p^k\cdot g$ and $g$ are equivalent modulo infinity. 
Moreover we have $\height(p^k g \cdot\gamma)=\height(g\cdot\gamma)\geq 1$. This proves the claim. 

Let us now remark that the elements $g\in G$ that satisfy $g\cdot\gamma\cap\phi(\cF)\neq\emptyset$, i.e., bring back a piece of the geodesic $\gamma$ to the fundamental domain, can be exactly obtained using the cutting sequence of $\gamma$, i.e.\ are exactly the elements of the form  $g_{\ab{k}}^{-1} \cdots g_{\ab{0}}^{-1}$ for $k\geq 0$ and $g_{\overline{\ab{k}}}^{-1} \cdots g_{\overline{\ab{-1}}}^{-1}$ for $k<0$. Thus, by \cref{gammaj} and the definition of hyperbolic naive height~\eqref{eq:height}, we get~\eqref{eq:heightlimsup}. 
%Moreover, \cref{thm:BowenSeries} states that, for $k\geq 0$, $g_{\ab{k}}^{-1} \cdots g_{\ab{0}}^{-1}$ maps the geodesic $\gamma(x,y)$ to a geodesic whose cutting sequence is exactly the shifted sequence $(\ab{n+k})_{n\in\ZZ}$.
%If $k<0$, the element  $g_{\overline{\ab{k}}}^{-1} \cdots g_{\overline{\ab{-1}}}^{-1}$, acts in the same way.
%Thus, recalling  
\end{proof}

Finally, let us record as a Lemma some simple observations, which follows  from the choice and geometry of the fundamental domain (we refer the reader to \cref{fig:dandh,fig:estimatesHall}).

\begin{lemma}\label{endpointscontrol} 
Let $\gamma(x,y)$ be a geodesic with initial endpoint $x$ and final endpoint $y$, whose cutting sequence, assuming that $\gamma(0)\in\cF$, is $(a_n)_{n \in \ZZ}$. Then:
\begin{enumerate}
\item if $y>x$, then if $a_{0} \neq {\eta}$ we have $y<\mu/2$ and if  $a_{-1} \neq {\eta}$ then $x>-\mu/2$;
\item if $x>y$, then if $a_{0} \neq \overline{\eta}$ we have $y>-\mu/2$ and if  $a_{-1} \neq \overline{\eta}$ then $x<\mu/2$;
\item combining (1) and (2), if both $a_{0}$ and $a_{-1}$ do not belong to $\{\eta, \overline{\eta}\}$, we have that
\[
\height(\gamma)= \frac{|x-y|}{2} \leq \frac{\mu}{2}.
\]
\end{enumerate}
\end{lemma}
\begin{proof}
For  Part (1) (resp.\ Part (2)), simply recall that the fundamental domain $\varphi(\cF) \subset \HH$ is bounded by the two vertical lines $\{ \Re z = - \mu/2\}$ and  $\{ \Re z = \mu/2\}$, whose \emph{external} labels are $\overline{\eta}$ and $\eta$ respectively (see \cref{fig:dandh}). Thus, for a geodesic with $\gamma(0) \in \varphi(\cF)$  to cross the side labeled by $\eta$ (resp.\ $\overline{\eta}$), so that $a_0=\eta$ (resp.\ $a_0=\overline{\eta}$) the final endpoint has to be greater than $\mu/2$ (resp.\ less than $-\mu/2$).  
The arguments for $a_{-1}$ are analogous, just reversing time and thus exchanging the role of the endpoints.
Finally, Part (3) follows simply by combining (1) and (2). 
\end{proof}

\subsection{Hall's argument for the height in any zonal Fuchsian group}\label{subsec:Hall_argument}
We now have all the elements to conclude the proof of  \cref{thm:Hall} following the scheme of Hall's original proof. 

\begin{proof}[Proof of \cref{thm:Hall}]
Let $N_0$ be given by \cref{thm:sumHallFuchsian}. We will show that
\begin{equation}\label{choicer}
[L_0, + \infty] \subset \cL(X,\infty), \qquad \text{ for any }  L_0> (N_0+1)\mu.
\end{equation}

\smallskip
\emph{Step one: construction of the bi-infinite word.}

By  \cref{cor:decomposition} (remark that in particular $L \geq \mu/2$ so we can apply it),  there exist
$x_1, x_2  \in \KK_N$ and $s\geq 1$ such that $L=s\mu + x_2-x_1$.  In particular, write $y=[\ab{0},\ab{1},\dots ,\ab{n},\dots]_{\partial\HH}$ and $x=[\bb{0},\bb{1},\dots,\bb{n}, \dots ]_{\partial\HH}$, with both sequences $(\ab{n})_{n\in\NN}$ and $(\bb{n})_{n\in\NN}$ in $\KK_N$. Thus, 
\begin{equation}\label{formL} L=s\mu +[\ab{0},\ab{1},\dots ,\ab{n},\dots]_{\partial\HH} -[\bb{0},\bb{1},\dots,\bb{n}, \dots ]_{\partial\HH}.\end{equation} 
Let us now construct an infinite word $(\cc{n})_{n \in \ZZ}$ that will give the cutting sequence of a geodesic $\gamma$ such that
\[
	L_G(\gamma)=2\height_G(\gamma)=L.
\]
We will define blocks of entries $W_j$, $j \in \ZZ$, which we will then concatenate to form the word $(\cc{n})_{n \in \ZZ}$.
Recall that $\eta$ is the letter such that $p=g_\eta$. % that corresponds to the parabolic transformation $p$.
Set
\[
	W_{j} = {\overline{\bb{|j|}}} \dots {\overline{\bb{0}}} \eta^s \ab{0} \dots \ab{|j|} , \qquad j \in \ZZ,
\]
where $\eta^s$ means that the letter $\eta$ is repeated $s$ times.
We remark that, by definition of the Cantor set $\KK_N$, we have that $\ab{0} \neq \overline{\eta}$ and $\overline{\bb{0}}\neq \overline{\eta}$ (since $\bb{0}\neq {\eta}$). Thus $W_j$  satisfies the no-backtracking condition~\eqref{eqnobacktrack}.
Let us choose letters to interpolate between $W_j$ (which ends in $a_j$) and $W_{j+1}$ (which starts with $\overline{b_{j+1}}$) as follows.
Since the alphabet $\cA$ has cardinality $2d>3$, we can pick  $\delta_j$ such that $\delta_j \neq \overline{\ab{j}}$ and $\ab{j} \delta_j$ is not a cuspidal word and then  $\delta'_j$ such that $\delta'_j \neq \overline{\delta_{j}}$, $\delta'_j \neq \bb{j+1}$ and $\delta'_j \overline{\bb{j+1}}$ is not a cuspidal word.
 Thus, the word
\begin{equation}\label{sequencenoN}
	\ab{0} \dots \ab{j} \delta_j \delta_j' \overline{\bb{j+1}} \dots \overline{\bb{0}}
\end{equation}
satisfies the no-backtracking condition~\eqref{eqnobacktrack}.
Moreover, as the two infinite words $(\ab{n})_{n\in\NN}$ and $({\bb{n}})_{n\in\NN}$ are in $\BB_N$, and thanks to our choice of $\delta_j, \delta_j'$, the word in~\eqref{sequencenoN} does not contain any parabolic word of length bigger than $N$.
It follows that the infinite word $(\cc{n})_{n \in \ZZ} $ obtained juxtaposing the blocks $W_j, \delta_j, \delta_j'$ in increasing order of $j \in \ZZ$ % , i.e. of the form
%\[ \cdots W_{j-1} \delta_j, \delta_j' W_j W_{j+1} \cdots \]
 satisfies the no-backtracking condition or, in other words, is actually the cutting sequence of some geodesic $\gamma$.

\smallskip
In the next two steps we will show that $\height_G(\gamma)=L/2$ and hence $L_G(\gamma)=2\height_G(\gamma)=L$.  
First, in \emph{Step two}, we will check that if we evaluate the $\limsup$ in~\eqref{eq:heightlimsup} along the subsequence of times where we see the parabolic word $\eta^s$ we obtain the desired value. Then, in \emph{Step three}, we will show that this subsequence actually realizes the $\limsup$.
\smallskip

\emph{Step two: the infinite word realizes the desired Lagrange value along a subsequence of times.}

We remark that $\eta^s$ is a cuspidal word and, since $\ab{0}$ and $\overline{\bb{0}}$ are different from $\eta$, it is actually a maximal cuspidal word. Let $r_k$ be the subsequence of times where an occurrence of the central word $\eta^s$ in $W_k$ begins. Thus,
\[ [ \cc{r_k},\cc{r_k+1}, \cdots ]_{\partial\HH} = [ \eta, \eta ,\dots ,\eta ,\cc{r_k+s}, \cc{r_k+s+1},  \cdots ]_{\partial\HH}= [\eta, \eta ,\dots ,\eta , \ab{0},\ab{1},\dots]_{\partial\HH} .\]
 By \cref{gammaj} (or directly by \cref{thm:BowenSeries}), this endpoint is obtained acting by $g_\eta^s$ on 
 \[ 
 	[\cc{r_k+s}\cc{r_k+s+1}\dots ]_{\partial\HH} = [\ab{0},\ab{1},\dots]_{\partial\HH} .\]
 Since $g_\eta^s=p^s$ (recall \cref{costruzioneF}), it acts on $ \HH$ as the translation $z\mapsto z+s \mu$. Thus,    
% the endpoints of the geodesic $g_\eta^s \gamma$. the naive height of $g_\eta^s \gamma$ is the same than $ \gamma$.
evaluating Perron's formula~\eqref{eq:heightlimsup} along the subsequence $r_k$, as $|k| \to \infty$, we get 
\begin{equation}\label{limsupseq}
\begin{split}
	\lim_{|k| \to \infty } \frac{\bigl| [\cc{r_k},\cc{r_k+1},\dots]_{\partial\HH} - [\cc{r_k-1},\cc{r_k-2},\dots]_{\partial\HH}^-\bigr|}{2} 
	&=\lim_{|k| \to \infty } \frac{\bigl| [\eta, \dots,\eta,\ab{0},\dots]_{\partial\HH} - [\overline{\bb{0}},\overline{\bb{1}},\dots]_{\partial\HH}^-\bigr|}{2} \\
	&=\lim_{|k| \to \infty } \frac{(s\mu +[\ab{0},\ab{1}, \dots]_{\partial\HH} - [\bb{0},\bb{1}, \dots ]_{\partial\HH})}{2} \\
	&= \frac{L}{2},
\end{split}
\end{equation}
where in the last line we used also the definition~\eqref{def:backwardCF} of $[\,\cdot\,]_{\partial\HH}^-$,the form of the words $W_k$ and~\eqref{formL}.
%\begin{equation}\label{limsupseq2}
%\begin{split}
%&
%\end{split}
%\end{equation}
%Thus, combining~\eqref{limsup} and~\eqref{limsup2}, we get $L_G(\gamma)\geq L$. 

\begin{figure}
\centering
\def\svgscale{0.8}
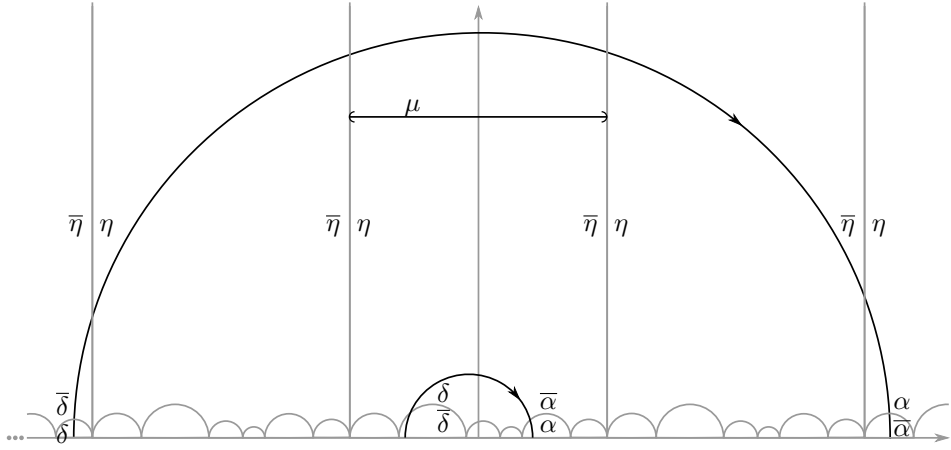
\caption{In grey, the fundamental domain $\phi(\cF)$, and a portion of the tessellation induced by it, with some of the side labels; in black, a geodesic with larger height, coded by $\dots,\overline{\delta},\eta,\eta,\eta,\eta,\overline{\alpha},\dots$, and one with smaller height, coded by $\dots,\delta,\alpha,\dots$.}
\label{fig:estimatesHall}
\end{figure}

\smallskip
\emph{Step three: estimates on the remaining times.} 
We now estimate the value of the $\limsup$ in the formula~\eqref{eq:heightlimsup} for the other times. For any $j \in \mathbb{Z}$, let $\gamma_j$ be the $j^{\text{th}}$ renormalized geodesic defined in  \cref{gammaj}, %i.e.  the  $\gamma_j = g_{\cc{j}}^{-1}\dots g_{\cc{0}}^{-1}\cdot \gamma$ is the 
which is a geodesic coded by $(\cc{n+j})_{n\in\ZZ}$  with endpoints $\tilde{x}_j:=[\cc{j},\cc{j+1},\dots]_{\partial\HH}$ and $\tilde{y}_j:=[\cc{j-1},\cc{j-2},\dots]_{\partial\HH}^-$ (see \cref{gammaj}). 

We begin with the simple remark that if we see a block of $k$ consecutive $\eta$'s, i.e.\ $\cc{n}=\cdots = \cc{n+k-1}=\eta$, then the naive height value $|\tilde{x}_j-\tilde{y}_j|/2$ of  $\gamma_j$ remains constant for $n\leq j \leq n+k$.  In fact, by \cref{gammaj}, $\gamma_{j}$ for $n< j \leq n+k$ is obtained from $\gamma_n$ applying a power of $g_\eta^{-1}=p^{-1}$, i.e.\  rigidly translating to the left by $\mu $ the two endpoints of $\gamma_n$.  In other words, for every occurrence $\cc{n}=\eta$ of $\eta$ (or similarly for $\overline{\eta}$), we have 
\[
	|[\eta,\cc{n+1},\dots]_{\partial\HH} - [\cc{n-1},\dots]_{\partial\HH}^-| = |[\cc{n+1},\dots]_{\partial\HH} - [\eta,\cc{n-1},\dots]_{\partial\HH}^-|.
%	\dots = |[\cc{n+k},\dots]_{\partial\HH} - [\eta,\dots,\eta, \cc{n-1},\dots]_{\partial\HH}^-|.
\] 
The same remark also holds for a sequence of consecutive $\overline{\eta}$, in this case we act by  $g_{\overline{\eta}}^{-1}=g_\eta=p$  and hence we are translating the endpoints to the right by $\mu$.

In particular, by \emph{Step one}, this gives that, for any $k \in \ZZ$, and any $r_k\leq j \leq r_k+s$, $\height(\gamma_j)= \height(\gamma_{r_k})$ and the argument in Perron's formula~\eqref{eq:heightlimsup}  is constant and equal to $L/2$. 
%By the previous remark, the first time we have to estimate is $r_k+s+1$, for $k\in\ZZ$.
We will now evaluate the argument of  Perron's formula~\eqref{eq:heightlimsup} for any $j$ which is not of this form.  
We will consider four sub-cases.

\smallskip
\noindent \emph{Case i}: If both $\cc{j}$ and $\cc{j-1}$ do not belong to $\{\eta,\overline{\eta}\}$, by \cref{endpointscontrol} 
% we see that the geodesic $\gamma'$ does not cross any vertical line of the form $k\mu/2$ for $k\in\ZZ$, as this would imply that at least one of $\cc{j}$ and $\cc{j-1}$ to be equal to either $\eta$ or $\overline{\eta}$, see \cref{fig:estimatesHall}.
 both the endpoints of $\gamma_j$ lie inside the interval $[-\frac{\mu}{2},\frac{\mu}{2}]$, thus $2\height(\gamma_j)\leq\mu$.

\smallskip
\noindent \emph{Case ii}: Suppose now that $\cc{j}\in \{\eta, \overline{\eta}\}$, but $\cc{j-1}\notin \{ \eta,\overline{\eta}\}$.
By our assumption on $j$, and the structure of the bi-infinite word ($\cc{n})_{n\in\ZZ}$, we can have at most $N$ consecutive $\eta$ or $\overline{\eta}$, beginning with $\cc{j}$. 
This implies that the geodesic $\gamma_j$ crosses at most $N$ vertical lines of the form $k\mu/2$ for $k\in\NN$, see \cref{fig:estimatesHall}. Let $N_j\leq N$ be the number of lines actually crossed and assume $\cc{j}=\eta$ (if $\cc{j}=\overline{\eta}$ the argument is analogous).  
So we have
\[
\begin{split}
2	\height(\gamma_j) &= |[\cc{j},\cc{j+1},\dots]_{\partial\HH} - [\cc{j-1},\cc{j-2},\dots]_{\partial\HH}^-| \\
	&= |[\eta,\dots,\eta,\cc{n+N_j},\dots]_{\partial\HH} - [\cc{j-1},\cc{j-2},\dots]_{\partial\HH}^-|\\
	&\leq  N_j \mu + |[\cc{n+N_j},\cc{n+N_j+1},\dots]_{\partial\HH} - [\cc{j-1},\cc{j-2},\dots]_{\partial\HH}^-|\\
	&< N_j \mu + \Bigl|\frac{\mu}{2}- \Bigl(-\frac{\mu}{2}\Bigr)\Bigr| \leq N \mu + \mu \leq  L.
\end{split}
\]
\smallskip
\noindent \emph{Case iii}: is symmetric to the previous one. If $\cc{j-1}\in \{\eta, \overline{\eta}\}$, but $\cc{j}\notin \{ \eta,\overline{\eta}\}$, we can repeat the previous argument in the past, i.e.\ we have that, if $\cc{j-1}=\eta$, 
\[
\begin{split}
2	\height(\gamma_j) &= |[\cc{j},\cc{j+1},\dots]_{\partial\HH} - [\eta, \dots, \eta, \cc{j-N_j-1},\cc{j-N_j-2},\dots]_{\partial\HH}^-| \\
	&\leq  N_j \mu + |[\cc{j},\cc{j+1},\dots]_{\partial\HH} - [\cc{j-N_j-1},\cc{j-N_j-2},\dots]_{\partial\HH}^-|\\
	&< N_j \mu + \Bigl|\frac{\mu}{2}- \Bigl(-\frac{\mu}{2}\Bigr)\Bigr| \leq N \mu + \mu \leq  L
\end{split}
\]
and an analogous estimate holds for $\cc{j-1}=\overline{\eta}$.
 
\smallskip
\noindent \emph{Case iv}:
Finally, if both $\cc{j}, \cc{j-1} \in \{ \eta, \overline{\eta}\}$, since by the no backtracking condition $\cc{j}\neq \overline{\cc{j-1}}$,  either $\cc{j}=\cc{j-1}=\eta$ or $\cc{j}=\cc{j-1}=\overline{\eta}$. We claim that in this case we can reduce this case to one of the previous Steps using the remark at the beginning of this step that the naive height of $\gamma_j$ does not change during a   block of consecutive $\eta$ or $\overline{\eta}$. More precisely, 
we look for the first $m\leq j$ such that $\cc{m-1}\notin \{\eta,\overline{\eta}\}$, that is we choose the time $m$ where the cuspidal word contain $\cc{j}$ begins.  If $m=r_k$ for some $k$, $\height(\gamma_j)= \height(\gamma_{r_k})=L/2$ by \emph{Step two}. 
Otherwise, the cuspidal word beginning at $\cc{m}$ must be at most of length $N$, by construction of the word $(\cc{n})_{n\in\ZZ}$.
In this case, we can use the above estimates for the time $j=m$ to see that $\height(\gamma_j)= \height(\gamma_{m})<L/2$.

\smallskip 
Thus, combining \emph{Step two} and \emph{Step three}, Perron's formula shows that $\height_G(\gamma) =  L/2$ and hence $L_G(\gamma) = 2 \height_G(\gamma)=L$. This concludes the proof.
\end{proof}

\section{Cantor sets in the boundary}\label{sec:Cantor}
In this section we describe the Cantor set $\mathbb{B}_N$ in the boundary of the disk which correspond to endpoints of geodesics whose excursions in the cusps are \emph{bounded}, in the sense that their boundary expansion contains only cuspidal words  of length bounded by $N$. We first describe their gaps combinatorially, through the symbolic sequences which correspond to them (see \cref{SectionCombinatorialDescriptionGaps}), then  prove some distortion estimates (see \cref{sec:distortionestimates}) which will be needed to apply the Stable Hall \cref{thm:stableHall}. 

\smallskip
Let us first recall the definition of the Cantor set we want to study through the cuspidal acceleration of the boundary expansion.  % $\mathbb{B}_N\subset \partial \mathbb{D}$.   
Consider an infinite word $(\ab{n})_{n\in\NN}$ satisfying the no-backtracking condition~\eqref{eqnobacktrack} and that is not eventually cuspidal. The cuspidal acceleration described in \cref{sec:cuspidalwords} provides a sequence of integers $0=:n(0)<n(1)<n(2)<\dots$ and maximal cuspidal words $C_r:=\ab{n(r)},\dots,\ab{n(r+1)-1}$ with $r\in\NN$, such that for any $r\geq 1$, $n(r+1)$ is the minimal $n>n(r)$ such that $\ab{n(r)}\dots\ab{n(r+1)}$ is \emph{not} a cuspidal word. Let us write $\wordlength{\ab{0}\dots\ab{n-1}}=n$ for the length of a word, so that the length $\wordlength{C_r}$ of the $r^{\text{th}}$ cuspidal word is 
\[
	\wordlength{\ab{n(r)},\dots,\ab{n(r+1)-1}}=n(r+1)-n(r).
\] 

Fix a positive integer $N$ and a letter $\eta\in\cA$ and let $\mathbb{B}_N:=\mathbb{B}(N,\eta) \subset \partial \mathbb{D}$ be the  Cantor set defined in \cref{sec:cantorpreliminaries}, %i.e.
%$$
%\mathbb{B}_N:=\mathbb{B}(N,\eta) = \{ \xi\in\partial\DD, \quad \text{s.t. \} \xi=[\ab{0}, \dots, \ab{n}, \dots]_\partial \quad \text{} 
%$$
which consists of  the  set of points  whose boundary expansion $(\ab{n})_{n\in\NN}$ is such that $\ab{0}\not=\eta, \overline{\eta}$ and the sequence $(\ab{n})_{n\in\NN}$ does not contain any cuspidal word of length $N+1$, that is for any $r\in\NN$ we have 
\begin{equation}
\label{EqBoundedTypeCuspidal}
\wordlength{C_r}=n(r+1)-n(r)\leq N.
\end{equation}

%Remark first of all that $\BB_N$
One can prove that $\BB_N$ is indeed a Cantor set. The proof is given in \S~7.2 of~\cite{AMU} for some analogous Cantor sets, so we refer the interested reader to it. In the next section, though, we recall  a combinatorial description of the gaps of the Cantor set $\BB_N$. 

\subsection{Combinatorial description of the gaps in the Cantor set}
\label{SectionCombinatorialDescriptionGaps}
The union of the gaps of the Cantor set $\BB_N$ can be described using \emph{deleted arcs} and the corresponding \emph{forbidden words} as follows.
If $\xi \in \BB_N$, by definition no cuspidal word of length $N+1$ can appear in its boundary expansion. An $N$-\emph{forbidden word} of level $m$, which for short we will call a $(N,m)$-\emph{forbidden word}, is a finite word which contains a cuspidal word of length $N+1$ after the $m^{\text{th}}$ letter, i.e.\  
  a word $\ab{0}\dots\ab{m+N}$ of length $m+N+1$ whose cuspidal decomposition 
\[
	\ab{0}\dots\ab{m+N}=C_1\dots C_r
\] 
is such that the first $r-1$ terms $C_1,\dots,C_{r-1}$ are cuspidal words of length strictly smaller than $N+1$ (so that the condition~\eqref{EqBoundedTypeCuspidal} is satisfied), while for the last term we have $\wordlength{C_r}=N+1$, that is 
\[
C_r=\ab{m}\dots\ab{m+N}.
\]
Arcs $\Arc[\ab{0},\dots,\ab{m+N}]$ corresponding to $(N,m)$-forbidden words with $m\in\NN$ are called $(N,m)$-\emph{deleted arcs}.
Let $ \cD_N$ be the family whose elements are all the $(N,m)$-deleted arcs for $m\in\NN$. Elements $\Arc\in\cD_N$ have mutually disjoint interior and are exactly all the arcs which are removed from $\partial\DD$ to obtain $\BB_N$. 
Any gap $B$ of $\KK_N$ is the countable union $B=\cup_{k\in\ZZ}\Arc_k$ of a collection of \emph{adjacent} deleted arcs in $\cD_N$, where by adjacent we mean that $\max \Arc_k=\min \Arc_{k+1}$ for any $z\in\ZZ$; the length $\arclength{\Arc_k}$ shrinks exponentially as $|k|\to\infty$. 
An explicit description of all arcs $\Arc_k$ fitting together in the same gap is given in \S~7.2 in~\cite{AMU}. The explicit description is not needed in this paper: % to estimate their side in ccording to \cref{SectionSizeGapsCantorBoundary} of this paper,
we just need to know that if we set $\cG_N\subset\partial\DD$ to be the union $\cG_N:=\bigcup_{\Arc\in\cD}\Arc$ of all deleted arcs $\Arc\in\cD_N$, the set $\cG_N$, described as a countable union of closed arcs, is an open set and it coincides precisely with the union of all gaps of $\BB_N$, in other words $$\BB_N=\big(\partial\DD\setminus(\Arc[\overline{\eta}]{\cup\Arc[\eta]})\big)\setminus \cG_N .$$ 

Let us now describe how to hierarchically produce all the gaps of the Cantor set $\BB_N$ through the generators of the boundary expansions. This also gives an ordering of the gaps and the corresponding intervals into \emph{levels}. %Recall that we denote by $\xi_i $, $1\leq i \leq 2d$ the ideal vertexes of $\cF$ in clockwise order. For $0\leq 2d-1 \leq 2d$, let $\alpha_i \in \mathcal{A}$ be such that  $s_{\alpha_i}$  is  the side of $\cF$ which has $\xi_{i-1}$ and $\xi_i$ as endpoints.

The \emph{gaps of level zero} are in one to one correspondence with the $2d-1$ ideal vertices $\xi_1, \dots, \xi_{2d-1}$ of the fundamental domain $\cF$. Let $\xi\in \partial\DD$ be any such vertex of $\cF$. Let us denote by $\alpha_\xi^l$ and $\alpha_\xi^r$ the two letters in $\cA$ such that the arcs $\Arc[\alpha_\xi^l]$ and $\Arc[\alpha_\xi^r]$ share $\xi$ as an endpoint. As the notation suggest, we assume that $\Arc[\alpha_\xi^r]$ has $\xi$ as \emph{right} endpoint, while $\Arc[\alpha_\xi^l]$ has $\xi$ as \emph{left} endpoint. Thus, in the clockwise ordering,
$$
\max \Arc[\alpha_\xi^r]= \xi = \min \Arc[\alpha_\xi^l].
$$
%For any $1\leq i \leq 2d$,   consider the two sides $s_{\alpha_{i}}$ and $s_{\alpha_{i+1}}$ which share $\xi_i$ as a common endpoint, so that $\xi^r_\alpha= \xi = \xi^r_\beta$. 
%Fix two letters $\alpha$ and $\beta$ in $\cA$ such that the sides $s_\alpha$ and $s_\beta$ share a common endpoint 
As we observed after \cref{defnestedarcs}, there are exactly two $(N,0)$-forbidden words $\alpha_0^l\dots \alpha_N^l$ and $\beta_0^r\dots \beta_N^r$ with respectively $\alpha_0^l=\alpha_\xi^l$ and $\beta_0^r=\alpha_\xi^r$. 
Moreover $\alpha_0^l\dots \alpha_N^l$ is left cuspidal and $\beta_0^r\dots \beta_N^r$ is right cuspidal.
The two corresponding $(N,0)$-deleted arcs $\Arc[\alpha_0^l,\dots,\alpha_N^l]$ and $\Arc[\beta_0^r,\dots,\beta_N^r]$ 
share a common endpoint, indeed we have
$
	\max \Arc[\alpha_0^r,\dots,\alpha_N^r]=
	 \xi =
	\min \Arc[\beta_0^l,\dots,\beta_N^l].
$
%  $B[\alpha,\beta]$ corresponding to the pair $\alpha,\beta$ is 

The gap $B[\xi]$ of level zero corresponding to some $\xi$ is by definition the connected component of $\cG$ which contains $\Arc[\alpha_0^l,\dots,\alpha_N^l]\cup \Arc[\beta_0^r,\dots,\beta_N^r]$. Thus in particular we have 
\[	 \Arc[\alpha_0^l,\dots,\alpha_{N}^l]\cup \Arc[\beta_0^r,\dots,\beta_{N}^r] \subset B[\xi] .\]
Keeping in mind the geometry of cuspidal arcs, one can also show  that 
\begin{equation}\label{EquationUpperSetHoleOrderZero}
	B[\xi] \subset \Arc[\alpha_0^l,\dots,\alpha_{N-1}^l]\cup \Arc[\beta_0^r,\dots,\beta_{N-1}^r].
\end{equation}
The level zero gaps are $B[\xi_i]$ for $1\leq i\leq 2d-1$. For any $n\geq 1$, to define the \emph{gaps of level} $n$, we transport the gaps of level zero through the generators as follows.  Let $B[\xi]$ be a gap  of level zero and $\ab{0}\dots\ab{n-1}$ an admissible word   with $\ab{0}\neq\overline{\eta}, \eta$ and such that both the two words 
\[
	\ab{0}\dots\ab{n-1}\alpha_0^l\dots \alpha^l_N
		\qquad
		\text{and}
		\qquad
	\ab{0}\dots\ab{n-1}\beta^l_0\dots \beta^l_N
\]
are admissible and moreover form $(N,n)$-forbidden words, where  $\alpha^l_0=\alpha_\xi^l$ and $\beta^r_0=\alpha_\xi^r$. 
The corresponding gap of level $n$ is the open interval
\[
	B[\ab{0},\dots,\ab{n-1};\xi]=
g_{\ab{0}}\circ\dots\circ g_{\ab{n-1}}\left(B[\xi]\right).
\]

Correspondingly,  it is also convenient to introduce the notion of \emph{intervals of level} $n$. To do so, for any  $n\in\NN$ and any letter $\alpha\in\cA$,  recall that  the ideal vertices of the arc $\Arc[\alpha]$ are $\xi_\alpha^l$ and $\xi_\alpha^r$.   
%and let $\alpha^l$ and $\alpha^r$ be the letters such that 
%$\xi^r_{\alpha^l}=\xi^r_\alpha$  and $
%\xi^l_{\alpha^r}=\xi^l_\alpha $. 
%Then 
Define the compact arc $K[\alpha]$ as the unique connected component of 
$
\partial\DD\setminus\bigcup_{0\leq i\leq 2d-1} B[\xi_i]
$ 
that shares an endpoint both with $B[\xi_\alpha^l]$ and $B[\xi_\alpha^r]$. This defines the $2d$ \emph{intervals of level zero} $K[\alpha]$, $\alpha\in\cA$.

To define the intervals of other levels, % generated by  a given $K_\alpha$, 
for any admissible word $\ab{0},\dots,\ab{n-1}$ of length $n$, define the intervals of level $n$ compatible with 
$
\ab{0},\dots,\ab{n-1}
$ 
as the intervals 
\[
	K[\ab{0},\dots,\ab{n-1};\alpha]:=
g_{\ab{0}}\circ\dots\circ g_{\ab{n-1}}\left( K[\alpha]\right),
\]
where $\alpha$ ranges among all the letters with $\alpha\not=\overline{\ab{n-1}}$.

\subsection{Distortion estimates}
\label{sec:distortionestimates}

In the following we consider the Poincar\'e disc $\DD$ as an open subset of the Riemann sphere $\overline{\CC}=\CC\cup\{\infty\}$; $| \cdot |$ will denote the usual absolute value in $\CC$ and by open disc we mean a set of the form $\{ z \in \CC : |z-z_0|<r\}$ for some $z_0\in \CC$ and radius $r>0$. 

For any letter $\alpha\in\cA$, recall that $s_\alpha$ is the side of $\cF$ which correspond to the arc $\Arc[\alpha]\subset \partial \DD$. Let $B_\alpha$ be the open disc in ${\CC}$ whose boundary contains the geodesic arc $s_\alpha\subset\DD$, see \cref{fig:cookie}. 
We remark that this is uniquely defined since by our choice of the fundamental domain $\cF$ no side is a diameter, see Condition~\eqref{eq:UpperSetHoleOrderZero}. % and whose interior is disjoint from $\cF$. 

More generally, for any admissible word $\ab{0}\dots\ab{n}$ let $B_{\ab{0}\dots\ab{n}}$ 
be the open disc in $\CC$ such that% identified by the conditions
\begin{gather*}
%	g_{\ab{0}}\circ\dots\circ g_{\ab{n-1}}(\cF)\cap B_{\ab{0}\dots\ab{n}}=\emptyset,\\
	g_{\ab{0}}\circ\dots\circ g_{\ab{n-1}}(s_{\ab{n}})\subset\partial B_{\ab{0}\dots\ab{n}}.
\end{gather*} 
In order to simplify the notation, for any admissible word $\ab{0}\dots\ab{n}$, set
\[
g_{\ab{0}\dots\ab{n}}:=g_{\ab{0}}\circ\dots\circ g_{\ab{n}}.
\]
Throughout this section, we consider $g_{\ab{0}\dots\ab{n}}$ as an automorphism of the Riemann sphere $\overline{\CC}$ (see the beginning of the proof of  \cref{LemmaDistortionHomography} for the explicit form of $g_{\ab{0}\dots\ab{n}}$ of $\Aut(\overline{\CC})$). % ( (given by elements of $\PSL(2,\CC)$)  which fix the disk $\DD$ . 
Recalling that  $g_\alpha$ is the isometry which sends the side $s_{\overline{\alpha}}$ onto the side $s_{\alpha}$, one can see that the map $g_{\ab{0}\dots\ab{n}}$ sends $B_{\overline{\ab{0}}}$ on $\CC\setminus B_{\ab{0}\dots\ab{n}}$.  
Let $\zeta_{\ab{0}\dots\ab{n}}\in\overline{\CC}$ be its pole, that is the point such that $g_{\ab{0}\dots\ab{n}}(\zeta_{\ab{0}\dots\ab{n}})=\infty$. 
We observe that 
$
\zeta_{\ab{0}\dots\ab{n}}\in B_{\overline{\ab{0}}}
$ (since $B_{\ab{0}\dots\ab{n}}$ is a disc in the complex plane $\CC$, so that $\infty\in\CC\setminus B_{\ab{0}\dots\ab{n}}$ 
and the pole $\zeta_{\ab{0}\dots\ab{n}}$, which is the preimage of $\infty$, belongs to $B_{\overline{\ab{0}}}$). 

Thus, the restriction to $\CC\setminus B_{\overline{\ab{0}}}$ of  $g_{\ab{0}\dots\ab{n}}\in\Aut(\overline{\CC})$ realizes a bijection, that we still denote by 
\[
	g_{\ab{0}\dots\ab{n}}\colon\CC\setminus B_{\overline{\ab{0}}}\to B_{\ab{0}\dots\ab{n}}.
\]
%\todo[inline]{Come ha detto anche Mauro, credo che dovremmo fare piu' attenzione a quando usiamo $\HH$ o $\DD$ e alla forma delle omografie... nell'introduzione definiamo solo PSL(2,R), qui improvvisamente usiamo PSL(2,C) e anzi gli automorfismi del disco... Luca tra l'altro nella dimostrazione del Lemma ignori del tutto la forma specifica, se la usassi cosa cambierebbe nei conti? Ho aggiunto il remark seguente qui e uno all'inizio della dimostrazione del lemma, ma non sono sicura sia sufficiente/ideale...}
%, which have the general form 
%Incidentalmente, credo vada scritto che solo in quel lemma stiamo considerando matrici in \PSL(2,\CC) e dire che le Moebius sul disco sono elementi di PSU(1,1)\subset\PSL(2,\CC). ( PSU(1,1) = {(a b \\ \bar{b} \bar{a}) : |a|^2 - |b|^2 =1}/ {\pm id}).
%Let $a$,$b$,$c$,$d$ be complex numbers with $ad-bc=1$ 
%\[
%g_{\ab{0}\dots\ab{n}}(\xi)=\frac{a\xi+b}{\bar{b}\xi+\bar{d}}, \qquad 
%\] 

According to the next \cref{LemmaDistortionHomography}, the restriction map obtained by any admissible word $\ab{0}\dots\ab{n}$ has bounded distortion on a subset $M_{\overline{\ab{0}}}\subset \CC\setminus B_{\overline{\ab{0}}}$. 
More precisely, for any $\alpha\in\cA$ let $M_\alpha\subset\overline{\DD}$ be the closed set (shaded in grey in the example in \cref{fig:cookie}) of those points 
$\xi\in\overline{\DD}$ 
with 
\[
	\inf\{|\xi-\xi'|: \xi'\in B_\alpha\} \geq \min_{\beta\in\cA}\frac{\arclength{\Arc[\beta]}}{2}.
\]
The set $M_\alpha$, as shown in \cref{fig:cookie}, has the shape of a \emph{moon} (thanks to its definition and Condition~\eqref{eq:UpperSetHoleOrderZero}).

\begin{figure}
\centering
\def\svgscale{0.6}
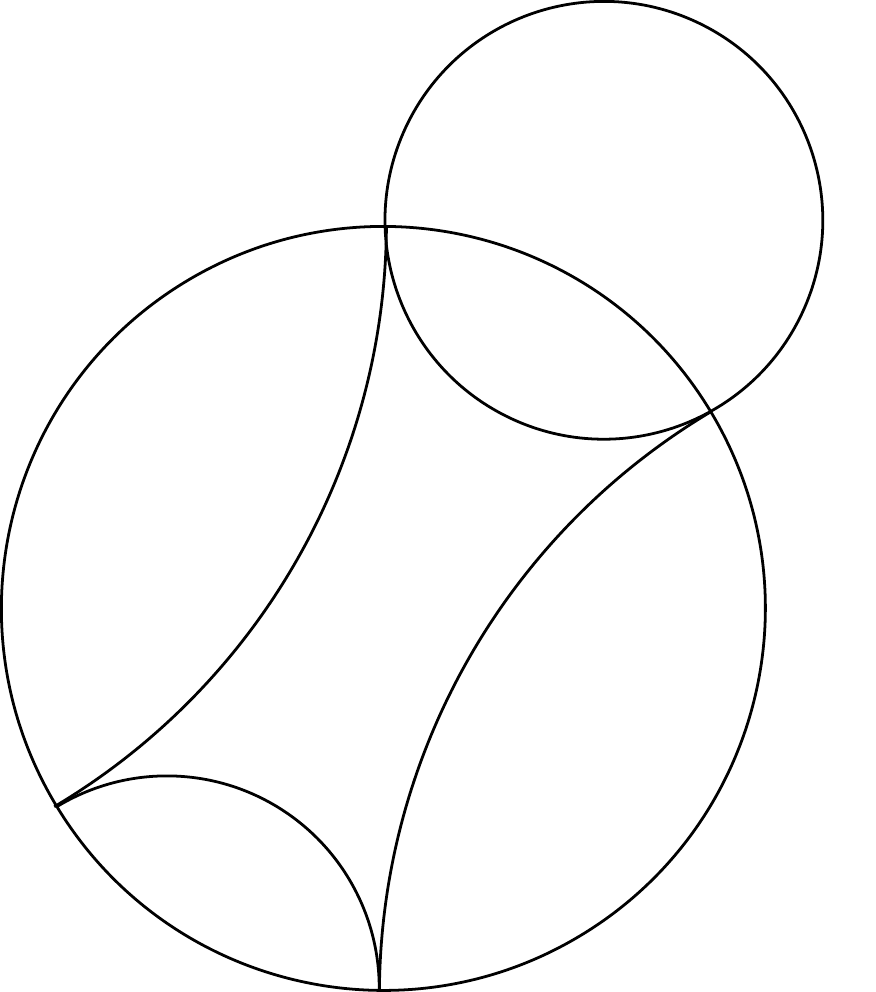
\caption{The objects defined in \cref{sec:distortionestimates}.}
\label{fig:cookie}
\end{figure}

\begin{lemma}
\label{LemmaDistortionHomography}
There exists a constant $C>1$, depending only on $\Gamma$ and on the choice of the fundamental polygon $\cF$ for $\Gamma$ such that the following holds.
Given any admissible word $\ab{0}\dots\ab{n}$ and $\xi_1$, $\xi_2$, $\xi_3$ points in $\partial\DD\cap M_{\overline{\ab{0}}}$ we have
\[
	\frac{1}{C}\frac{|\xi_1-\xi_2|}{|\xi_1-\xi_3|}\leq
	\frac{|g_{\ab{0}\dots\ab{n}}(\xi_1)-g_{\ab{0}\dots\ab{n}}(\xi_2)|}{|g_{\ab{0}\dots\ab{n}}(\xi_1)-g_{\ab{0}\dots\ab{n}}(\xi_3)|}
	\leq C\frac{|\xi_1-\xi_2|}{|\xi_1-\xi_3|}.
\] 
\end{lemma}

\begin{proof}
The automorphism $g_{\ab{0}\dots\ab{n}}\in \Aut(\overline{\CC})$ has the form 
%Let $a$,$b$,$c$,$d$ be complex numbers with $ad-bc=1$ such that for any $\xi\in\CC$ we have 
\[
g_{\ab{0}\dots\ab{n}}(\xi)=\frac{a\xi+b}{c\xi+d},\qquad \text{where}\ a,b,c,d \in \CC \ \text{and} \ ad-bc=1.
\]
Since $g_{\ab{0}\dots\ab{n}}$ sends the disk $\DD$ into itself (so in particular $g_{\ab{0}\dots\ab{n}}\in \Aut(\DD)$), we also have that $c=\overline{b}$ and $d=\overline{a}$, but we will not make use of this in what follows.
%Since $$a$,$b$,$c$,$d$ be complex numbers with $ad-bc=1$$

The pole of $g_{\ab{0}\dots\ab{n}}$ is $\zeta_{\ab{0}\dots\ab{n}}=-d/c$. 
%ncidentalmente, credo vada scritto che solo in quel lemma stiamo considerando matrici in \PSL(2,\CC) e dire che le Moebius sul disco sono elementi di PSU(1,1)\subset\PSL(2,\CC). ( PSU(1,1) = {(a b \\ \bar{b} \bar{a}) : |a|^2 - |b|^2 =1}/ {\pm id}).
Observe in particular that $c\not=0$, otherwise $g_{\ab{0}\dots\ab{n}}$ is complex linear, thus for $\xi\in\CC$ let us write
\[
g_{\ab{0}\dots\ab{n}}(\xi)=
\frac{1}{c}\left(a-\frac{1}{c\xi+d}\right).
\]
In particular, for any pair of points $\xi_1,\xi_2$ in $\CC$ we have
\[
g_{\ab{0}\dots\ab{n}}(\xi_1)-g_{\ab{0}\dots\ab{n}}(\xi_2)=
\frac{\xi_1-\xi_2}{(c\xi_2+d)(c\xi_1+d)}=
\frac{1}{c^2}
\frac{\xi_1-\xi_2}{(\xi_1-\zeta_{\ab{0}\dots\ab{n}})(\xi_2-\zeta_{\ab{0}\dots\ab{n}})}.
\]
Hence, for any three points $\xi_1,\xi_2,\xi_3$ as in the statement we have
\[
\frac
{|g_{\ab{0}\dots\ab{n}}(\xi_1)-g_{\ab{0}\dots\ab{n}}(\xi_2)|}
{|g_{\ab{0}\dots\ab{n}}(\xi_1)-g_{\ab{0}\dots\ab{n}}(\xi_3)|}
=
\frac{|\xi_3-\zeta_{\ab{0}\dots\ab{n}}|}{|\xi_2-\zeta_{\ab{0}\dots\ab{n}}|}	
\cdot
\frac{|\xi_1-\xi_2|}{|\xi_1-\xi_3|}.
\]
As we  observed before the statement of the Lemma, the pole 
$
\zeta_{\ab{0}\dots\ab{n}}\in B_{\overline{\ab{0}}}
$.  
%indeed the map $g_{\ab{0}\dots\ab{n}}$ sends $B_{\overline{\ab{0}}}$ on $\CC\setminus B_{\ab{0}\dots\ab{n}}$, and since $B_{\ab{0}\dots\ab{n}}$ is a disc in the complex plane $\CC$, then $\infty\in\CC\setminus B_{\ab{0}\dots\ab{n}}$ 
%and the pole $\zeta_{\ab{0}\dots\ab{n}}$ belongs to $B_{\overline{\ab{0}}}$. 
Since on the other hand 
$
\xi_i\in M_{\overline{\ab{0}}}
$ 
for $i=1,2,3$, then it follows that 
\[
|\xi_3-\zeta_{\ab{0}\dots\ab{n}}|,|\xi_2-\zeta_{\ab{0}\dots\ab{n}}|
>
\min_{\beta\in\cA}\frac{\arclength{\Arc[\beta]}}{2}.
\]
Moreover any $B_\alpha$ is a disc in the complex plane, thus $\diameter(B_\alpha)<+\infty$ (remark that it is here that we crucially use Condition~\eqref{eq:UpperSetHoleOrderZero}, since it otherwise $B_\alpha$ could have been a semi-plane or the complement of a disk, hence unbounded). 
Since $\diameter(\DD)=2$ we have also 
\[
|\xi_3-\zeta_{\ab{0}\dots\ab{n}}|,|\xi_2-\zeta_{\ab{0}\dots\ab{n}}|
<
2+\diameter(B_{\overline{\ab{0}}}).
\]
The Lemma follows with $C>0$ defined by   
\[
C:=
\left(
2+\max_{\alpha\in\cA}\diameter(B_{\alpha})
\right)
\cdot
\left(
\min_{\beta\in\cA}\frac{\arclength{\Arc[\beta]}}{2}
\right)^{-1}.\qedhere
\]
\end{proof}

\subsection{Size of gaps for the Cantor set in the boundary} \label{SectionSizeGapsCantorBoundary}
The next Lemma gives the estimate for the size of gaps of level zero. We refer to the notation introduced in \cref{SectionCombinatorialDescriptionGaps} to give a hierarchical description of gaps. %Let $B[\xi_i]$ be a gap of level zero, and let $\alpha_i$ and $\alpha_{i+1}$ the pair of letters such that the sides $s_{\alpha_i}$ and $s_{\alpha_{i+1}}$ share $\xi_i$ as a common endpoint. More precisely,  %$\xi^l_\alpha=\xi^r_\beta$.

\begin{lemma}\label{LemmaLengthCuspidalArcs}
Fix $\delta>0$.
There exists $N_0$, depending only on $\delta$, on $\Gamma$ and on the choice of its fundamental domain $\cF\subset\DD$, such that any $N\geq N_0$ and for  gaps of level zero in $\BB_N$, we have
\[
	\arclength{B[\xi_i]}\leq\delta, \qquad 0\leq i \leq 2d-1.
\]
\end{lemma}

\begin{proof}
Since each level zero gap $B[\xi]$  is contained in the union of two adjacent arcs of level $N$ by \cref{EquationUpperSetHoleOrderZero}, 
the Lemma follows directly from the convergence of the Bowen-Series expansion, which  implies that finite cuspidal words $\ab{0}\dots\ab{n-1}$ satisfy $\arclength{\Arc[\ab{0},\dots,\ab{n-1}]}\to 0$ as $n$ tends to infinity.
%\[
%	\bigl|\Arc[\ab{0},\dots,\ab{n-1}]\bigr|\to 0
%	\quad
%	\text{ for }
%	\quad
%	n\to\infty.\]
\end{proof}

Recall that for any $\alpha \in \cA$,  $B[\xi_\alpha^r]$ and $B[\xi_\alpha^l]$ (in the notation of \cref{SectionCombinatorialDescriptionGaps})  are the two gaps of level zero that share an endpoint with the the zero level interval $K[\alpha]$.

\begin{cor}
\label{CorollaryLengthCuspidalArcs}
Fix $\epsilon\in(0,1)$. There exists $N_0$, depending only on $\epsilon$, on $\Gamma$ and on the choice of its fundamental domain $\cF\subset\DD$, such that for any $N\geq N_0$, the following estimate holds for holes and intervals of level $n$ in the  Cantor set $\BB_N$. For any $n\in\NN$  and any admissible word $\ab{0},\dots,\ab{n-1}$ of length $n$ and any letter $\alpha\in\cA$ with 
$
\alpha\not=\overline{\ab{n-1}}
$, we have  
\[
\frac
{\arclength{B[\ab{0},\dots,\ab{n-1};\xi^r_\alpha]}}
{\arclength{K[\ab{0},\dots,\ab{n-1};\alpha]}}
\leq 1-\epsilon
\quad
\text{ and }
\quad
\frac
{\arclength{B[\ab{0},\dots,\ab{n-1};\xi^l_\alpha]}}
{\arclength{K[\ab{0},\dots,\ab{n-1};\alpha]}}
\leq 1-\epsilon.
\]
\end{cor}

%In other words, the Cantor set $\BB_N\subset\partial\DD$ satisfies the $\epsilon$-stable gap condition and the $\epsilon$-size condition introduced in \cref{SectionStableHallTheorem}.

\begin{proof}
By definition of holes and intervals of level $n$ (we refer to \cref{SectionCombinatorialDescriptionGaps}),  
\begin{equation}\label{rewriting_ration}
\frac
{\arclength{B[\ab{0},\dots,\ab{n-1};\xi^r_\alpha]}}
{\arclength{K[\ab{0},\dots,\ab{n-1};\alpha]}} = 
\frac
{\arclength{g_\ab{0}\circ \dots \circ g_\ab{n-1} \left( B[\xi^r_\alpha]\right)}}
{\arclength{g_\ab{0}\circ \dots \circ g_\ab{n-1} \left( K[\alpha]\right)}}
\end{equation}
and the same expression holds with $\xi^r_\alpha$ replaced by $\xi^l_\alpha$. 

Let us remark now  that if $\xi,\xi'\in \partial \DD$ bound an arc $\Arc[\xi,\xi']\subset \partial \DD$ of length strictly less than $\pi$, then  
$|\cdot|$ and  $\arclength{\cdot}$ are comparable, i.e.\ %there exists a constant $C>0$ such that
%Let us first remark that $|\cdot|$ and  $\arclength{\cdot}$ are comparable under the assumption \cref{EquationUpperSetHoleOrderZero}, i.e. for 
$$
\left|\xi-\xi'\right|\leq \arclength{\Arc[\xi,\xi']} \leq \pi \left|\xi-\xi'\right|.
$$
Thus, since each hole or gap is contained in $\Arc[\alpha]$ for some $\alpha \in \cA$  and  $\arclength{\Arc[\alpha]}<\pi$ (see  assumption~\eqref{EquationUpperSetHoleOrderZero}), we can apply this remark to~\eqref{rewriting_ration}. The proof hence follows  from  \cref{LemmaLengthCuspidalArcs} and \cref{LemmaDistortionHomography}.
\end{proof}

\subsection{Sum of Cantor sets on the real line}\label{sec:sum_in_R}
In order to apply results on the sum of Cantor sets,  we now consider the image in  $\RR$ of the Cantor set  $\BB_N=\BB_N^\eta$ (defined in \cref{sec:cantorpreliminaries} and described  in \cref{SectionCombinatorialDescriptionGaps}).  Following \cref{sec:cantorpreliminaries}, let 
$
\KK_N=\KK_N^\eta:=\varphi(\BB_N)
$ 
be its image in $\RR$ under the map 
$
\phi\colon\DD\to\HH
$, 
where $\varphi$ is the inverse of the Caley map $\mathscr{C}$ defined in~\eqref{eq:Cayley}. Explicitly, $\KK_N$ is the set of points 
$ 
x=[\ab{0},\dots,\ab{n},\dots]_{\partial\HH}
$ 
corresponding to no-backtracking cutting sequences $(\ab{n})_{n\in\NN}$ not containing any cuspidal word of length $N+1$ and whose first letter satisfies $\ab{0}\notin \{ {\eta}, \overline{\eta}\}$. Remark that this implies in particular that $\KK_N$ is contained in $[-\mu/2,\mu/2]$ (see \cref{endpointscontrol}). 

Let us write  for, respectively, minimum, maximum and translates of  $\KK_N$ by $z \mapsto z+s\mu$, where $s \in \ZZ$: 
%For $s\in\NN$ let $\KK_N^s=\KK_N+s\mu$. Write
\[
	m_N=\min\KK_N,  \qquad M_N=\max\KK_N, \qquad \KK_N^s:= \KK_N+s\mu, s \in \ZZ.
\]
We claim that, for any integer $s$, the Cantor sets $\KK_N$ and  $\KK_N^s$  satisfy the assumptions of the Stable Hall \cref{thm:stableHall}, namely the $\epsilon$-stable gap condition and of the $\epsilon$-size condition which were defined in \cref{SectionStableHallTheorem}.

\begin{lemma}\label{conditionsStableHall}
There exists an integer $N_0\geq 0$  such that for any $N\geq N_0$ and any integer $s$:
\begin{enumerate}
\item the Cantor sets $\KK_N$, $-\KK_N$ and  $\KK_N^s$ satisfy the $\epsilon$-stable gap condition;
\item   the pairs $\left(\KK_N, \KK_N^s\right)$ and $\left(-\KK_N, \KK_N^s\right)$  satisfy the $\epsilon$-size condition. 
\end{enumerate}
\end{lemma}

\begin{proof}
We claim that it is enough to show that there exists $N_0\in\NN$ such that for any $N\geq N_0$ the Cantor set $\KK_N$ satisfies the $\epsilon$-stable gap condition. This obviously implies that the same holds also for any of its translated image $\KK^s_N$ and also for its reflection $-\KK_N$ (since reflecting only inverts the role of left and right intervals
$K^L$ and $K^R$  in the definition $\epsilon$-size condition~\eqref{EqStableGapCondition}). It is also clear that, for $N$ large enough, the pairs $(\pm \KK_N,\KK_N^s)$ satisfy the $\epsilon$-size condition for any $s$, indeed $|\KK_N|\to {\mu}$ for $N\to\infty$, while the size of the holes shrinks to zero, and the same holds for $-\KK_N$ and the translates $\KK_N^s$. 

%\todo[inline]{Proof still to clarify a little}
Let us hence prove that $\KK_N$ satisfies the $\epsilon$-stable gap condition if $N$ is sufficiently large. Recall that, by the choices made in \cref{sec:setup}, the sides $s_\eta$ and $s_{\overline{\eta}}$ share $\xi_0$, or, more precisely $\xi^l_\eta = \xi_0 = \xi^r_{\overline{\eta}}$ (see \cref{costruzioneF}) and the inverse $\phi\colon\DD\to\HH$ of the Cayley map~\eqref{eq:Cayley} is such that
\[
\phi(\xi_0),=\infty %= \phi(\xi^l_\eta)= \phi(\xi^r_\overline{\eta}),
\qquad 
\phi(\xi^r_\eta)=\frac{\mu}{2}
\qquad
\text{and}
\qquad
\phi\big(\xi^l_{\overline{\eta}})\big)=-\frac{\mu}{2}.
\]
Hence, the arc $\Arc[\eta] \cup \Arc[\overline{\eta}]   \subset\partial\DD$ is the arc with endpoints $\xi^l_{\overline{\eta}}$ and  $\xi^r_\eta$   
which contains in its interior the point $\xi_0=\phi^{-1}(\infty)$, which is the pole of $\phi$. It follows that there is a constant $\kappa>0$, depending only on the choice of the fundamental domain $\cF$ for $\Gamma$, such that 
$
\arclength{\xi-\phi^{-1}(\infty)}\geq \kappa
$ 
for any 
$
\xi\in\partial\DD\setminus \Arc[\eta] \cup \Arc[\overline{\eta}] 
$, 
and thus the restricted map 
\[
\phi\colon\partial\DD\setminus (\Arc[\eta] \cup \Arc[\overline{\eta}])  \to\RR
\]
has bounded derivative. Since by definition $\KK_N\subset \partial\DD\setminus (\Arc[\eta] \cup \Arc[\overline{\eta}])$,  combining the control on the derivative of $\phi$ and the estimate in \cref{CorollaryLengthCuspidalArcs}, the $\epsilon$-stable gap condition for $\KK_N$ follows.  
\end{proof}

The above results together with the Stable Hall Theorem stated in the introduction (and proved in \cref{sec:proof_Stable_Hall}) allow to conclude the proof of \cref{thm:sumHallFuchsian} (and hence \cref{thm:Hall}). 

\begin{proof}[Proof of \cref{thm:sumHallFuchsian}]
Let us check that, for $N$ sufficiently large, we can apply the Stable Hall theorem to the Cantor sets $\KK_N^s$ and $-\KK_N$, in the special case in which $S=S_0$ is the sum function and $U=\mathbb{R}^2$. 
This is the case, since the Lipschitz norm condition~\eqref{EqConditionStableHallTheorem} is trivially satisfied ($S=S_0$) and the $\epsilon$-stable gap and $\epsilon$-size conditions are proved in \cref{conditionsStableHall} for $N \geq N_0$. The Stable Hall theorem then gives
\[ \begin{split}
	\KK_N^s - \KK_N &=S_0 \bigl([\min \KK_N^s ,\max \KK_N^s  ]\times [\min (- \KK_N ),\max (- \KK_N )]\bigr) \\ 
		& = \bigl[\min \KK_N^s+\min (- \KK_N ),\max \KK_N^s + \max (- \KK_N )\bigr] \\ 
		& = \bigl[ (m_N+s\mu) + (-M_N), (M_N+s\mu) + (-m_N)\bigr] , 
\end{split}\]
which is the desired expression. The form of $\KK_N^s + \KK_N$ follows analogously. Thus, $\left| \KK_N^s \pm	\KK_N\right| =2(M_N-m_N)$. Since, as $N \to \infty$, $M_N \to \mu/2$ and $m_N \to -\mu/2$ (and hence $M_N-m_N \to \mu$),  it is enough to increase $N_0$ to ensure that $M_N-m_M > \mu/2$ to have also  that $\left| \KK_N^s \pm	\KK_N\right|\geq \mu$. 
\end{proof}

\section{Penetration estimates}\label{sec:penetration}

In this section we bound the height of a geodesic knowing that the cuspidal words of a piece of its cutting sequence are not too long. Let $\gamma\colon \mathbb{R} \to \DD$ be a geodesic with cutting sequence $(\ab{n})_{n\in\ZZ}$ and let $(t_n)_{n\in\ZZ}$ be the sequence of times when $\gamma$ crosses a side of the tessellation of $\DD$ given by $\cF$, as defined in~\eqref{eq:timescuspidalacceleration}. 
For any $r \in \ZZ$, the cuspidal words $C_r$ and the integers $n(r)$ such that $C_r:=\ab{n(r)},\dots,\ab{n(r+1)-1}$ are also defined in \cref{sec:cuspidalwords}. 
%Let $\gamma\colon[0,+\infty)\to\DD$ be a geodesic such that $\gamma(0)\in\cF$ and that does not converge to a cuspidal point and let $(\ab{n})_{n\in\NN}$ be its cutting sequence. Let $(t_n)_{n\in\NN}$ be the sequence of times defined in~\eqref{eq:timescuspidalacceleration}, that is the instants when $\gamma$ crosses a side of the tessellation of $\DD$ given by $\cF$.  %The sequence is defined by $\gamma(t_0)\in s_{\ab{0}}$ 
%and \[\gamma(t_n)\in g_\ab{0}\circ\dots\circ g_\ab{n-1} (s_\ab{n})\]
%for any $n\geq 1$. 
%Recalling \cref{sec:cuspidalwords}, 
%For any $r\in\NN$, let  
%$
%C_r:=\ab{n(r)},\dots,\ab{n(r+1)-1}
%$ 
%be the $r$-th cuspidal word in the cutting sequence $(\ab{n})_{n\in\NN}$, as defined in \cref{sec:cuspidalwords} and let 
%$
%\wordlength{C_r}=n(r+1)-n(r)
%$ 
%be its length (where the integers $0=n(0)<n(1)<\dots$ are also defined in \cref{sec:cuspidalwords}). %In particular for any $r\in\NN$ the segment $\gamma[t_{n(r)},t_{n(r+1)})$ only intersects images copies $G(s)$ of sides $s$ of $\partial D$ all sharing the same endpoint. 

\begin{lemma}
\label{LemmaEstimateExcursionCuspidalWords}
For any positive integer $N\geq1$ there exists a compact $\cK_N\subset\DD$ such that the following holds. Consider any geodesic $\gamma\colon\RR\to\DD$ with $\gamma(0)\in \cF$.
Let $(\ab{n})_{n\in\ZZ}$ be its cutting sequence and $(C_r)_{r\in\ZZ}$ the corresponding decomposition into cuspidal words. 
Then for any positive integer $r\geq1$ such that  
\[
\wordlength{C_{r-1}}\leq N, \qquad \wordlength{C_{r}}\leq N, \qquad \wordlength{C_{r+1}}\leq N,
\]
and any $t$ with $t_{n(r)}\leq t \leq t_{n(r+1)}$, we have
\[
	g_{\ab{n(r)-1}}^{-1}\circ\dots\circ g_{\ab{0}}^{-1}\big(\gamma(t)\big)\in \cK_N.
\]
\end{lemma}

\begin{proof}
For any ideal vertex $\xi_i$, $i=0,\dots,2d-1$ of the fundamental domain $\cF$, let $\xi_i^{(-)}$ and $\xi_i^{(+)}$ be the two points in $\partial\DD$ at distance $\delta_N>0$ from $\xi_i$, where we set
\[
\delta_N:=
\min\{ \arclength{\Arc[\bb{0},\dots,\bb{N}]}, \bb{0}\dots\bb{N} \text{ cuspidal word of length $N+1$ }\},
\]
and where we recall that $\arclength{\Arc[\bb{0},\dots,\bb{N}]}$ denotes the length of the arc in $\partial\DD$.
Consider the open disc $\DD\subset\CC$ as embedded in the complex plane.
For any ideal vertex $\xi_i$, $i=0,\dots,2d-1$ of the fundamental domain $\cF$, let $B(N,\xi_i)\subset\CC$ be the open Euclidean ball whose boundary $\partial B(N,\xi_i)$ intersects $\partial\DD$ orthogonally at $\xi_i^{(-)}$ and $\xi_i^{(+)}$. Observe that $\DD\setminus B(N,\xi_i)$ is an \emph{hyperbolic half-space}, that is the region of $\DD$ delimited by its boundary and a complete geodesics. In particular 
$\DD\setminus B(N,\xi_i)$ is \emph{hyperbolic convex}, that is it contains the entire segment of hyperbolic geodesic connecting any two of its endpoints. Define a compact set $F_N\subset\DD$ by 
\[
	F_N:=\overline{\cF}\setminus\bigcup_{i=0}^{2d-1}B(N,\xi_i),
\]
that is the subset of the closure of $\cF$ which do not intersects the open balls $B(N,\xi_i)$ defined above, see \cref{fig:fn}. Since the set $F_N$ is an intersection of hyperbolic half-spaces, then it is also hyperbolic convex. Let $\widetilde{\cK}_N$ be the set defined by
\[
\widetilde{\cK}_N:=
\bigcup_{\bb{0}\dots\bb{n-1}} g_{\bb{0}}\circ\dots\circ g_{\bb{n-1}}F_N,
\]
where $\bb{0}\dots\bb{n-1}$ varies among all cuspidal words with $n\leq N$. In particular, considering the trivial word, it is evident that 
$
F_N\subset\widetilde{\cK}_N
$. 
The set $\widetilde{\cK}_N$ is compact, since it is the finite union of images of the compact $F_N$ under the continuous maps 
$
g_{\bb{0}}\circ\dots\circ g_{\bb{n-1}}
$. 
On the other hand, it is possible to see that $\widetilde{\cK}_N$ is not hyperbolic convex. Thus we define $\cK_N$ as the hyperbolic convex hull of $\widetilde{\cK}_N$, so that $\cK_N$ is hyperbolic convex by definition and it is also compact, since $\widetilde{\cK}_N$ is compact.

\begin{figure}
\centering
\def\svgscale{0.5}
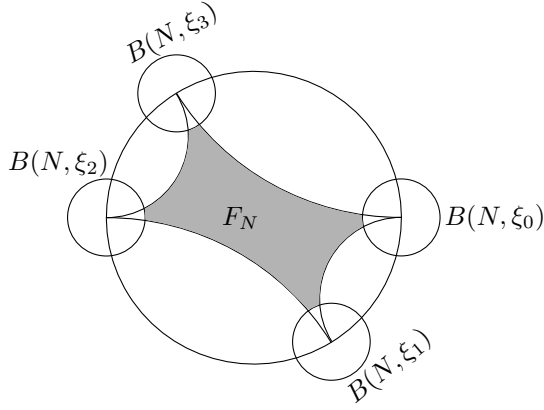
\caption{A schematic picture of the set $F_n$, in dark grey inside $\cF$.}
\label{fig:fn}
\end{figure}

Fix an integer $r$ as in the statement.
Since $\cK_N$ is hyperbolic convex, it is enough to prove the statement for $t=t_{n(r)}$ and $t=t_{n(r+1)}$.
Consider  the normalized geodesics $\gamma_r(t)$ and $\gamma_{r+1}(t)$  in $\DD$ (which are simply the $n(r)^{\text{th}}$ and  $n(r+1)^{\text{th}}$ normalized geodesics as  defined in \cref{gammaj}) given by: 
\[
\begin{split}
	\gamma_r(t)&=g_{\ab{n(r)-1}}^{-1}\circ\dots\circ g_{\ab{0}}^{-1}(\gamma(t)),\\
	\gamma_{r+1}(t)&=g_{\ab{n(r+1)-1}}^{-1}\circ\dots\circ g_{\ab{0}}^{-1}(\gamma(t)).
\end{split}
\]
Using these geodesics, the statement is equivalent to the two conditions
\begin{align*}
\gamma_r(t_{n(r)})&\in\cK_N,\\
\gamma_{r}(t_{n(r+1)}) &= g_{\ab{n(r)}}\circ\dots\circ g_{\ab{n(r+1)-1}}(\gamma_{r+1}(t_{n(r+1)}))\in\cK_N.
\end{align*}
Recall that hyperbolic geodesics intersect in at most one point. Observe also that we have
\begin{align*}
	\gamma_r(+\infty)&=[\ab{n(r)},\ab{n(r)+1},\dots]_{\partial\DD},	&
	\gamma_{r+1}(+\infty)&=[\ab{n(r+1)},\ab{n(r+1)+1},\dots]_{\partial\DD},
	\\
	\gamma_r(-\infty)&=[\overline{\ab{n(r)-1}},\overline{\ab{n(r)-2}},\dots]_{\partial\DD},	&
	\gamma_{r+1}(-\infty)&=[\overline{\ab{n(r+1)-1}},\overline{\ab{n(r+1)-2}},\dots]_{\partial\DD}.
\end{align*}

In order to prove $\gamma_r(t_{n(r)})\in \cK_N$ we prove the stronger condition 
$
\gamma_r(t_{n(r)})\in F_N
$. 
If the latter does not hold, then there is some $i$ for which we either have $\gamma_r(+\infty)\in B(N,\xi_i)\cap\partial\DD$ or $\gamma_r(-\infty)\in B(N,\xi_i)\cap\partial\DD$.
However the first condition would imply that $\ab{n(r)}\dots\ab{n(r+1)-1}$ is a cuspidal word of length greater than $N+1$.
Similarly, the second condition would imply that $\overline{\ab{n(r)-1}}\dots\overline{\ab{n(r-1)}}$ is a cuspidal word of length $N+1$, which is equivalent to the same condition for $\ab{n(r-1)}\dots\ab{n(r)-1}$. 

In order to prove $\gamma_{r}(t_{n(r+1)})\in\cK_N$, we prove the stronger condition 
$
\gamma_{r}(t_{n(r+1)})\in\widetilde{\cK}_N
$. 
To do so, observe that the same argument as above applies to the words $\ab{n(r+1)}\dots\ab{n(r+2)-1}$ and $\ab{n(r)}\dots\ab{n(r+1)-1}$ and implies $\gamma_{r+1}(t_{n(r+1)})\in F_N$, that is 
\[
	\gamma_{r}(t_{n(r+1)})=g_{\ab{n(r)}}\circ\dots\circ g_{\ab{n(r+1)-1}}(\gamma_{r+1}(t_{n(r+1)}))
	\in
	g_{\ab{n(r)}}\circ\dots\circ g_{\ab{n(r+1)-1}} F_N \subset \widetilde{\cK}_N.\qedhere
\]
\end{proof}

\section{Stable Hall rays for proper functions}\label{sec:Hallperturbations}
In this section we prove \cref{thm:Hallperturbations}.

Let $G$ be a Fuchsian group that is a non uniform lattice and such that $\infty$ is a parabolic fixed point of $G$. 
As in \cref{sec:setup}, we also let $\Gamma<G$ be its maximal finite index normal subgroup without elliptic elements and we choose a fundamental domain $\cF$ for $\Gamma$ that satisfies the conclusions of \cref{costruzioneF} and label by $\alpha \in \cA$ its sides as in \cref{sec:background}. Every geodesic boundary expansion $(a_n)_{n \in \ZZ}$ in this section is with respect to this fundamental domain and is such that $a_n \in \cA$.

Finally, let us remark that any $h\colon \HH \to \RR$ which is $G$-invariant is in particular $\Gamma$-invariant, since $\Gamma<G$
Throughout this section, we will only use $\Gamma$-invariance.

\subsection{Naive height as a function of endpoints}\label{sec:heightendpoints} 
Let $h\colon\HH\to\RR_+$ be %as before (or, more in general for this section, any  $h\colon\HH\to\RR_+$
%which is  
a $\Gamma$-invariant continuous function such that the induced function 
$
h\colon\Gamma\backslash\HH\to\RR_+
$ 
is proper. Given $l>0$ define the function $h_l\colon\HH\to\RR_+$ by
\[
	h_l(z)=
			\begin{cases}
				h(z), & \text{if $\Im(z)>l$,}\\
				0,	&	\text{if $\Im(z)\leq l$.}
			\end{cases}
\]

Recall that hyperbolic geodesics are uniquely determined by their endpoints on $\RR$ and that, for $x_1, x_2 \in \RR $ with $ x_1 \neq x_2$ we denote by $\gamma(x_1,x_2)$ the unique  geodesic  $\gamma(t)$ with  $\gamma(-\infty)=x_1$ and $\gamma(+\infty)=x_2$.  Denoting by 
\[
	\Delta:= \{(x_1,x_2)\in\RR^2: x_1=x_2\}
	\]
the diagonal in $\RR^2$, we  define the function $H \colon \RR^2\setminus \Delta \to\RR_+$ by 
\begin{equation}\label{def:H}
	H(x_1,x_2):=
	\sup_{t\in\RR}
	\left\{
	h_l(\gamma(t)), \text{ where } \gamma= \gamma(x_1,x_2) \right\}.
\end{equation}
%With the terminology of the introduction, we have 
%$H(x_1,x_2)=L(h_m,\gamma)$, 
%where $\gamma=\gamma(x_1,x_2)$ is the unique geodesic in $\HH$ with $\gamma(-\infty)=x_1$ and $\gamma(+\infty)=x_2$. 
In parallel, consider also the function  $H_0\colon\RR^2\setminus \Delta \to\RR_+$ given by
$$
H_0(x_1,x_2):= \frac{|x_2-x_1|}{2} =\sup_{t\in\RR}
	\left\{ 	\text{Im} (\gamma(t)), \text{ where } \gamma= \gamma(x_1,x_2) \right\}.
$$
Finally, for any $l>0$ let $U_l\subset\RR^2$ be the complement of the $2l$-neighborhood of $\Delta$ defined as the set $U_l=\{(x_1,x_2)\in\RR^2: |x_2-x_1|\geq 2l\}$.

\begin{rem}\label{correspondenceU}
Observe that, for any $(x_1,x_2)\in\RR^2$, %denoting as before by $\gamma(x_1,x_2)$  the unique geodesic in $\HH$ with $\gamma(-\infty)=x_1$ and $\gamma(+\infty)=x_2$, 
we have $\gamma(x_1,x_2)\cap\cU_l\neq\emptyset$ if and only if $(x_1,x_2)\in U_l$. 
\end{rem}

The regularity of $H\colon\RR^2\to\RR_+$ depends on the regularity of $h\colon\HH\to\RR_+$ via \cref{LemmaUniformNorm} and \cref{PropositionLipschitzNorm} below. The proof of  \cref{LemmaUniformNorm} is an easy estimate and it is left to the reader; \cref{PropositionLipschitzNorm} is proved in \cref{app2}.

\begin{lemma}
\label{LemmaUniformNorm}
For any $l>0$, we have
\[
	\|(H-H_0)|_{U_l} \|_\infty
	\leq 
	\|(h-\Im)|_{\cU_l}\|_\infty.
\]
\end{lemma}

The definition of Lipschitz constant $\lipschitz ( \cdot )$ %and Lipschitz norm $\|\cdot \|_{\text{Lip}} 
  for $H\colon\mathbb{R}^2 \to \mathbb{R}$ corresponds to \cref{def:Lip_const} under the standard identification between $\CC$ and $\RR^2$. 
% We will here below identify $\CC$ and $\mathbb{R}^2$ and hence consider $G$ as a function $G:\mathbb{R}^2 \to \mathbb{R}$. The definition of Lipschitz function 

\begin{prop}
\label{PropositionLipschitzNorm}
For any $l>0$ we have
\[
	\lipschitz \bigl( (H-H_0)|_{U_l} \bigr)
	\leq 
	\left(\sqrt{2}+\frac{\sqrt{2}}{l}\right)\cdot\|(h-\Im)|_{\cU_l}\|_{\text{Lip}}.
\]
In particular the function $H\colon\RR^2\setminus \Delta 
%\{x_1=x_2\}
\to\RR_+$ is continuous.
\end{prop}

\begin{rem}
We observe that in general the function $H\colon\RR^2\to\RR_+$ is not differentiable even if $h\colon\HH\to\RR_+$ is. For example, for $h(x,y):=\sqrt{x^2+y^2}$ one gets 
\[
	H(x_1,x_2)=\max\{|x_1|,|x_2|\}.
\]
This is the reason why it is natural to consider the Lipschitz category rather than the category of smooth functions.
\end{rem}

\subsection{Intervals from endpoints in Cantor sets}
\label{sec:stablesum}
Let $\BB_N=\BB_N^\eta$ be the Cantor set defined in \cref{sec:cantorpreliminaries} and described combinatorially in \cref{SectionCombinatorialDescriptionGaps}. Following \cref{sec:cantorpreliminaries}, let 
$\KK_N=\KK_N^\eta:=\varphi(\BB_N)$ 
be its image in $\RR$ under the map 
$\phi\colon\partial\DD\to\partial\HH$,
where $\varphi$ is as usual the inverse of the Caley map $\mathscr{C}$ defined in~\eqref{eq:Cayley}.

\begin{prop}
\label{thm:sumofcantors}
If there exist $0<\epsilon<1$, $l>0$ and a function $H\colon\RR^2\to\RR$ such that 
\begin{equation}\label{LipH}
	\frac
	{1-2\cdot\lipschitz((H-H_0)|_{U_l})}
	{1+2\cdot\lipschitz((H-H_0)|_{U_l})}>
	1-\epsilon, 
\end{equation}
then there exist natural numbers $N_0$ and $s_0$ such that, for any $N\geq N_0$ and $s\geq s_0$, $H(\KK_N\times\KK_N^s)$ is an interval. More precisely, we have
\[
	|\KK_N|=M_N-m_N>\frac{\mu}{2} \qquad \text{and} \qquad H(\KK_N\times\KK_N^s)=H([m_N,M_N]\times[m_N+s\mu,M_N+s\mu]),
\]
where $m_N=\min\KK_N$ and $M_N=\max\KK_N$.
\end{prop}

The proof is simply a reduction of the statement to an application of the Stable Hall \cref{thm:stableHall}.
We remark that, in order for~\eqref{LipH} to be satisfied, we must have $\lipschitz((H-H_0)|_{U_l})<1/2$.

\begin{proof}
Let us apply  a change of variable that allows to reduce the function $H_0$ (restricted to the set $\{ x_2 > x_1 \}$) to the sum function $S_0$. 
Let 
\[
	(y_1,y_2) = \psi(x_1,x_2) := \biggl(-\frac{x_1}{2},\frac{x_2}{2}\biggr),
\]
so that, if $x_2> x_1$,
\[ 
	S_0(y_1,y_2)= -\frac{x_1}{2}+ \frac{x_2}{2} = \frac{|x_2-x_1|}{2}=  H_0 (x_1,x_2). 
\]
Then define the function $S\colon\RR^2\to\RR$ (which we will also use only on $\{ x_2>x_1\}$) by 
$$
S(y_1,y_2):=H(\psi^{-1}(y_1,y_2)) =  H(-2y_1, 2y_2).
$$ 

Choose $s_0\geq 1$ sufficiently large so that $\KK_N \times \KK_N^s \subset U_l$ for any $s\geq s_0$. Fix now any $s \geq s_0$ and consider the Cantor sets $\KK':=-(1/2)\KK_N$ and $\FF':=(1/2)\KK_N^s$, so that $$\KK' \times \FF' = \psi \left(\KK_N \times \KK_N^s \right) .$$ 
We showed in \cref{conditionsStableHall} that there exists $N_0 \in \NN$ such that, for any $N\geq N_0$, $\KK_N$ and $\KK_N^s$ each satisfy the  $\epsilon$-stable gap condition introduced in \cref{SectionStableHallTheorem} and jointly satisfy as a pair $(\KK_N, \KK_N^s)$ the $\epsilon$-size condition (also defined in \cref{SectionStableHallTheorem}). It is clear from the definitions that this implies that  the Cantor sets $\KK'$ and $\FF'$, which are  images by affine maps of $\KK_N$ and $\KK_N^s$ respectively,  also satisfy such conditions.

Let $U' = \psi(U_l)$ be the image of $U_l \subset  \mathbb{R}^2$ (defined in \cref{sec:heightendpoints}). Since  $\KK_N \times \KK_N^s \subset U_l$ (we fixed $s\geq s_0$),   $\KK'\times\FF'\subset U'$. Moreover, observe that we have
$
\lipschitz(S-S_0)=2\cdot\lipschitz(H-H_0)
$, 
so that 
\[
\frac
{1-\lipschitz(S-S_0)}
{1+\lipschitz(S-S_0)}
=
\frac
{1-2\cdot\lipschitz(H-H_0)}
{1+2\cdot\lipschitz(H-H_0)}.
\]
Hence, from the assumption~\eqref{LipH}   on the Lipschitz constant   of $(H-H_0)$ restricted to ${U_l}$, it follows that $S-S_0$ satisfies the Lipschitz constant assumption~\eqref{EqConditionStableHallTheorem} of \cref{thm:stableHall} on $U'$. 

Thus, we can apply \cref{thm:stableHall}. Let us now rephrase its  conclusion in terms of $H$.  Notice  that for $s_0\geq 1$, $\min \KK_N^s > \max \KK_N$, so that $( \KK_N, \KK_N^s )\subset \{ x_2>x_1\}$. Thus, on the set $\KK_N \times \KK_N^s$ we have, as seen above, that $S\circ \psi = H$.  This gives that
$$
S(\KK'\times\FF')= S \circ \psi (\KK_N\times\KK_N^s)= H (\KK_N\times\KK_N^s) ,
$$ 
and, similarly, that 
\[
S\bigl([\min\KK',\max\KK']\times[\min\FF',\max\FF']\bigr)=
H\bigl([\min\KK_N, \max\KK_N]\times[\min\KK^s_N,\max\KK_N^s]\bigr)
.\]
This in particular implies that $H (\KK_N\times\KK_N^s)$ is an interval (see \cref{rk:stable_gives_Cantorsum}) and concludes the proof. 
\end{proof}

The following Corollary gives the starting point for  the existence of a Hall ray. 
The reader should compare this Corollary (and its proof) with the simpler analogue \cref{cor:decomposition}. 
\begin{cor}\label{thm:decompositionperturbations}
If $H\colon\RR^2\to\RR$ is such that there exist $0<\epsilon<1$, $l>0$ for which \eqref{LipH} holds and $\|H-H_0\|_\infty <1/4$,   up to increasing $N_0$, for any $N\geq N_0$, any $s_1\geq s_0$ and any
\[
L\geq \inf H(\KK_N\times\KK_N^{s_1})
\]
there exist points
$
x_1=[\ab{0},\ab{1},\dots]_{\partial\HH}\in\KK_N
$ 
and 
$
x_2=[\bb{0},\bb{1},\dots]_{\partial\HH}\in\KK_N
$ 
and an integer $s\geq s_1$ such that 
\[
L=H(x_1,x_2+s\mu).
\] 
\end{cor}

\begin{proof}
Recall that we are assuming that the Margulis neighborhood starts at $m=1$, which implies that $\mu\geq 1$. Hence, since $\|H-H_0\|_\infty <1/4$
and   $	M_N-m_N \to \mu $ as $N$ tends to infinity, 
increasing $N$ if needed, we can assume that we have the following \emph{overlapping condition}
\[
	M_N-m_N=|\KK_N|\geq \frac{\mu}{2}+2\|H-H_0\|_\infty.
\]
We will now show that this condition implies that for any $s\geq s_0$ we have
\begin{equation}\label{overlap} 
H(\KK_N\times\KK_N^s)\cap H(\KK_N\times\KK_N^{s+1})\neq\emptyset.
\end{equation} 
Since by \cref{thm:sumofcantors} we know that both $H(\KK_N\times\KK_N^s)$ and $H(\KK_N\times\KK_N^{s+1})$ are intervals, this implies that they overlap and that there is no gap between them.
This is then enough to conclude, since, remarking also that $\sup H(\KK_N\times\KK_N^s) $ tends to $+\infty$ as $s$ grows, it implies that for any $s_1\geq s_0$,
\[
\bigcup_{s\geq s_1 } H(\KK_N\times\KK_N^s) \supset \left( \inf H(\KK_N\times\KK_N^{s_1}), +\infty \right) ,
\]
so  any $L\geq \inf H(\KK_N\times\KK_N^{s_1})$ belongs to $H(\KK_N\times\KK_N^s)$ for some $s\geq s_1$ and hence can be written as $H(x_1,x_2+s\mu)$ for some $x_1 \in \KK_N$ and $x_2 +s \mu \in K^s_N$. 

It remains to show~\eqref{overlap}.
On one hand we have that, for $s\geq s_0$, 
\[
\begin{split}
	\sup H(\KK_N\times\KK_N^s) 
	\geq H(m_N,M_N+s\mu) & > 
	H_0(m_N,M_N+s\mu)-\|H-H_0\|_\infty  \\ &= \frac{M_N +s\mu-m_N}{2}-\|H-H_0\|_\infty;
\end{split}
\]
while, for $s+1$, we have
\[
	\begin{split}
\inf H(\KK_N\times\KK_N^{s+1}) 
	\leq H(M_N,m_N+(s+1)\mu)& < 
	H_0(M_N, m_N+(s+1)\mu)+\|H-H_0\|_\infty \\ & = \frac{m_N+(s+1)\mu - M_N}{2}+\|H-H_0\|_\infty ,
\end{split}
\]
where to remove the absolute value in $H_0$ we used that $m_N + (s+1)\mu \geq M_N$ for any $s\geq0$.
Since one can check that the overlap condition implies that
\[
 \frac{M_N+s\mu -m_N}{2}-\|H-H_0\|_\infty  \geq \frac{m_N+(s+1)\mu - M_N}{2}+\|H-H_0\|_\infty ,\\
\]
the combination of the last three inequalities shows that $\sup H(\KK_N\times\KK_N^s) > \inf H(\KK_N\times\KK_N^{s+1}) $ and hence~\eqref{overlap} holds.  As explained above, this concludes the proof.
\end{proof}

\subsection{A Perron-like formula to produce values in the Hall ray}\label{sec:Perronperturbations}

The next result is the key step for the proof of \cref{thm:Hallperturbations} on the existence of a Hall ray for proper functions. It provides a formula which will allows us to verify that certain geodesics realize values of the Lagrange spectrum. The formula resembles the generalized Perron formula in \cref{sec:Perron} (see~\eqref{eq:heightlimsup}).
However, since we have the additional difficulty of controlling the values of the proper function $h$ in the other cusps, we can only prove it for sequences of a special form, which we will use to prove the existence of the Hall ray in \cref{subsec:proofHallperturbations}.

\smallskip
Let  $h\colon\HH\to\RR_+$ be a function satisfying the assumptions of \cref{thm:Hallperturbations} and let $H$ denote the function defined in \cref{def:H}.  Recall also that we are assuming that the $\Gamma$ contains the parabolic element $p=\left(\begin{smallmatrix} 1 & \mu \\ 0 & 1 \end{smallmatrix}\right)$ which acts on $\HH$ as $p(z)=z+\mu$ and that $p=g_\eta$ for some $\eta\in \cA$. 
%has a parabolic fixed point at $\infty$ and  

\begin{prop}
\label{PropositionFormulaContinuedFraction}
For any integer $N\geq 1$ there exists an integer $M=M(l_0,h,N)\geq N$ such that the following holds. 
Let $\gamma\colon\RR\to\HH$ be a hyperbolic geodesic and let $(\ab{n})_{n\in\ZZ}$ be its cutting sequence. Assume that the cuspidal words $(C_r)_{r\in\NN}$ in the cuspidal decomposition of the positive half sequence $(\ab{n})_{n\in\NN}$ satisfy the properties below.
\begin{enumerate}
\item There exists an increasing subsequence $(r_k)_{k\in\NN}$ such that, for any $k\in\NN$,  $C_{r_k}$ 
is a cuspidal word obtained concatenating $M\geq M_0$ times the letter $\eta$  corresponding to the parabolic element $p=g_\eta$, i.e.\
\[
C_{r_k}= \eta^M= 
\underbrace{\eta \dots \eta}_{\text{M times}}, \qquad M\geq M_0.
\] 
Moreover, we eventually have $r_k-r_{k-1}>3$.
\item For any $ r \neq r_k$, $k\in\ZZ$, we have $\wordlength{C_r}\leq N$.\end{enumerate}
Then, we have
\[
	\limsup_{t\to+\infty}{h(\gamma)} = 
	\limsup_{k\to\infty}{H([\ab{n(r_k)-1},\ab{n(r_k)-2},\dots]_{\partial\HH}^-},
	[\ab{n(r_k)},\ab{n(r_k)+1},\dots]_{\partial\HH},
	\bigr)
\]
where the notation  $[\, \cdot\, ]_{\partial\HH}^-$ was defined in~\eqref{def:backwardCF}. 
\end{prop}

The latter formula, %Equation \cref{Perronlike}, 
in the special case of $h=\Im $ and $H_0(x,y) = |x-y|/2$, should be compared to the generalization~\eqref{eq:heightlimsup} of   Perron's formula~\eqref{Perron_formula_CF}.

Before starting the proof of the Proposition, let us  explain the strategy of the proof.  

\medskip
\emph{Outline of the Proof of \cref{PropositionFormulaContinuedFraction}.}  
Let $(t_n)_{n\in\ZZ}$ be the sequence of hitting times of $\gamma$ with sides of the ideal tessellation of $\HH$ induced by $\cF$ (see~\eqref{eq:timescuspidalacceleration}) and let $\gamma([t_{n(r)},t_{n(r+1)}))$ be the segment of $\gamma$  encoded by the $r^{\text{th}}$ cuspidal word $C_r=\ab{n(r)},\dots,\ab{n(r+1)-1}$ (we refer to \cref{sec:cuspidalwords} for definitions). 

We will split the integers $r$ indexing cuspidal words into two groups and consider two different scenarios. First, consider any $r$ which is not equal to any of  $r_k$, $r_k-1$ or  $r_k+1$, for some $k\in\NN$.  We will call these indexes $r$ \emph{intermediate times}.   For these $r$, the parabolic word $C_{r}$, as well as the preceding  and following parabolic words $C_{r-1}$ and $C_{r+1}$ all have by assumption length at most $N$. %, so we can apply \cref{}. 
Thus,  using \cref{LemmaEstimateExcursionCuspidalWords}, we will establish for any such $r$ a uniform bound $C_1(N,h)>0$ for the supremum of $h$ along the  $r^{\text{th}}$ segment of $\gamma$, namely we will show (in \emph{Step} 1 of the proof) that
$$
\sup_{t_{n(r)}\leq t<t_{n(r+1)}}h\big(\gamma(t)\big) \leq C_1(N,h), \qquad \forall \ r\notin\cup_{k \in \ZZ}  \{ r_k-1, r_k, r_k+1\}.
$$
%or more precisely along the image under $F_r$ of any segment $\gamma\big([t_{n(r)},t_{n(r+1)}]\big)$ composing it, where $r_k+2\leq r\leq r_{k+1}-2$. 

Then, we will consider the parabolic words $C_{r_k}$, coupled together with the precedent and following parabolic words, respectively $C_{r_{k}-1}$ and $C_{r_{k}+1}$, and consider the segments of $\gamma$ corresponding to the triple $C_{r_{k-1}} C_{r_k} C_{r_{k+1}}$, that we will denote $\gamma^{(k)}$, in other words we set
$$
\gamma^{(k)}:= \{ \gamma(t), \ %\gamma\big([I_k]\big), \qquad \text{where}\quad  I_k :=
 t_{n(r_{k}-1)} \leq t \leq  t_{n(r_{k}+2)} \}, \qquad k \in \NN.
$$
We show (in \emph{Step} 2 and \emph{Step} 3 of the proof) that the supremum of $h$ along $\gamma^{(k)}$  is bigger than $C_1(N,h)$. 
%$
%F_{r_k}\circ\gamma\big([t_{n(r_{k}-1)},t_{n(r_{k}+2)}]\big)
%$ 
 % and more over is taken in a central sub-arc 
%$
%F_{r_k}\circ\gamma\big([t_{n(r_{k})},t_{n(r_{k}+1)}]\big)$ 
In order to do this, since the function $h$ is \emph{proper} on $X=\Gamma\backslash\HH$, so that $h(z)$ diverges as $z$ get closer to $\partial\HH=\RRbar$, we first need to establish a lower bound on $\Im(z)$ for  $z \in \gamma^{(k)}$ (this is done in \emph{Step} 2). This then allows to guarantee that, for sufficiently large excursions into the cusp at infinity (i.e.\ when $M$ is sufficiently large), the supremum of $h$ along  $\gamma^{(k)}$ is taken for $z$ in the central part of the segment.
%\gamma\big([t_{n(r_{k}-1)},t_{n(r_{k}+2)}]\big)

Finally, in \emph{Step} 3, we will show that, when we \emph{normalize} the geodesic segment $\gamma^{(k)}$ to bring it back to the fundamental domain (so that it crosses the center of the disk  in the time interval $[t_{n(r_{k}-1)} , t_{n(r_{k}+2)}]$), then the maximum of $h$  is taken inside the region $\cU_l$. The Proposition then follows because this last property enables to express the supremum of $h$ along the central segment as the value of $H(\cdot,\cdot)$ at the two endpoints of the renormalized geodesic, which leads to the desired formula. 
%\[
%F_{r_k}\circ\gamma(-\infty)=
%[\overline{\ab{n(r_k)-1}},\overline{\ab{n(r_k)-2}},\dots]_{\partial\HH}
%\quad
%\textrm{ and }
%\quad
%F_{r_k}\circ\gamma(+\infty)=
%[\ab{n(r_k)},\ab{n(r_k)+1},\dots]_{\partial\HH}.
%\]

\begin{proof}[Proof of \cref{PropositionFormulaContinuedFraction}]
 We begin by introducing some auxiliary notation for the proof. Recall that $\gamma(0)\in\phi(\cF)$. For any $r\geq1$ it is convenient to introduce the group element associated to the $r^{\text{th}}$ word $C_r$ in the parabolic decomposition, i.e.\
$$
%g(C_r)
G_r:=g_{\ab{n(r)}} \circ g_{\ab{n(r)-1}} \circ \cdots \circ g_{\ab{n(r+1)-1}}\in\Gamma.
$$
We also define the $r^{\text{th}}$ \emph{normalizing element} $F_r$ to be the product:
$$
F_r:= G_{r-1}^{-1} \cdot G_{r-2}^{-1} \dots \circ G_0^{-1} = % g(C_{r-1})^{-1} \cdot g(C_{r-2})^{-1} \dots \circ g(C_0)^{-1} = 
g_{\ab{n(r)-1}}^{-1}\circ\dots\circ g_{\ab{0}}^{-1}\in\Gamma.
$$
This is the element of $\Gamma$ that \emph{renormalizes} the geodesic at time $t_{n(r)}$, in the sense that the geodesic $F_r(\gamma(t))$ passes through the fundamental domain $\phi(\cF)$ for some portion of the time $[t_{n(r)},t_{n(r+1)})$.
%\todo[inline]{Check formulas above}

We remark that, for any  $r\geq 1$, we have 
\begin{equation}\label{F_recursion}
%F_r= g(C_{r-1})^{-1} F_{r-1} =g_{\ab{n(r)-1}}^{-1}\circ\dots g_{\ab{n(r-1)}}^{-1}\circ F_{r-1}.
F_r= G_{r-1}^{-1} F_{r-1} =g_{\ab{n(r)-1}}^{-1}\circ\dots g_{\ab{n(r-1)}}^{-1}\circ F_{r-1}. 
\end{equation}

\smallskip
\emph{Step 0. Reformulation of the $\limsup$.}
 
Let us first  express the $\limsup$ of $h$ along $\gamma$ as the $\limsup$ over $r$ of the supremum of $h$ along the $r^{\text{th}}$ segment of $\gamma$, and use the invariance of $h$ under $F_r$, to get
\[
\begin{split}
	\limsup_{t\to+\infty}h\bigl(\gamma(t)\bigr) &= 
	\limsup_{r\to+\infty} \sup_{t_{n(r)}\leq t<t_{n(r+1)}} h\bigl(\gamma(t)\bigr) \\
	&=\limsup_{r\to+\infty} \sup_{t_{n(r)}\leq t<t_{n(r+1)}} h\bigl(F_r\circ\gamma(t)\bigr).
\end{split}
\]

\emph{Step 1. Upper bound on intermediate times.} 

Fix $r\in\NN$. We first establish a uniform bound for the value of $h$ along any segment 
$
\gamma\big([t_{n(r)},t_{n(r+1)}]\big)
$ when $r$ is not equal to $r_k, r_{k}-1$ or $r_k+1$ for any $k\in\ZZ$. 
 Fix an integer $N\geq1$, let $\cK_N\subset\DD$ be the compact set provided by \cref{LemmaEstimateExcursionCuspidalWords} and consider its image $\phi(\cK_N)\subset\HH$ in the upper half plane, which is compact too. Since $h$ is continuous, set
\[
C_1=C_1(N,h):=\max_{z\in\phi(\cK_N)}h(z)<+\infty.
\]
By assumption $\wordlength{C_r}\leq N$ for any $r\neq r_k$ for any $k\in\ZZ$.  Therefore, for all but at most finitely many intermediate $r$, i.e.\ any $r$ different than any of  $r_k$, $r_k+1$ and $r_k-1$, $C_r$ is preceded and followed by cuspidal words with length less or equal to $N$ and hence  \cref{LemmaEstimateExcursionCuspidalWords} implies that, for any such $r$ and any $t$ with $t_{n(r)}\leq t \leq t_{n(r+1)}$, we have
\[
F_r(\gamma(t))\in \phi(\cK_N).
\]
Thus, for any intermediate $r$ we have
\begin{equation}\label{Equation1PropositionFormulaContinuedFraction}
	\sup_{t_{n(r)}\leq t<t_{n(r+1)}}h\big(F_r(\gamma(t))\big)\leq C_1(N,h).
\end{equation}

\smallskip
\emph{Step 2. Lower bound for the imaginary part along special segments.} 

Now we consider one of the special segments $\gamma^{(k)}$ which corresponds the block $C_{r_k-1} C_{r_k} C_{r_k+1}$ and we establish a lower bound on $\Im(z)$ for  $z \in \gamma^{(k)}$, in order to guarantee that the supremum of $h$ along  $\gamma^{(k)}$ is taken for $z$ in the central part of the segment.

Observe first that for any hyperbolic geodesic $t\mapsto\gamma'(t)$ in $\HH$, for any $a,b$ in $\RR$ with $a<b$ and for any $t\in[a,b]$ one has 
$
\Im\big(\gamma'(t)\big)\geq 
\min
\left\{
\Im\big(\gamma'(a)\big),\Im\big(\gamma'(b)\big)
\right\}.
$ 
Thus, it is enough to control the imaginary part at the endpoints of the geodesic segment $\gamma^{(k)}$, which correspond to $t_{n(r_k-1)}$ and $t_{n(r_k+2)}$ respectively. 

To this end,  since for all $k$'s large enough, $r_k-r_{k-1}\geq 4$, $C_{r_k-2}$ and $C_{r_k+2}$ are both preceded and followed by cuspidal words of length less than $N$, using again \cref{LemmaEstimateExcursionCuspidalWords} with $r=r_k-2$ and $t=t_{n(r_k-1)}$ and, respectively, $r=r_k+2$ and $t=t_{n(r_k+2)}$ we have
\[
	F_{r_k-2}\left(
	\gamma \left(t_{n(r_k-1)}\right)\right)
	\in\phi(\cK_N)
	\quad
	\text{ and } 
	\quad
	F_{r_k+2}\left(
	\gamma \left(t_{n(r_k+2)}\right)\right)
	\in\phi(\cK_N).
\]
Therefore, recalling \cref{F_recursion}, we have % observing that $F_r=g_{\ab{n(r)-1}}^{-1}\circ\dots g_{\ab{n(r-1)}}^{-1}\circ F_{r-1}$ for any $r\geq 1$, we have 
	\begin{align*}
	F_{r_k}\left(
	\gamma \left(t_{n(r_k-1)}\right)\right)
	&= G_{r_k-1}^{-1} G_{r_k-2}^{-1}  F_{r_k-2}\left(
	\gamma\left(t_{n(r_k-1)}\right)\right)
	\\ & \in
G_{r_k-1}^{-1} G_{r_k-2}^{-1}
%	g_{\ab{n(r_k)-1}}^{-1}\circ\dots\circ g_{\ab{n(r_k-2)}}^{-1}
	%\bigl(
\left(	\phi\left(\cK_N\right)\right),%\bigr),
\end{align*}
and, similarly,
\begin{align*}
	F_{r_k}\left( \gamma \left(t_{n(r_k+2)}\right)\right) 
&= G_{r_k} G_{r_{k+1}}% g_{C_{r_k}} g_{C_{r_{k+1}}}
 F_{r_k+2} \left(\gamma\left(t_{n(r_k+2)}\right)\right)\\ 
	& \in G_{r_k} G_{r_{k+1}} %g_{\ab{n(r_k)}}\circ\dots\circ g_{\ab{n(r_k+2)-1}}
	 \left(
	\phi\left(\cK_N\right) \right)
	    =  G_{r_{k+1}} 
		\left(		\phi(\cK_N)\right)
		+M\mu,\\
	\end{align*}
where the last line follows observing that that, by assumption (1) $C_{r_k}=\eta^M$,  $G_{r_k}(z)=z+M\mu$. 

Now, since by assumptions, the cuspidal words $C_{r_k-2}$, $C_{r_k-1}$ and $C_{r_k+1}$ all have length at most $N$, they correspond to elements of $\Gamma$ whose norm is uniformly bounded. Moreover the image of the compact set $\phi(\cK_N)$ under such elements of $\Gamma$ is contained in a bigger compact subset of $\HH$, whose size depends only on $N$.
It follows that there exists $\epsilon_N>0$, depending only on $N$, such that for any $k\geq1$ we have
\[
	\Im \left(F_{r_k}\left(\gamma(t_{n(r_k-1)})\right)\right)\geq\epsilon_N
	\quad
	\text{and }
	\quad
	\Im \left(F_{r_k}\left(\gamma(t_{n(r_k+2)})\right)\right)\geq\epsilon_N. 
\]
By applying the observation at the beginning of this step to
$
\gamma':=F_{r_k}\circ\gamma
$ 
and $a=t_{n(r_k-1)}$ and $b=t_{n(r_k+2)}$, this shows that
\[
	\Im \bigl(F_{r_k}(\gamma(t))\bigr)\geq\epsilon_N
	\quad
	\text{ for any }
	\quad
	t_{n(r_k-1)}\leq t\leq t_{n(r_k+2)}.
\]
Recalling that $h$ is $\Gamma$-periodic, and thus in particular periodic under the translation $z\mapsto z+\mu$, set 
\[
	C_2=C_2(N,h,l):=\max_{\epsilon_N\leq \Im(z)\leq l}h(z)<+\infty.
\]

\smallskip
\emph{Step 3. Lower bound on the supremum on special segments.} 

Let us recall that the block $C_{r_k-1} C_{r_k} C_{r_k+1}$ codes the geodesic $\gamma$ from time $t=t_{n(r_k-1)}$ up to time $t=t_{n(r_k+2)}$.
Moreover the central word $C_{r_k}$ corresponds to $M$ iterations of the parabolic transformation $p(z)= z+\mu$.
Hence, as in the proof of \cref{thm:Hall}, we see that the renormalized geodesic $F_{r_k}\circ\gamma$ crosses exactly $M$ vertical lines of the form $\cV_j:=\{z\in\HH: \Re(z)=j\frac{\mu}{2}, j\in\ZZ\}$ in the upper half plane, see \cref{fig:estimatesHall}.
It follows that 
\[
	\sup_{t_{n(r_k-1)}\leq t\leq t_{n(r_{k}+2)}} \Im\bigl(F_{r_k}(\gamma(t))\bigr) = \sup_{t\in\RR}\Im\bigl(F_{r_k}(\gamma(t))\bigr) \geq  (M-1)\frac{\mu}{2}.
\]
We now choose $M$ such that
\[
	(M-1)\frac{\mu}{2} \geq\max\left\{l,C_1(N,h)+\delta,C_2(N,h,l)+\delta\right\},
\]
with $\delta=\delta_G$, defined in~\cref{thm:Hallperturbations}.
It follows that there exists some $t^{(k)}$ with $t_{n(r_k-1)}\leq t^{(k)}\leq t_{n(r_{k}+2)}$ such that $F_{r_k}(\gamma(t^{(k)}))\in \cU_l$ and, moreover, $\Im\bigl(F_{r_k}(\gamma(t^{(k)}))\bigr)\geq (M-1)\mu/2$. 
Since $|h(z)-\Im(z)|<\delta$ for any $z\in\cU_l$, then for such $t^{(k)}$ we have
\begin{equation}
\label{Equation2PropositionFormulaContinuedFraction}
	h\bigl(F_{r_k}(\gamma(t^{(k)}))\bigr)\geq
	\Im\bigl(F_{r_k}(\gamma(t^{(k)}))\bigr)-\delta\geq
	(M-1)\frac{\mu}{2}-\delta\geq
	\max
	\left\{C_1(N,h),C_2(N,h,l)\right\}.
\end{equation}

\smallskip
\emph{Step 4. Final arguments.} 
We can now conclude the proof.  From Equation~\eqref{Equation1PropositionFormulaContinuedFraction} and~\eqref{Equation2PropositionFormulaContinuedFraction} it follows that
\[
	%\begin{split}
		\limsup_{t\to+\infty}h\bigl(\gamma(t)\bigr)
		%	&
			=\limsup_{r\to+\infty} \sup_{t_{n(r)}\leq t\leq t_{n(r+1)}} h\bigl(F_r(\gamma(t))\bigr).
%\end{split}
\]
Now, Equation~\eqref{Equation2PropositionFormulaContinuedFraction} also implies that the large values of $h\bigl(F_{r_k}(\gamma(t))\bigr)$ are always taken when $F_{r_k}(\gamma(t))\in \cU_l$. Moreover, we claim that $h_l\bigl(F_{r_k}(\gamma(t))\bigr)=0$ for $t\notin[t_{n(r_k-1)},t_{n(r_{k}+2)}]$. 
In fact, we recall that the fundamental horodisk $\cU_l$ is precisely invariant, meaning that for each $g\in G$ we either have $g(\cU_l)=\cU_l$ or $\cU_l\cap g(\cU_l)=\emptyset$, and the former happens only if $g$ is a power of $p$.
Hence, for any $t\notin[t_{n(r_k-1)},t_{n(r_{k}+2)}]$, as we can assume without loss of generality that $l\geq 1$, there exists some non parabolic $g\in\Gamma$ with $F_{r_k}(\gamma(t))\in g(\cU_l)$ and $\cU_l\cap g(\cU_l)=\emptyset$.
Thus, we get that
\[
	%\begin{split} 
		\sup_{t_{n(r)}\leq t\leq t_{n(r+1)}} h\bigl(F_r(\gamma(t))\bigr)
		%	&
			= \sup_{t_{n(r_k-1)}\leq t\leq t_{n(r_{k}+2)}} h_l \, \bigl(F_{r_k}(\gamma(t))\bigr)%\\
		%	&
			= \sup_{t\in\RR} h_l\bigl(F_{r_k}(\gamma(t))\bigr).
%\end{split}
\]
Finally, recalling the definition of the function $H(\cdot,\cdot)$ and the expression of the endpoints of the geodesic $F_{r_k}\circ\gamma$, we get
\[
	\begin{split} 	 \sup_{t\in\RR} h_l\bigl(F_{r_k}(\gamma(t))\bigr)
			&= H\left(F_{r_k}(\gamma(-\infty)),F_{r_k}(\gamma(+\infty))\right)\\
			&= H\left(\left[\ab{n(r_k)-1},\ab{n(r_k)-2},\dots\right]_{\partial\HH}^-,
	\left[\ab{n(r_k)},\ab{n(r_k)+1},\dots\right]_{\partial\HH}
	\right).
\end{split}
\]
Combining all the last series of equalities together, the proof is hence concluded.
\end{proof}

\subsection{Proof of \cref{thm:Hallperturbations}}\label{subsec:proofHallperturbations}
We have all the ingredients in order to give the proof of \cref{thm:Hallperturbations} following the same scheme of the proof of \cref{thm:Hall}. 

%\smallskip
\begin{proof}[Proof of \cref{thm:Hallperturbations}]
Let us first assume that $m=1$ and prove the result under the Lipschitz condition~\eqref{Lip_assumption} in \cref{m1}. We will then show at the end how to deduce the result for other values of $m$ from this special case. 
Let us first verify that we can apply \cref{thm:sumofcantors} and \cref{thm:decompositionperturbations}. Assume without loss of generality that  $l_0$ in the assumptions of \cref{thm:Hallperturbations} is greater than $1$.  Thus, from the assumption~\eqref{Lip_assumption} on the Lipschitz control of the perturbation on $\cU_{l_0}$, by \cref{PropositionLipschitzNorm},
\[
	\lipschitz \bigl( (H-H_0)|_{U_{l_0}} \bigr)
	\leq 
	\left(\sqrt{2}+\frac{\sqrt{2}}{l_0}\right)\cdot\|(h-\Im)|_{\cU_{l_0}}\|_{\text{Lip}} \leq \frac{2 \sqrt{2}}{4\sqrt{2}}=\frac{1}{2}. 
\]
Moreover, since by \cref{LemmaUniformNorm}  $\|(H-H_0)|_{U_{l_0}}\|_\infty \leq \|(h-\Im)|_{\cU_{l_0}}\|_\infty \leq  \|(h-\Im)|_{\cU_{l_0}}\|_{\text{Lip}}< 1/4$.  
Thus, there exists  $0<\epsilon<1$ such that the assumptions of \cref{thm:sumofcantors} and \cref{thm:decompositionperturbations}  hold for the function $H$ corresponding to $h$ and $l=l_0$.  

\smallskip
Let $s_0$ be given by \cref{thm:sumofcantors} and let  $N_0$ be as  in \cref{thm:decompositionperturbations}. Let also $M_0:=M(l_0,h,N_0)$ be given by \cref{PropositionFormulaContinuedFraction} in correspondence to $N_0$. We will show that
\[
[L_0,+\infty) \subset \cL(X,h), \quad \text{ where } L_0:= \max \{ \inf H(\KK_N\times\KK_N^{s_0}),\inf H(\KK_N\times\KK_N^{M_0})\}.
\]
Take any $L\geq L_0$. By \cref{thm:decompositionperturbations} (with $s_1= \max\{M_0, s_0\}$),   there exist points
$
x_2=[a_0,a_1,\dots]_{\partial\HH}\in\KK_N
$ 
and 
$
x_1=[b_0,b_1,\dots]_{\partial\HH}\in\KK_N
$ 
and an integer $s$ such that 
\begin{equation}\label{LwithH}
L=H(x_1,x_2+s\mu), \qquad s\geq s_1.
\end{equation}
We  now construct a geodesic $\gamma$ such that $L(\gamma,h)=L$, by prescribing its symbolic coding $(\cc{n})_{n \in \ZZ}$.  The construction is the same that the one in the proof of  \cref{thm:Hall}, so we only sketch it to avoid unnecessary repetitions.  We construct the word  $(\cc{n})_{n \in \ZZ}$ by concatenating blocks of the form $W_{j} = {\overline{\bb{|j|}}} \dots {\overline{\bb{0}}} \eta^{s} \ab{0} \dots \ab{|j|} $, $j \in \ZZ$,  interpolated via letters $\delta_j$, $\delta'_j$ chosen exactly as in the proof of \cref{thm:Hall}, so that in particular  $(\cc{n})_{n \in \ZZ}$ satisfies the non-backtracking condition~\eqref{eqnobacktrack} and hence is the cutting sequence of a geodesic $\gamma$. Recall also that the central block $\eta^{s}$ in the word $W_k$ is, by construction, a single parabolic word.

The assumptions of \cref{PropositionFormulaContinuedFraction} apply by construction to the geodesic $\gamma$, by letting $(r_k)_{k\in  \ZZ}$ the sequence such that the parabolic word $C_{r_k}=\eta^{s}$ is the central block of $W_k$.
In fact, Condition (1) is obvious since $s\geq M_0$ by construction, and the distance between $r_k$ and $r_{k-1}$ grows linearly; Condition (2) follows since   $(a_n)_{n\in \NN}$ and $(b_n)_{n\in \NN}$ code points in $\KK_N$ and by choice of the interpolating letters, we refer to the proof of \cref{thm:Hall} for details.

Thus,  \cref{PropositionFormulaContinuedFraction}, the form of the word $(\cc{n})_{n \in \ZZ}$ and~\eqref{LwithH}, give that $L(h,\gamma)=L$.  
This concludes the proof under the assumptiont that $m=1$.

\smallskip
Finally, let us deal with the case when $m\neq 1$.
We can conjugate the group $G$ with an element $g\in\PSL(2,\RR)$ that fixes infinity and that normalizes $m$ to $1$.
In particular we can choose $g(z)=z/m$. Denote by
$$h':= h \circ g^{-1}, \qquad \Im' =\Im \circ g^{-1}=m\cdot\Im.$$
If we set $G':=gGg^{-1}$, there is a one to one correspondence between geodesics $\gamma$ on $X=G\backslash\HH$ and geodesics  $\gamma'$ on $X'=G'\backslash\HH$,  given by $\gamma'(t)=g(\gamma(t))$. 
Using this observation and recalling the definition~\eqref{def:limsupdef} of Lagrange values, we have that
\begin{equation}\label{eq:twospectra}
	L_{G'}(h',\gamma')=\limsup_{t\to\infty} h'(\gamma'(t)) = \limsup_{t\to\infty} h\circ g^{-1}(g(\gamma(t))= \limsup_{t\to\infty} h(\gamma(t)) = L_G(h,\gamma).
\end{equation}
%where we have stressed the dependence on the groups $G$ and $G'$.
The formula implies that the two corresponding spectra coincide, that is, $\cL(X,h)=\cL(X',h')$.
Thus, it is now enough to show that the group $G'$ and the function $h'$ satisfy the assumptions of \cref{thm:Hallperturbations} with $m=1$.
We begin by observing that a point $z\in\cU_l$ if and only if $g^{-1}(z)\in\cU_{lm}$.
This implies that, for any $l>0$,
\[
\|(h'-\Im')|_{\cU_l}\|_\infty 
		= \sup_{x\in\cU_l}|((h-\Im)\circ g^{-1})(x)|
		= \sup_{y\in\cU_{lm}}|((h-\Im))(y)|
		= \|(h-\Im)|_{\cU_{lm}}\|_\infty.
\]
Similarly we have that
\[
\begin{split}
	\lipschitz((h'-\Im')|_{\cU_l}) 
	&= \sup_{x,y\in\cU_l} \frac{|(h'-\Im')(x)-(h'-\Im')(y)|}{|x-y|}\\
	&= \sup_{x,y\in\cU_l} \frac{|((h-\Im)\circ g^{-1})(x)-((h-\Im)\circ g^{-1})(y)|}{|x-y|}\\
	&= \sup_{x',y'\in\cU_{lm}} \frac{|(h-\Im)(x')-(h-\Im)(y')|}{|g(x')-g(y')|}\\
	&= \sup_{x',y'\in\cU_{lm}} \frac{|(h-\Im)(x')-(h-\Im)(y')|}{\frac{1}{m}|x'-y'|} =m\cdot\lipschitz((h-\Im)|_{\cU_{lm}}).
\end{split}
\]
Thus, the computations above show that 
\[
	\|(h'-\Im')|_{\cU_l}\|_{\text{Lip}} = \|(h-\Im)|_{\cU_{lm}}\|_\infty + m \lipschitz((h-\Im)|_{\cU_{lm}})\leq \max\{1,m\}\cdot \|(h-\Im)|_{\cU_{lm}}\|_{\text{Lip}}.
\]
 %Thus $\|(h'-\Im')|_{\cU_l}\|_{\text{Lip}}\leq cdot\|(h-\Im)|_{\cU_{lm}}\|_{\text{Lip}}$.
Let us now assume that, for some $l_0\geq m >0$, we have
\[
	\|(h-\Im)|_{\cU_{l_0}}\|_{\text{Lip}} <\delta_G:=\min \biggl\{ \frac{1}{ 4 m \sqrt{2}}, \frac{1}{ 4  \sqrt{2}}\biggr\},
\]
which yields
\[
	\|(h'-\Im')|_{\cU_{\frac{l_0}{m}}}\|_{\text{Lip}} \leq \max\{1,m\}\cdot \|(h-\Im)|_{\cU_{l_0}}\|_{\text{Lip}}< \max\{1,m\} \cdot \delta_G = \frac{1}{4\sqrt{2}},
\]
where the last equality follows considering separately the cases $m\leq 1$ and $m\geq 1$. 
%If $m\leq 1$, 
%\[
%	\|(h'-\Im')|_{\cU_{\frac{l_0}{m}}}\|_{\text{Lip}} \leq \|(h-\Im)|%_{\cU_{l_0}}\|_{\text{Lip}} < \frac{1}{ 4  \sqrt{2}}.
%\]
%On the other hand, if $m\geq 1$, 
%\[
%	\|(h'-\Im')|_{\cU_l}\|_{\text{Lip}} \leq m \|(h-\Im)|_{\cU_l}\|%_{\text{Lip}} < \frac{1}{ 4   \sqrt{2}}.
%\]
%Thus, in both cases,~\eqref{assumptionprime} holds, so the first part of the proof implies that $\cL(X',h')$ contains a Hall ray.
Thus, the first part of the proof implies that $\cL(X',h')$ contains a Hall ray.
Thanks to~\eqref{eq:twospectra}, this shows also that $\cL(X,h)$ contains a Hall ray and hence concludes the proof in the general case.
\end{proof}

\section{Proof of the Stable Hall theorem}\label{sec:proof_Stable_Hall}

In this section we prove Theorem \cref{thm:stableHall} stated in the introduction (see \cref{SectionStableHallTheorem}), which generalizes Hall's theorem on the sum of Cantor sets to Lipschitz perturbations of the sum. Throughout the section we use the notation introduced in \cref{SectionStableHallTheorem}.

\smallskip
%\subsection{Preliminary results}
Consider any Cantor set $\KK$ and let $(\KK(n))_{n\in\NN}$ be a slow subdivision for $\KK$. By \cref{rem:orderedholes}, the collection of holes of $\KK$ inherits from the subdivision an ordering. We will denote by $B_n$ the $n^{\text{th}}$ hole, so that $(B_n)_{n\in\NN}$ is the collection of holes of $\KK$ ordered according to $(\KK(n))_{n\in\NN}$.  We say that $(\KK(n))_{n\in\NN}$ is a \emph{monotone slow subdivision} for $\KK$  if the ordered sequence of holes $(B_n)_{n\in\NN}$ satisfies
\[
	|B_{n+1}|\leq |B_n|
	\quad
	\text{ for any }
	\quad
	n\in\NN.
\]
It is clear that  monotone slow subdivisions always exist.

Let us now state two preliminary Lemmas which will be used in the proof of
 \cref{thm:stableHall}
%In the proof of the Theorem we will use the two  below.

\begin{lemma}\label{LemGapConditionForGeometricSubdivision}
Let $(\KK(n))_{n\in\NN}$ be a monotone slow subdivision for the Cantor set $\KK$. 
If $\KK$ admits another slow subdivision $(\widetilde{\KK}(n))_{n\in\NN}$ which satisfies the $\epsilon$-stable gap condition~\eqref{EqStableGapCondition}, then the same is true for the monotone slow subdivision $(\KK(n))_{n\in\NN}$. 
\end{lemma}
This Lemma was proved as Lemma A.1 in the Appendix of~\cite{AMU} (see also~\cite{CusickFlahive}). 

\begin{rem}\label{rem:intervals}
Observe that if $K$ and $F$ are closed intervals then $K\times F$ is connected and since $S\colon U\to \RR$ is continuous then the image $S(K\times F)$ is connected too, that is it is an interval. Moreover, for the same reason, if $K$ and $F$ are closed intervals, then $K\times F$ is compact and thus its image $S(K\times F)$ is compact, and thus closed.
\end{rem}

The next Lemma provides the key step to prove  the Stable Hall theorem. 
\begin{lemma}\label{LemSumOfIntervals}
Let $S\colon U\to\RR$ be a function satisfying Condition~\eqref{EqConditionStableHallTheorem}.
Let $K$ and $F$ be two compact intervals with $K\times F\subset U$.
Let $B$ be an open interval contained in $K$ such that $|B|<(1-\epsilon)|F|$.
Then we have
\[
	S(K\times F)=S(K^L\times F)\cup S(K^R\times F).
\]
Similarly, if $C$ is an open interval contained in $F$ with $|C|\leq (1-\epsilon)|K|$ then we have
\[
	S(K\times F)=S(K\times F^L)\cup S(K\times F^R).
\]
\end{lemma}

\begin{proof}
We only prove the first statement, the argument for the second being the same.
Set
\[
	G:=S-S_0.
\]
Let $K=[a,b]$, $F=[c,d]$ and $B=(e,f)\subset K$ for real numbers 
$$a<e<f<b,  \qquad c<d.$$ 
Let us first show that
\begin{equation}\label{inf:ineq}
	\inf S(K\times F)=\inf S(K^L\times F).
\end{equation}
Since the inequality $\leq$ is obvious, it is enough to prove the inequality $\geq$. Moreover, since $K=K^L \cup [e,b]$, it is enough to show that 
\begin{equation}\label{inf:ineq2}
	\inf S([e,b]\times F)\geq \inf S(K^L\times F).
\end{equation}
To prove this, consider any  $(x_1,x_2)$ with $e\leq x_1\leq b$ and $x_2 \in F$, i.e.\ $c\leq x_2\leq d$ and $(a,c)$, which belongs to $K^L\times F$.
Recalling the definition of Lipschitz constant~\eqref{def:Lip_const} and using that $
\lipschitz(G)<1$ and $x_1\geq a$, $x_2\geq c$, we have that 
%then for any $x_1,x_2$ with $e\leq x_1\leq b$ and $c\leq x_2\leq d$ 
\[
\begin{split}
	S(x_1,x_2)-S(a,c)	&=S_0(x_1,x_2)-S_0(a,c)+G(x_1,x_2)-G(a,c)\\
						&=(x_1+x_2)-(a+c)+G(x_1,x_2)-G(a,x_2)+G(a,x_2)-G(a,c)\\
							&=|x_1-a|+|x_2-c|+\frac{G(x_1,x_2)-G(a,x_2)}{x_1-a}(|x_1-a|)+\frac{G(a,x_2)-G(a,c)}{x_2-c}(|x_2-c|)\\
						&\geq |x_1-a|\bigl(1-\lipschitz(G)\bigr)+|x_2-c|\bigl(1-\lipschitz(G)\bigr)>0.
\end{split}
\]
This proves \cref{inf:ineq2} and hence concludes the proof of \cref{inf:ineq}. With a similar argument, we get also that 
\[
	\sup S(K\times F)=\sup S(K^R\times F).
\]
Since $S(K\times F)$, $S(K^L\times F)$ and $S(K^R\times F)$ are three intervals (see \cref{rem:intervals}),  it is hence enough to prove that
$$
\sup S(K^L\times F) \geq \inf S(K^R\times F).
$$
To show this, we will show that  $S(e,d)>S(f,c)$ (remark that $(e,d) \in K^L\times F$ and $(f,c)\in K^R\times F$). Reasoning as before, we have
\[
\begin{split}
	S(e,d)-S(f,c)	&=S_0(e,d)-S_0(f,c)+G(e,d)-G(f,c)\\
					&=(d-c)-(f-e)+G(e,d)-G(e,c)+G(e,c)-G(f,c)\\
					&\geq|d-c|\bigl(1-\lipschitz(G)\bigr)-|f-e|\bigl(1+\lipschitz(G)\bigr)
\end{split}
\]
%\todo[inline]{Prima qui sopra c'era un $1-Lip$ che invece ho corretto con $1+Lip$}
so that $S(e,d)>S(f,c)$ is implied by
\[
	|B|=|f-e|<
	\frac{1-\lipschitz(G)}{1+\lipschitz(G)}\cdot|d-c|=
	\frac{1-\lipschitz(G)}{1+\lipschitz(G)}\cdot|F|,
\]
which is satisfied according to Condition~\eqref{EqConditionStableHallTheorem}, because 
$
|B|<(1-\epsilon)|F|
$ 
by assumption.
\end{proof}

We can now give the Proof of \cref{thm:stableHall}.
\begin{proof}[Proof of \cref{thm:stableHall}]
Let $\bigl(\KK(n)\bigr)_{n\in\NN}$ and $\bigl(\FF(n)\bigr)_{n\in\NN}$ be slow monotone subdivisions respectively for $\KK$ and $\FF$.
Since by assumption  $\KK$ and $\FF$ admit a slow subdivision which satisfy Condition~\eqref{EqStableGapCondition}, then by \cref{LemGapConditionForGeometricSubdivision} the same is true for the subdivisions $\bigl(\KK(n)\bigr)_{n\in\NN}$ and 
$\bigl(\FF(n)\bigr)_{n\in\NN}$. 
Set $K_0:=[\min\KK,\max\KK]$ and $F_0:=[\min\FF,\max\FF]$ and fix any point $x\in S(K_0\times F_0)$.

The Theorem follows if we show that we can construct two sequences $(n_i)_{i \in \NN}$ and  $(m_i)_{i \in \NN}$ such that $m_{i+1}\geq m_i$ and $n_{i+1}\geq n_i$ for any $i\in\NN$ and $n_i \to \infty$, $m_i \to \infty$, and two sequences of nested closed intervals $(K_i)_{i\in\NN}$ and $(F_i)_{i\in\NN}$, where $K_i$ is an interval of the level $\KK(n_i)$ and $F_i$ is an interval of the level $\FF(m_i)$, such that for any $i \in \NN$ we have
\[
	x\in S(K_i\times F_i).
\]
Indeed setting $k:=\bigcap_{i\in\NN}K_i$ and $f:=\bigcap_{j\in\NN}F_j$ continuity of $S$ implies $x=S(k,f)$, where $k\in\KK$ and $f\in\FF$.
Observe that we require $n_i \to \infty$, but steps $i$ for which $n_{i+1}=n_i$ are allowed, and similarly for the integers $m_i$.

We will construct the sequences $(n_i)_{i \in \NN}$ and  $(m_i)_{i \in \NN}$ and the  two families of nested intervals by induction on  $i$  in $\NN$.
Fix $i\in\NN$ and assume that respectively the first $i+1$ nested intervals $K_0\supset K_1\supset\dots\supset K_i$ and the first $i+1$ nested intervals 
$F_0\supset F_1\supset\dots\supset F_i$ are defined.
Let $n(K_i)$ be the minimum $n\in\NN$ such that $K_i\cap\KK(n)\neq K_i $ and let $B_i$ be the hole in $K_i$, i.e.\ the  open subinterval $B_i\subset K_i$ such that $\KK(n(K_i))\cap K_i=K_i\setminus B_i$. 
Similarly, let $m(F_i)$ be the minimum $m\in\NN$ such that $F_i\cap\FF(m)\neq F_i$ and let $C_i$ be the hole in $F_i$, i.e.\ the  open subinterval $C_i\subset F_i$ such that $\FF(m(F_i))\cap F_i=F_i\setminus C_i$.

During the inductive construction, we will also prove that for every $i$ the intervals $(K_i, F_i)$ and the holes $B_i \subset K_i$, $C_i \subset F_i$ in our construction satisfy the following \emph{balanced gap} condition:
\begin{equation}\label{EqBalanceCondition}
	|B_i|<(1-\epsilon)|F_i| 
	\quad
	\text{ and }
	\quad
	|C_i|<(1-\epsilon)|K_i|.
\end{equation}
Observe that for $i=0$ the condition is true according to the $\epsilon$-size condition~\eqref{EqStableSizeCondition}.
Assume that balanced gap condition~\eqref{EqBalanceCondition} is satisfied for $i\geq0$.
To define the intervals at level $i+1$, we subdivide the interval having the bigger hole.
Assume that $|B_i|\geq|C_i|$, the other case being the same.
Since $|B_i|<(1-\epsilon)|F_i|$ then \cref{LemSumOfIntervals} implies 
\[
	S(K_i\times F_i)=S(K_i^L\times F_i)\cup S(K_i^R\times F_i).
\]
If $x\in S(K_i^L\times F_i)$ (respectively $x\in S(K_i^R\times F_i)$), set $K_{i+1}:=K_i^L$ (respectively  $K_{i+1}:=K_i^R$) and $n_{i+1}=n(K_i)$, so that $K_{i+1}\in\KK(n_{i+1})$. Set also $F_{i+1}=F_i$ and $m_{i+1}=m_i$, so that $F_{i+1}\in\FF(m_{i+1})$ holds trivially.
By the property of a monotone slow subdivision, the hole $B_{i+1}\subset K_{i+1}$ satisfies $|B_{i+1}|\leq|B_i|$ and therefore by inductive assumption we get 
\[
	|B_{i+1}|\leq|B_i|<(1-\epsilon)|F_i|=(1-\epsilon)|F_{i+1}|.
\]
On the other hand Condition~\eqref{EqStableGapCondition} implies $|B_i|<(1-\epsilon)|K_i^L|=(1-\epsilon)|K_{i+1}|$ and therefore, since $C_i$ is by choice the smaller of the two holes, we get 
\[
	|C_i|\leq|B_i|<(1-\epsilon)|K_{i+1}|.
\]
Thus, the pair of intervals $(K_{i+1},F_{i+1})$, with holes $B_{i+1}$ and $C_{i+1}=C_i$ satisfies balanced gap condition~\eqref{EqBalanceCondition}, %as the pair $(K_i,F_i)$ with holes $B_i$ and $C_i$, 
and moreover we have $x\in S(K_{i+1}\times F_{i+1})$. The inductive step is complete. 

Finally, since there are only finitely many holes which are longer than a given positive constant, it is clear that both $(n_i)_{i \in \NN}$ and $(m_i)_{i \in \NN}$ satisfy $n_i\to\infty$ and $m_i\to\infty$. \cref{thm:stableHall} is proved. 
\end{proof}

\appendix

\section{Proofs of Lemmas on parabolic words}\label{app1}
We present here the simple proofs of \cref{LemmaCombinatorialPropertiesCuspidal} and \cref{Parabolicwordslemma} on parabolic words in \cref{subsec:parabolicwords}. With a slightly different notation, the proofs were essentially  contained in~\cite{AMU}.  
% that give a combinatorial description of cuspidal words, by showing that cuspidal words are obtained by repearing \emph{parabolic words} (defined below) which are in one to one correspondence with cusps. The Lemmas are essentially proved in~\cite{AMU}. For compleness, we include their easy proofs in Appendix \ref{app1}. 

\begin{proof}[Proof of \cref{LemmaCombinatorialPropertiesCuspidal}]
In order to prove point (1), observe that for any $k=0,\dots,n-1$ we have
\[
\begin{split}
	\Arc[\ab{0},\dots,\ab{k}] &=g_{\ab{0}}\circ\dots\circ g_{\alpha_{k-1}}\Arc[\ab{k}],\\
	\Arc[\ab{0},\dots,\ab{k},\ab{k+1}] &= g_{\ab{0}}\circ\dots\circ g_{\alpha_{k}}\Arc[\ab{k+1}].
\end{split}
\]
Applying $\bigl(g_{\ab{0}}\circ\dots\circ g_{\ab{k-1}}\bigr)^{-1}$ and recalling that elements of $\Gamma$ preserve the orientation on $\partial\DD$, we see that $\Arc[\ab{0},\dots,\ab{k}]$ and $\Arc[\ab{0},\dots,\ab{k},\ab{k+1}]$ 
share the same left endpoint if and only if 
\[
	\xi^l_{\ab{k}}= \inf \Arc[\ab{k}]= g_{\ab{k}}\bigl(\inf \Arc[\ab{k+1}]\bigr)=
	g_{\ab{k}}(\xi^l_{\ab{k+1}}).
\]
Point (2) for right endpoints follows with the same argument.
In order to prove point (3), recall that, for any letter $\alpha$, $\xi^l_{\alpha}=g_{\alpha}(\xi^r_{\overline{\alpha}})$ and $g_\alpha^{-1}=g_{\overline{\alpha}}$.
Therefore, according to first two points, $\ab{0}\dots\ab{n}$ is left cuspidal if and only if for any $k=0,\dots,n-1$ we have 
\[
	g_{\ab{k}}(\xi^l_{\ab{k+1}})=\xi^l_{\ab{k}}\iff
	g_{\ab{k}}g_{\ab{k+1}}(\xi^r_{\overline{\ab{k+1}}})=g_{\ab{k}}(\xi^r_{\overline{\ab{k}}})\iff
	\xi^r_{\overline{\ab{k+1}}}=g_{\ab{k+1}}^{-1}(\xi^r_{\overline{\ab{k}}})=
	g_{\overline{\ab{k+1}}}(\xi^r_{\overline{\ab{k}}}),
\]
that is $\overline{\ab{n}}\dots\overline{\ab{0}}$ is right cuspidal.
\end{proof}

\begin{proof}[Proof of \cref{Parabolicwordslemma}]
Point (1) of \cref{LemmaCombinatorialPropertiesCuspidal} implies $\xi^l_{\ab{k}}=g_{\ab{k}}(\xi^l_{\ab{k+1}})$ 
for any $k=0,\dots,n-1$, thus
\[
	\xi^l_{\ab{0}}=g_{\ab{0}}\circ\dots\circ g_{\ab{k}}(\xi^l_{\ab{k+1}}).
\]
If $\ab{0}\dots\ab{n}\ab{0}$ is also cuspidal then the condition above holds with $k=n$ 
and therefore 
\[
	g(\xi^l_{\ab{0}})=\xi^l_{\ab{0}}.
\]
Since $\ab{0}\dots\ab{n}\ab{0}$ is left cuspidal, then also $\ab{n}\ab{0}\dots\ab{n}\ab{0}$ 
is so, and finally $\ab{n}\ab{0}\dots\ab{n}$ is left cuspidal too.
Hence, we have
\[
	\xi^l_{\ab{0}}=g_{\ab{n}}^{-1}(\xi^l_{\ab{n}})=g_{\ab{n}}^{-1}\circ g_{\ab{n}}(\xi^r_{\overline{\ab{n}}})=
\xi^r_{\overline{\ab{n}}}.
\]

According to points (2) and (3) of \cref{LemmaCombinatorialPropertiesCuspidal}, the word $\overline{\ab{n}}\dots\overline{\ab{0}}\overline{\ab{n}}$ is right cuspidal and reasoning as above we have that
\[
	g^{-1}(\xi^r_{\overline{\ab{n}}})= g_{\overline{\ab{n}}}\circ\dots\circ g_{\overline{\ab{0}}}(\xi^r_{\overline{\ab{n}}})= \xi^r_{\overline{\ab{n}}}.
\]

Observe also that 
\[
\begin{split}
	g\Arc[\ab{0}]&=g_{\ab{0}}\circ\dots g_{\ab{n}}\Arc[\ab{0}]= \Arc[\ab{0},\dots,\ab{n},\ab{0}]
\subset \Arc[\ab{0}],\\
g^{-1}\Arc[\overline{\ab{n}}]&=g_{\overline{\ab{n}}}\circ\dots g_{\overline{\ab{0}}}\Arc[\overline{\ab{n}}]=
\Arc[\overline{\ab{n}},\dots,\overline{\ab{0}},\overline{\ab{n}}]\subset \Arc[\overline{\ab{n}}].
\end{split}
\]
Thus $\xi^l_{\ab{0}}$ is a fixed point of $g$ and $\Arc[\ab{0}]$ is a right neighborhood of it where $g$ is contracting.
On the other hand $\Arc[\overline{\ab{n}}]$ is a left neighborhood of $\xi^l_{\ab{0}}$ where $g^{-1}$ acts contracting, thus $g$ is expanding.
It follows that $\xi^l_{\ab{0}}$ is not hyperbolic and is thus the unique fixed point of $g$.
Thus $g$ is a parabolic element of $\Gamma$. 
\end{proof}

\section{Lipschitz norm estimates}\label{app2}
In this Appendix, we present the proof of \cref{PropositionLipschitzNorm}. 

We use in this Appendix the notation introduced in the beginning of \cref{sec:Hallperturbations}. We recall, in particular that, for $l>0$, $\cU_l\subset\HH$ is a horocyclic neighborhood   and $U_l\subset\RR^2$ a the diagonal neighborhood. In order to simplify the notation, we write simply $H$ and $H_0$ instead of $H|_{U_l}$ and $H_0|_{U_l}$ respectively and simply $h$ and $\Im$ instead of $h|_{\cU_l}$ and $\Im|_{\cU_l}$ respectively. For $x_1,x_2$ in $\RR$ let $\gamma(x_1,x_2,\cdot)\colon\RR\to\HH$, $t\mapsto\gamma(x_1,x_2,t)$ be the geodesic parametrization of the hyperbolic geodesic $\gamma(x_1,x_2)$ in $\HH$ with endpoints $x_1,x_2$, such that $\gamma(x_1,x_2,0)$  is its highest point, i.e.\
\begin{equation}\label{parametrization}
\lim_{t\to-\infty}\gamma(x_1,x_2,t)=x_1,
\quad \quad
\lim_{t\to+\infty}\gamma(x_1,x_2,t)=x_2,
\quad \quad
\gamma(x_1,x_2,0)=\frac{x_1+x_2}{2}+i\frac{x_2-x_1}{2}.
\end{equation}
%For any $t\in\RR$ we have explicitly
%\[
%\gamma(x_1,x_2,t)=
%\frac{x_2e^{2t}+x_1e^{-2t}}{e^{2t}+e^{-2t}}+
%i\frac{x_2-x_1}{e^{2t}+e^{-2t}}.
%\]
Finally set 
$
\delta:=\|h-\Im\|_{\text{Lip}}
$ 
and recall that we have
\[
\sup_{z\in\cU_l}|h(z)-\Im(z)|\leq\delta
\quad
\textrm{ and }
\quad
\sup_{z,z'\in\cU_l}
\frac
{\bigl|\bigl(h(z)-\Im(z)\bigr)-\bigl(h(z')-\Im(z')\bigr)|\bigr|}
{|z-z'|}
\leq\delta.
\]
We recall that the Lipschitz constant of any $G\colon U\subset \mathbb{R}^2 \to R$ is given by
\[
\sup_{(x_1,x_2),(x_1',x_2') \in U } \frac{	|G(x_1,x_2)-G(x'_1,x'_2)| }{	|(x_1,x_2)-(x'_1,x'_2)|}, 
\]
where $	|(y_1,y_2)| $ denotes the Euclidean norm of $	|(y_1,y_2)| \in \RR^2$ (and hence corresponds to the absolute value $|y_1+iy_2|$ in $\CC$).
\smallskip

The idea of the proof is to first estimate the Lipschitz constant of $H$ in two directions which are geometrically meaningful and hence easier to control. We remark indeed that if we consider a point $(x_1',x_2')$ of the form $ (x_1,x_2)+s (1,1) = (x_1+s,x_2+s)$, where $s \in \RR$, the geodesic $\gamma(x_1',x_2')$ is obtaining by \emph{sliding}  horizontally the  endpoints of $\gamma(x_1,x_2)$; in particular, the geodesics are rigidly translated. On the other hand, if we consider a point $(x_1',x_2')$ of the form $ (x_1,x_2)+s (-1,1) = (x_1-s,x_2+s)$, for small values of $s \in \RR$, the geodesic $\gamma(x_1',x_2')$, as a Euclidean semi-circle, is \emph{concentric} to  $\gamma(x_1,x_2)$, while the radii differ by $s$, see \cref{fig:lipschitz}.

\begin{figure}
\centering
\def\svgscale{0.8}
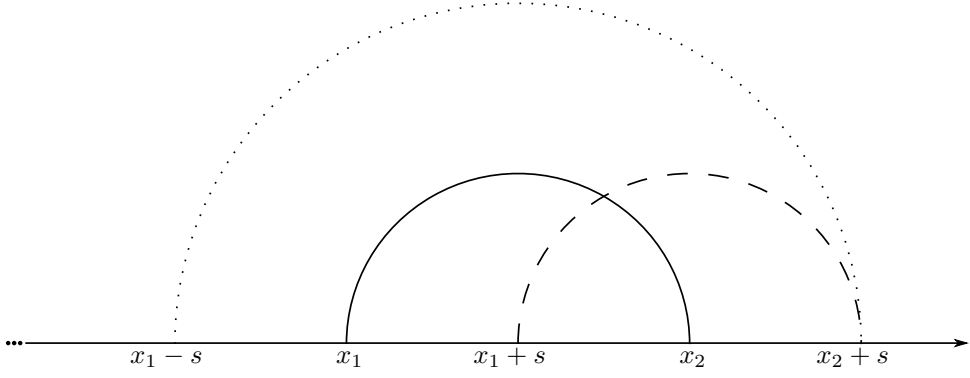
\caption{The geometric meaning of the vectors $(1,1)$ and $(-1,1)$ considered in the proof of \cref{PropositionLipschitzNorm}.}
\label{fig:lipschitz}
\end{figure} 

\begin{proof}[Proof of \cref{PropositionLipschitzNorm}]
Set $G:=H-H_0$.  Consider the vectors $v:=(1,1)$ and $w:=(-1,1)$ (for the  motivation explained before the proof). 
We will prove that, for any line  $\cV$  parallel to $v$ and  any line $\cW$ parallel to $w$, we have
\[
	\lipschitz(G|_\cV)\leq \frac{\delta}{\sqrt{2}}
	\quad
	\text{ and }
	\quad
    \lipschitz(G|_\cW)\leq \left(\frac{1}{\sqrt{2}}+\frac{\sqrt{2}}{r}\right)\delta,
\]
where $G|_\cV$ and $G|_\cW$ denote respectively the restriction of $G$ to the lines $\cV$ and $\cW$. We claim that this is enough to conclude, since for any $(x_1,x_2)$ and $(x_1',x_2')$ in $\RR^2$, there exists $\lambda,\mu$ in $\RR$ and $(x_1^\ast,x_2^\ast)\in\RR^2$ such that 
\[
	(x_1,x_2)-(x_1^\ast,x_2^\ast)=\lambda v
	\quad
	\text{ and }
	\quad
	(x_1^\ast,x_2^\ast)-(x'_1,x'_2)=\mu w,
\]
so that, remarking that $v$ and $w$ are orthogonal and hence one can use Phytagora theorem,
\[
	\begin{split}
	|G(x_1,x_2)-G(x'_1,x'_2)| 	&\leq |G(x_1,x_2)-G(x_1^\ast,x_2^\ast)| + |G(x_1^\ast,x_2^\ast)-G(x'_1,x'_2)|\\
								&\leq \frac{\delta}{\sqrt{2}}\cdot|(x_1,x_2)-(x_1^\ast,x_2^\ast)| +
									\left(\frac{1}{\sqrt{2}}+\frac{\sqrt{2}}{r}\right)\delta\cdot|(x_1^\ast,x_2^\ast)-(x'_1,x'_2)|\\
								&\leq \left(\sqrt{2}+
									\frac{\sqrt{2}}{r}\right)\delta\cdot|(x_1,x_2)-(x'_1,x'_2)|.
\end{split}
\]

We will consider separately the estimate for $G|_\cV$ and the one for $G|_\cW$. 
Before doing that, for any pair of points $(x_1,x_2)$ and $(x'_1,x'_2)$ in $\RR^2$,  let 
%$\gammma(t) := \Gamma(x_1,x_2,t)$, \qquad \gammma(t) := \Gamma(x_1,x_2,t)$
$$\gamma= \{ \gamma(t):=\gamma(x_1,x_2,t),\ t\in\RR\}, \qquad \gamma':=\{\gamma'(t):=\gamma(x'_1,x'_2,t), \ t\in\RR\},$$
%and let $\gamma(t)= \gamma \gamma(x'_1,x'_2,t) $ and $\gamma'(t)= \gamma'(x'_1,x'_2,t) $
 be the time parametrizations of the geodesics with respective endpoints $x_1,x_2$ and $x_1',x_2'$  described in~\eqref{parametrization}.
% ,t),t\in\RR\} and $\gamma':=\{\gamma(x'_1,x'_2,t),t\in\RR\}$$.
Let also $t_0,t'_0$ in $\RR$ be such that 
\[	H(x_1,x_2)=h\left(\gamma(t_0)\right)
	\quad
	\text{ and }
	\quad
	H(x'_1,x'_2)=h\left(\gamma'(t'_0)\right).
\]

\smallskip

\noindent \emph{Estimate for $G|_\cV$.} In order to prove the estimate for $G|_\cV$, where $\cV$ is any line parallel to $v$, consider any $(x_1,x_2)$ and $(x'_1,x'_2)$ in such line $\cV$ and let   
\[
s:=x'_2-x_2=x'_1-x_1, \qquad \text{so} \quad (x'_1,x'_2)= (x_1,x_2)+s (1,1).
\]
As we remarked before the proof, the geodesic  $\gamma'$ is hence obtained by sliding horizontally $\gamma$ by $s$, i.e.\ for every $t\in \RR$ we have $\gamma'(t)=\gamma(t)+s$. In particular, $\gamma(t_0)+s\in\gamma'$. 
% and $\gamma(,t_0)+s\in\gamma'$. 
Moreover, remark that since the function $z\mapsto \Im (z)$ is constant along horizontal lines, we have that  
$
\lipschitz \left( h|_\cV \right)=\lipschitz \left( \left(h-\Im \right) |_\cV \right) \leq \delta.
$
It follows from these two remarks that
\[
%\begin{split}
	H(x'_1,x'_2)	%&
	=h\bigl(\gamma'(t'_0)\bigr)
					\geq h\bigl(\gamma(t_0)+s\bigr)%\\&
					\geq h\bigl(\gamma(t_0)\bigr)-\delta\cdot|s|
					= H(x_1,x_2)-\delta\cdot|s|.
%\end{split}
\]
Similarly, using this time that $\gamma'(t'_0)-s\in \gamma$,   
\[
%\begin{split}
	H(x_1,x_2)		%&
	= h\bigl(\gamma(t_0)\bigr)
					\geq h\bigl(\gamma'(t'_0)-s\bigr)%\\&
					\geq h\bigl(\gamma'(t'_0)\bigr)-\delta\cdot|s|
					=H(x'_1,x'_2)-\delta\cdot|s|.
%\end{split}
\]
Observing that $H_0(x'_1,x'_2)=H_0(x_1,x_2)$, it follows that 
\[
	|G(x'_1,x'_2)-G(x_1,x_2)|= |H(x'_1,x'_2)-H(x_1,x_2)|\leq \delta\cdot|s|= \frac{\|(x'_1-x_1,x'_2-x_2)\|}{\sqrt{2}}\cdot\delta.
\]

\smallskip 
\noindent \emph{Estimate for $G|_\cW$.} In order to prove the estimate for $G|_\cW$, where $\cW$ is any line parallel to $w$, consider any $(x_1,x_2)$ and $(x'_1,x'_2)$ in such line $\cW$ and set 
\[
	s:=x'_2-x_2=-(x'_1-x_1), \qquad \text{so} \quad (x'_1,x'_2)= (x_1,x_2)+ s (-1,1).
\]
Observe that for any $z',z\in\HH$ we have
\[
\bigl|
\bigl(h(z')-\Im(z')\bigr)-\bigl(h(z)-\Im(z)\bigr)
\bigr|
\leq
\lipschitz(h-\Im)\cdot|z'-z|
\leq
\delta\cdot|z'-z|
\]
and thus
$
h(z')-\Im(z')
\geq 
h(z)-\Im(z)-\delta\cdot|z'-z|
$, 
which implies 
\begin{equation}\label{hlower}
h(z')
\geq 
h(z)+\Im(z'-z)-\delta\cdot|z'-z|.
\end{equation}
Consider now a  parametrization in polar coordinates of the semicircles described by the geodesics $\gamma$ and $\gamma'$, i.e.\ for any $t\in\RR$ let $\theta(t)\in [0,\pi]$ be the angle such that 
\[
\begin{split}
	\gamma(t)	&=\frac{x_1+x_2}{2}+\frac{x_2-x_1}{2}e^{i\theta(t)},\\
	\gamma'(t)	&=\frac{x'_1+x'_2}{2}+\frac{x'_2-x'_1}{2}e^{i\theta(t)}.
\end{split}
\]
Since, as remarked before the proof, $\gamma'$ has the same center of $\gamma$ but  radius increased by $s$, if we set $\theta_0:=\theta(t_0)$ and $\theta'_0:=\theta(t'_0)$, we have that 
$$\gamma'(t'_0)-s\cdot e^{i\theta'_0}\in \gamma, \quad \text{and} \quad \gamma(t_0)+s\cdot e^{i\theta_0}\in\gamma'.$$
We claim that we must have
\begin{equation}\label{sinineq}
	\sin\theta_0>1-\frac{2\delta}{l}%\frac{4\delta}{x_2-x_1}
	\quad
	\text{ and }
	\quad
	\sin\theta'_0>1-\frac{2\delta}{l}.%\frac{4\delta}{x_2'-x_1'}.
\end{equation}
Indeed, recalling that $\|h-\Im\|_\infty\leq\delta$, we have
\[
\begin{split}
	H(x_1,x_2)	&=\max_{0<\theta<\pi} h\left(\frac{x_1+x_2}{2}+\frac{x_2-x_1}{2}e^{i\theta}\right)\\
				&\geq \max_{0<\theta<\pi} \Im\left(\frac{x_1+x_2}{2}+\frac{x_2-x_1}{2}e^{i\theta}\right)-\delta =\frac{x_2-x_1}{2}-\delta,
\end{split}
\]
but, if the first half of~\eqref{sinineq} fails, since $(x_1,x_2) \in\cU_l$ and hence $x_2-x_1>l$, we have
\[	\sin\theta_0\leq 1-\frac{2\delta}{l}\leq 1-\frac{4\delta}{x_2-x_1}, 
\] 
so that 
\[
\begin{split}
	H(x_1,x_2)	&= h\left(\frac{x_1+x_2}{2}+\frac{x_2-x_1}{2}e^{i\theta_0}\right)
				\leq \Im\left(\frac{x_1+x_2}{2}+\frac{x_2-x_1}{2}e^{i\theta_0}\right)+\delta\\
				&=\frac{x_2-x_1}{2}\sin\theta_0+\delta \leq\frac{x_2-x_1}{2}-\delta,
\end{split}
\]
which is absurd.  The same argument holds for $\theta'_0$ and proves the second part of~\eqref{sinineq}. 

Combining~\eqref{hlower} and~\eqref{sinineq} we get
%Setting $z':=\gamma(x'_1,x'_2,t'_0)$ and $z=\gamma(x_1,x_2,t_0)$ and combining the results above we get
\[
\begin{split}
	H(x'_1,x'_2) 	&=h\left(\gamma'(t'_0)\right)
					\geq h\left(\gamma(t_0)+s\cdot e^{i\theta_0}\right)\\
					&\geq h\left(\gamma(t_0)\right)+\Im(s e^{i\theta_0})-\delta\cdot|s|
					\\
					&\geq h\left(\gamma(x_1,x_2,t_0)\right)+s\cdot\left(1-\frac{2\delta}{l}\right)
					-\delta\cdot|s|\\
					&\geq H(x_1,x_2)+s-\left(1+\frac{2}{l}\right)\delta\cdot|s|.
\end{split}
\]
Similarly, one can also get 
\[
H(x_1,x_2)	\geq H(x'_1,x'_2)-s-\left(1+\frac{2}{l}\right)\delta\cdot|s|.
\]

Therefore, observing that $H_0(x'_1,x'_2)-H_0(x_1,x_2)=s$ it follows that 
\[
\begin{split}
	|G(x'_1,x'_2)-G(x_1,x_2)|&=\bigl|\bigl(H(x'_1,x'_2)-H(x_1,x_2)\bigr)-\bigl(H_0(x'_1,x'_2)-H_0(x_1,x_2)\bigr)\bigr|\\
								&=\bigl|H(x'_1,x'_2)-\bigl(H(x_1,x_2)+s\bigr)\bigr|\\
								&\leq \left(1+\frac{2}{l}\right)\delta\cdot|s|
							=\left(\frac{1}{\sqrt{2}}+\frac{\sqrt{2}}{l}\right)\delta\cdot |(x'_1-x_1,x'_2-x_2)|.
\end{split}
\]
This concludes the proof.
\end{proof}

\subsection*{Acknowledgements}
This collaboration was made possible by the support of the ERC Starting Grant ChaParDyn. 
M. A. is supported by Unicredit Bank.
C. U. is also supported by the Leverhulme Trust through a Leverhulme Prize. The research leading
to these results has received funding from the European Research Council under the European Union Seventh
Framework Program (FP/2007-2013) / ERC Grant Agreement n.~335989. L. M. is grateful to French CNRS and De Giorgi Center for financial and logistic support during his stay in Pisa.

%%%%%%%%%%%%%%%%%%%%%%%%%%%%%%%%%%
%%%								%%%
%%%			BIBLIOGRAPHY		%%%
%%%								%%%
%%%%%%%%%%%%%%%%%%%%%%%%%%%%%%%%%%

%\clearpage
%\phantomsection
%\addcontentsline{toc}{section}{\refname}

\end{document}